\documentclass[11pt,a4paper]{article}
\usepackage{authblk}
\usepackage{bm}
\usepackage[margin=1in]{geometry}
\usepackage{fancyhdr}
\usepackage{booktabs}
\usepackage{amsmath,amsthm,amssymb,amsfonts,mathrsfs,amscd}
\usepackage{enumerate}
\usepackage[T1]{fontenc}
\usepackage{lmodern}
\usepackage{microtype}
\usepackage[nocompress]{cite}
\usepackage[hidelinks]{hyperref}
\hypersetup{
  pdftitle={Optimal weighted sampling of positive and normal operator orbits},
  pdfauthor={Jian Wu},
  pdfkeywords={Dynamical sampling, operator orbit frames, positive operators, condition numbers}
}
\newtheorem{thm}{Theorem}[section]
\newtheorem{prop}[thm]{Proposition}
\newtheorem{lem}[thm]{Lemma}
\newtheorem{cor}[thm]{Corollary}
\theoremstyle{definition}

\newtheorem{example}[thm]{Example}
\newtheorem{remark}[thm]{Remark}
\newenvironment{abs}{\par\noindent\textbf{Abstract.}\enspace\normalfont\ignorespaces}{\par}
\DeclareMathOperator{\dist}{dist}
\begin{document}

\title{Optimal weighted sampling of positive and normal operator orbits}
\author{Jian Wu\thanks{Corresponding author: \texttt{xingxingwu2022@163.com}}}
\affil{}
\renewcommand*{\Affilfont}{\small\itshape}
\renewcommand\Authands{ and }
\date{}
\maketitle

\begin{abs}
We study weighted sampling of operator orbit frames in infinite
dimensional separable Hilbert spaces \(H\), keeping the operator and
generating vectors fixed.
Let $N_\mu(R)$ count the sampling times in $[0,R]$.
For a single orbit of a positive operator, let $m_A(R)$ count the
eigenvalues in $(0,e^{-1/R}]$.
Every real sampling design that preserves the frame property must satisfy
the lower bound $N_\mu(R)\ge m_A(R)-O_\mu(1)$.
Conversely, for every $\eta>0$, we construct an integer sampling design
that preserves the frame property and satisfies
$N_\mu(R)\le(1+\eta)m_A(R)+O_\eta(1)$.
Thus the coefficient $1$ is optimal.
Under additional spectral conditions, we attain
$N_\mu(R)=m_A(R)+O(1)$ and characterize the weights that preserve
the frame property on the constructed sets.

For finitely or countably many orbits of a positive operator,
let $P_j$, $j\ge1$, be the spectral projection corresponding to
$2^{-j-1}<-\log\lambda\le2^{-j}$ and define
$K_A(R)=\#\{j\ge1:2^j\le R,\ P_j\ne0\}$.
Every real sampling design that preserves the frame property satisfies
$N_\mu(R)\ge c_\mu K_A(R)-C_\mu$ for $R\ge2$,
with constants $c_\mu>0$ and $C_\mu\ge0$.
Conversely, we construct an integer sampling design that preserves
the frame property and satisfies $N_\mu(R)\le L K_A(R)+L$
for $R\ge2$, where $L$ depends only on the bounds of the given orbit frame.

For normal operators of norm one, we prove that the smallest frame
condition number is attained among integer sampling designs that
preserve the frame property and satisfy a prescribed counting bound,
whenever this class is nonempty.
Finally, we construct one normal operator with two families of
generating vectors for which the optimal growth orders of
$N_\mu(R)$ are $\log R$ and $R$, respectively, for both real
and integer sampling.


\end{abs}

\section{Introduction}
\label{sec:introduction}
All Hilbert spaces are complex and separable, and inner products are linear
in the first variable. Unless otherwise stated, $H$ is an infinite
dimensional Hilbert space. Let $\mathcal B(H)$ denote the bounded linear
operators on $H$, and write
$\mathbb N=\{1,2,\ldots\}$ and $\mathbb N_0=\{0,1,2,\ldots\}$. Let $I$ be a finite or countably infinite index set.
A family $(g_i)_{i\in I}\subset H$ is a frame for
$H$ if there are constants $0<a\le b<\infty$ such that
$a\|x\|^2\le\sum_{i\in I}|\langle x,g_i\rangle|^2\le b\|x\|^2$
for every $x\in H$. A family satisfying the upper inequality is called a Bessel family.
A Riesz basis is the image of an orthonormal basis under a bounded invertible operator. The frame inequalities ensure stable recovery
from the coefficients $(\langle x,g_i\rangle)_{i\in I}$;
see \cite{Christensen2016}.
If $A\in\mathcal B(H)$ describes evolution in discrete time,
measurements of an initial state $x$ have the form
$\langle A^nx,f_j\rangle=\langle x,(A^*)^nf_j\rangle$,
where $n\in\mathbb N_0$, $J$ is a finite or countably infinite
index set, and $(f_j)_{j\in J}\subset H$ is a family of
generating vectors. Stable recovery is therefore equivalent
to the frame property of $\{(A^*)^nf_j:n\in\mathbb N_0,\ j\in J\}$.
This is the basic connection between dynamical sampling and
operator orbit frames.

Dynamical sampling uses observations at several times to recover
information that may be unavailable at a single time. Early work
showed that, under suitable assumptions, taking measurements at
additional times allows recovery of the initial state from fewer
spatial sampling locations
\cite{AldroubiDavisKrishtal2013,AldroubiDavisKrishtal2015}.
The frame formulation was developed in \cite{ACMT2017}, and spectral
characterizations of frame and Bessel properties for normal operator
orbits were obtained in \cite{AldroubiEtAl2017,Philipp2017}.
Frames generated by finitely many orbits were studied through spectral
decompositions and model spaces in
\cite{CabrelliMolterPaternostroPhilipp2020,CabrelliMolterSuarez2021}.
The question of which frames can be represented as a single orbit
of a bounded operator was studied in the context of dynamical sampling
in
\cite{ChristensenHasannasabRashidi2018,ChristensenHasannasabPhilipp2020}.
In particular, for a single orbit of a diagonal operator, the frame
property is characterized by a Carleson interpolation condition on
the eigenvalues and estimates
for the coordinates of the generating vector \cite{ACMT2017}.
For Carleson's characterization of interpolating sequences for bounded
holomorphic functions in the unit disk, see \cite{Carleson1958}.
For an overview of dynamical sampling, see
\cite{DynamicalSamplingSurvey2026}.

Once an orbit forms a frame, it is natural to ask which sampling times
can be omitted while preserving the frame property.
This amounts to asking whether the initial state can still be
recovered stably from measurements at fewer time instants and,
if so, how slowly the number of retained sampling times can grow
as time increases.
This is a natural refinement of the initial state recovery problem,
one of the central problems in dynamical sampling
\cite[Sections~1--2]{DynamicalSamplingSurvey2026}.
For a frame generated by a single orbit of a positive injective
operator, Christensen et al.\ proved that sampling on
$N\mathbb N_0$ preserves the frame property for every integer
$N\ge1$ \cite[Theorem~2.1]{ChristensenEtAl2024}.
Krishtal and Miller extended this result to sampling at times
$Nk+j_k$, $k\in\mathbb N_0$, where $j_k\in[0,N)$ may vary
with $k$ \cite{KrishtalMiller2026}.
For diagonal operators whose eigenvalues are separated in the
pseudohyperbolic metric and lie in a Stolz domain and an angular
sector $|\arg z|\le c<\pi$,
\cite{KrishtalMashreghiMiller2026} shows that the associated
Carleson frames remain frames when sampled at times
$rk+j_k$, $k\in\mathbb N_0$, where $r>0$, $rc<\pi$,
and $j_k\in[0,r)$.

Density conditions provide another way to describe sampling sets
that preserve the frame property.
Krishtal and Miller introduced block diagonal Carleson frames,
allowing uniformly bounded Jordan blocks
\cite{KrishtalMillerBDC2026}.
For a sampling set $T\subset[0,\infty)$, write
$N_T(R)=\#(T\cap[0,R])$.
For these frames with point spectrum in $(0,1)$,
\cite[Theorem~1.3]{KrishtalMillerBDC2026} characterizes
unweighted sampling whenever the natural density
$L=\lim_{R\to\infty}N_T(R)/R$ exists: the sampled orbit is a
frame if and only if $0<L<\infty$. Wu \cite{WuUniversalNonuniform2026} studied common unweighted
sampling sets that preserve all frames generated by finitely
or countably many orbits of positive operators.
The conditions $0\in T$ and $N_T(R)\asymp R$ are necessary
but not sufficient for this universal preservation:
a set satisfying both conditions may fail to preserve the frame
property even for a single positive diagonal orbit.
For finitely generated dynamical frames associated with invertible
operators under suitable spectral assumptions,
\cite{KrishtalMashreghiMillerSectorial2026} gives sufficient
density conditions for sampling fractional operator orbits.
The number of required observations is also a central issue in
other sampling problems.
Landau's density conditions \cite{Landau1967} give necessary
bounds in classical sampling and interpolation.
Exponential frames with densities arbitrarily close to Landau's
necessary lower bound were constructed in
\cite{BownikVanVelthoven2025}.
In finite dimensions, one can select a small subset of frame vectors
and assign positive weights so that the resulting frame bounds
are close to those of the full frame
\cite{BSS2012,deCarliSilvaHarveySato2015}.
Matrix product approximation by row sampling was studied in
\cite{CohenNelsonWoodruff2016}, and random and deterministic
methods for subsampling finite frames were developed in
\cite{BartelSchaferUllrich2023}.
Weighted frames and the choice of weights were studied in
\cite{BalazsAntoineGrybos2010,BalazsCorsoStoeva2026,KutyniokOkoudjouPhilippTuley2013},
and the choice of nonnegative weights to minimize the frame
condition number was studied in \cite{CasazzaChen2017}.

These results motivate two questions for a given operator orbit
frame: how sparse can the sampling times be while preserving the
frame property, and how small can the condition number of the sampled
frame be under prescribed upper bounds on the number of sampling
times in $[0,R]$ for every $R\ge0$? Sparser sampling reduces the number of observation times needed
for stable recovery of the initial state, while a smaller frame
condition number improves the numerical conditioning of the
reconstruction.
For a fixed frame generated by a single orbit of a positive operator
on an infinite dimensional Hilbert space, unweighted sampling
cannot preserve the frame property if $N_T(R)=o(R)$, as follows
from \cite[Theorem~4.1]{KrishtalMillerBDC2026}. Common weighted integer designs with $O(\log R)$ sampling times
were constructed in \cite{WuUniversalWeighted2026} for positive
systems with a prescribed bound on the condition number of the
given orbit frame.
We allow positive weights and keep both the operator and the
generating vectors fixed.
The sampling times and weights may depend on this fixed frame.
In this setting, our weighted constructions for frames generated
by a single orbit of a positive operator permit integer sampling
with only $O(\log R)$ sampling times in $[0,R]$ as $R\to\infty$;
for suitable spectra, even $o(\log R)$ growth is possible.
Our aim is to determine how slowly the counting function can grow
while preserving the frame property, and how this growth depends
on the spectrum and the generating vectors.
We also study when the resulting counting bounds can be attained
and when the infimum of the frame condition number is attained
under a prescribed pointwise upper bound on the counting function.

To state the problem, suppose that
$\{(A^*)^nf_j:n\in\mathbb N_0,\ j\in J\}$ is a frame
with bounds $0<a\le b<\infty$.
Define the observation operator $C:H\to\ell^2(J)$ by
$Cx=(\langle x,f_j\rangle)_{j\in J}$.
The frame inequalities then take the form
$a\|x\|^2\le\sum_{n\ge0}\|CA^nx\|^2\le b\|x\|^2$
for every $x\in H$.
These inequalities ensure stable recovery of $x$ from
$(CA^nx)_{n\ge0}$ and express exact observability in infinite
time for the discrete time system
\cite[Sections~2.1--2.3]{DiazMedriMolterControl2021}.
In this paper, we call such a pair $(A,C)$ an observable system.

A sampling design is a positive atomic measure
$\mu=\sum_{t\in T}w_t\delta_t$, where $w_t>0$,
$\delta_t$ is the unit point mass at $t$, and
$T\subset[0,\infty)$ is locally finite, meaning that $T\cap[0,R]$
is finite for every $R\ge0$. Its counting function is
$N_\mu(R)=\#(T\cap[0,R])$.
All generating vectors are sampled at each sampling time with
the same weight. The sampled frame operator is
$S_\mu=\sum_{t\in T}w_t(A^t)^*C^*CA^t$.
We call $\mu$ admissible if $S_\mu$ is bounded and has a bounded
inverse, or equivalently if
$\{\sqrt{w_t}(A^t)^*f_j:t\in T,\ j\in J\}$ is a frame.
Integer designs satisfy $T\subset\mathbb N_0$;
real designs may use any nonnegative sampling times.
The definition of real powers for normal operators and the
convergence of these operator sums are specified in
Section~\ref{sec:framework}.
Write $S=\sum_{n\ge0}(A^n)^*C^*CA^n$ for the frame
operator and $\kappa(S_\mu)=\|S_\mu\|\|S_\mu^{-1}\|$ for
the sampled frame condition number. This is the ratio of the
optimal upper frame bound to the optimal lower frame bound.

Our first main result concerns a single orbit of a positive
operator $A$ on $H$. If $\{A^nf:n\in\mathbb N_0\}$ is a frame
for $H$, then the nonzero eigenvalues of $A$ can be written as
$\lambda_j=e^{-s_j}$, $j\ge1$, with $\lambda_j\uparrow1$. Define
$m_A(R)=\#\{j\ge1:s_j\ge R^{-1}\}$ for $R>0$.
Thus $m_A(R)$ counts the eigenvalues in $(0,e^{-1/R}]$.
Theorem~\ref{thm:positive-count} shows that every admissible real
sampling design $\mu$ must satisfy the lower bound
$N_\mu(R)\ge m_A(R)-O_\mu(1)$.
Conversely, for every $\eta>0$, there is an admissible integer
design $\mu_\eta$ with
$N_{\mu_\eta}(R)\le(1+\eta)m_A(R)+O_\eta(1)$.
The construction uses $w_n=n$ at every positive sampling time
$n$, and $w_0=1$ if zero is a sampling time.
Thus the coefficient $1$ in the lower bound is sharp.
For both real and integer sampling, the infimum of
$\limsup_{R\to\infty}N_\mu(R)/\log R$ over admissible designs
equals $\limsup_{R\to\infty}m_A(R)/\log R$; see
Corollary~\ref{cor:positive-coefficient}.
The proof combines spectral estimates with finite dimensional
sparsification. The counting bounds also apply when
$m_A(R)=o(\log R)$.

We also study the stronger equality
$N_\mu(R)=m_A(R)+O(1)$, which we call the critical count.
Under any of the spectral conditions in
Theorems~\ref{thm:critical-regular-spectra} and
\ref{thm:positive-separated-critical}, we construct both real
and integer designs with this count.
The sampled frames have finite excess, meaning that finitely
many deletions leave a Riesz basis.
On the constructed sets, positive weights preserve the frame
exactly when $w_t\asymp t$ outside a finite set
(Corollary~\ref{cor:critical-orbit-weights}); here $u_t\asymp v_t$
means $cv_t\le u_t\le Cv_t$ for constants $c,C>0$.
The arguments use Schoenberg's characterization of totally positive
functions \cite{Schoenberg1951}, results on zeros of Zak transforms
\cite{Kloos2015}, and interpolation in shift-invariant spaces
\cite{deDiosPontGrochenigLiehrShafkulovskaTaylor2026}.
They also use Fredholm theory \cite{BottcherSilbermann2006}
and the theory of frame excess \cite{BalanCasazzaHeilLandau2003}.

For frames generated by finitely or countably many orbits of a
positive operator, the optimal counting order is determined by
the counting function of the nonempty spectral bands.
For $j\ge1$, let $P_j$ be the spectral projection for
$2^{-j-1}<-\log\lambda\le2^{-j}$, and put
$K_A(R)=\#\{j\ge1:2^j\le R,\ P_j\ne0\}$.
Each nonempty band is counted once, regardless of the dimension
of its spectral subspace. Theorem~\ref{thm:band-count} gives
$N_\mu(R)\ge c_\mu K_A(R)-C_\mu$ for every admissible real
design, and constructs an admissible integer design with
$N_\mu(R)\le L K_A(R)+L$, for $R\ge2$.
The constants $c_\mu>0$ and $C_\mu\ge0$ may depend on the
design, whereas $L$ depends only on the bounds of the given orbit frame.
These conclusions also hold for infinite dimensional spectral
subspaces. If the generating vectors span a space of finite
dimension $r$, we further prove
$rN_\mu(R)\ge n_A(R)-O_\mu(1)$, where $n_A(R)$ counts
the eigenvalues in $(0,e^{-1/R}]$ with multiplicity
(Proposition~\ref{prop:finite-channel-count}).
For an observable system $(A,C)$ with $A$ normal,
finitely many sampling times suffice exactly when $\|A\|<1$
(Theorem~\ref{thm:finite-times}).

Our second main question is whether there is a sampling design
with the smallest frame condition number under a prescribed bound
on the counting function.
Let $A$ be normal with $\|A\|=1$. The norm condition is automatic
when finitely many orbits form a frame for $H$.
Fix a set of allowed times $T_0\subset\mathbb N_0$ and a function
$F:[0,\infty)\to[0,\infty)$.
If there is an admissible design supported on $T_0$ with
$N_\mu(R)\le F(R)$ for every $R\ge0$, then one such design
has the smallest condition number
(Theorem~\ref{thm:counting-condition-attainment}).
Both the sampling set $T\subset T_0$ and the positive weights
may vary, subject to the prescribed counting bound.
For a single orbit of a positive operator,
Corollary~\ref{cor:optimal-counting-envelope} gives attainment
results under the bound $N_\mu(R)\le c\,m_A(R)+K$ for every
$R\ge R_0$, where $K\ge0$ and $R_0\ge2$ are fixed.
A finite dimensional example shows that the infimum of the frame
condition number need not be attained when $\|A\|<1$.

For prescribed sampling sets, we obtain more explicit conditioning
results. For a fixed Riesz basis $(u_n)_{n\ge1}$, the infimum of the
frame condition number over all positive weights $(a_n)_{n\ge1}$
for which $(\sqrt{a_n}u_n)_{n\ge1}$ is a frame is attained
(Proposition~\ref{prop:optimal-riesz-weights}).
Fix $h>0$ and an orthonormal basis $(e_j)_{j\ge1}$, and consider
the orbit generated by $f=\sum_{j\ge1}e^{-jh/2}e_j$ under the positive
diagonal operator defined by $Ae_j=e^{-e^{-jh}}e_j$.
For a fixed $\xi\in\mathbb R$, we compute the minimum frame
condition number among admissible designs whose supports differ
from the set $\{e^{\xi+kh}:k\ge1\}$ by finitely
many points, whenever this class is nonempty.
The minimum is attained by a design in this class with
$w_t=t$ for every sampling time $t$
(Theorem~\ref{thm:geometric-optimal-weights}).
Let $\Gamma_{\mathrm{geom}}(h)$ denote the minimum frame condition
number over all admissible designs described above, allowing
$\xi$ to vary. Then $\Gamma_{\mathrm{geom}}(h)\sim\tfrac18 e^{\pi^2/h}$
as $h\downarrow0$
(Corollary~\ref{cor:geometric-phase-conditioning}).
The proofs use Edrei's product representation for generating
functions of doubly infinite totally positive sequences
\cite{Edrei1953} and the Poisson summation formula.

For general normal operators, the angular distribution of the
spectrum also affects sampling. Suppose there exist $M\in\mathbb N$ and $L_0>0$ such that,
in each radial band $2^{-j-1}<-\log|z|\le2^{-j}$, $j\ge1$,
the spectral arguments lie in at most $M$ arcs of length
at most $L_0 2^{-j}$.
For every $0<\varepsilon<1$, Theorem~\ref{thm:normal-angular}
constructs an integer design with
$(1-\varepsilon)S\le S_\mu\le(1+\varepsilon)S$ and
$N_\mu(R)\le L K_{|A|}(R)+L$ for $R\ge2$.
Here $|A|=(A^*A)^{1/2}$ and $K_{|A|}$ counts its nonempty
spectral bands as above.
For a prescribed spectral set satisfying this condition and
prescribed frame bounds, a common design preserves the
frame property for all corresponding generating families
(Corollary~\ref{cor:uniform-spectral-designs}).

For a fixed normal operator, the optimal counting order can
depend on the generating vectors.
Proposition~\ref{prop:ring-observation} gives a normal operator
with two countably infinite generating families. For the first,
there is an admissible integer design with
$N_\mu(R)=\log_2R+O(1)$. For the second, every admissible real
design satisfies $N_\mu(R)\ge c_\mu R$ for some $c_\mu>0$ and
all sufficiently large $R$.
The optimal counting growth is logarithmic for the first family,
meaning $N_\mu(R)\asymp\log R$, and linear for the second,
meaning $N_\mu(R)\asymp R$, as $R\to\infty$.
Both conclusions hold for real and integer sampling.
For the second family, Section~\ref{sec:ring-conditioning} estimates
the infimum of the upper sampling density
$\limsup_{R\to\infty}N_\mu(R)/R$ under a prescribed upper bound
on the frame condition number.

The paper is organized as follows.
Section~\ref{sec:framework} gives the observation formulation
and the basic spectral estimates.
Section~\ref{sec:positive-core} proves the counting bounds and
constructs critical sampling sets for single orbits of positive operators.
Section~\ref{sec:multiple} treats finitely and countably many
orbits, including sampling at finitely many times and unweighted sampling.
Section~\ref{sec:conditioning} studies optimal condition numbers.
Section~\ref{sec:normal} treats normal operators and the role
of the generating vectors.
Section~\ref{sec:questions} discusses common sampling sets
for orbits of positive and normal operators.

\section{Operator orbits and sampling}
\label{sec:framework}

We first establish the Stein identity and a formula for the contribution
of all observations from a given integer time onward
(Lemma~\ref{lem:stein}).
For positive operators, we then derive operator norm estimates
on the subspaces corresponding to the spectral bands.
These estimates are used to construct weighted sampling designs for
single and multiple orbits and to replace real sampling times by
integers while preserving the frame property.

Let $Y$ be a Hilbert space, called the observation space, and let
$A\in\mathcal B(H)$ and $C\in\mathcal B(H,Y)$.
We call $(A,C)$ an observable system if there are constants
$0<a\le b<\infty$ such that
\[
 a\|x\|^2\le \sum_{n\ge0}\|CA^nx\|^2\le b\|x\|^2,
 \qquad x\in H.
\]
Set $B=C^*C$ and define the frame operator by
$S=\sum_{n\ge0}(A^n)^*BA^n$. Sums of positive operators are understood
in the strong operator topology whenever their partial sums are
uniformly bounded in norm.
Define the analysis operator $\mathcal O:H\to\ell^2(\mathbb N_0;Y)$
by $\mathcal O x=(CA^nx)_{n\ge0}$. Then $S=\mathcal O^*\mathcal O$.
If $(e_j)_{j\in J}$ is an orthonormal basis of $Y$, put $f_j=C^*e_j$.
Then $\langle CA^nx,e_j\rangle=\langle x,(A^*)^nf_j\rangle$,
so $(A,C)$ is observable if and only if
$\{(A^*)^nf_j:n\ge0,\ j\in J\}$ is a frame.
For a positive operator, these vectors are $A^nf_j$.
A single orbit corresponds to $Y=\mathbb C$, and a finite family
of orbits to $\dim Y<\infty$. The family has $|J|=\dim Y$
generating vectors. When this number is finite,
$\operatorname{rank}C=\dim\operatorname{span}\{f_j:j\in J\}\le|J|$,
with equality precisely when the generating vectors are linearly
independent.

For a normal operator $A$, let $E_A$ denote its spectral measure.
Real powers of $A$ are defined by the spectral calculus using
$\arg z\in(-\pi,\pi]$. We set $0^t=0$ for $t>0$ and $A^0=I$.
A sampling design is a positive atomic measure
$\mu=\sum_{t\in T}w_t\delta_t$ with locally finite support
$T\subset[0,\infty)$ and weights $0<w_t<\infty$ for every $t\in T$.
The weighted analysis operator is
$\mathcal O_\mu x=(\sqrt{w_t}CA^tx)_{t\in T}$. When bounded,
it maps $H$ to $\ell^2(T;Y)$ and satisfies
$S_\mu=\mathcal O_\mu^*\mathcal O_\mu$.
The corresponding sampled vectors are
$\sqrt{w_t}(A^t)^*f_j$, with $t\in T$ and $j\in J$.
The sampled frame operator and counting function are
\[
 S_\mu=\sum_{t\in T}w_t(A^t)^*BA^t,
 \qquad N_\mu(R)=\#(T\cap[0,R]).
\]
The design is admissible if $\alpha I\le S_\mu\le\beta I$ for some
$0<\alpha\le\beta<\infty$. Since $N_\mu$ depends only on the support $T$,
we also denote it by $N_T$.
An integer sampling design has $T\subset\mathbb N_0$.
For a positive invertible operator $Q$, define its condition number by
$\kappa(Q)=\|Q\|\|Q^{-1}\|$. If $Q$ is a frame operator, its
optimal lower and upper bounds are $\|Q^{-1}\|^{-1}$ and $\|Q\|$.

\begin{lem}\label{lem:stein}
Let $(A,C)$ be an observable system. Then $B=S-A^*SA$ and
$\sum_{n\ge N}(A^n)^*BA^n=(A^N)^*SA^N$ for every $N\in\mathbb N_0$.
Moreover, $A^Nx\to0$ for every $x\in H$. If $A$ is normal, then
$\|A\|\le1$ and its spectral projection on the unit circle is zero.
\end{lem}
\begin{proof}
The definition of $S$ gives the first two identities. Since $S\ge aI$,
$a\|A^Nx\|^2\le\sum_{n\ge N}\|CA^nx\|^2\to0$.
The same inequalities give $\|A^N\|^2\le b/a$, so the spectral
radius is at most one. For a normal operator, the norm equals the
spectral radius. The strong convergence $A^N\to0$ then implies
that the spectral projection on the unit circle is zero.
\end{proof}

Adding finitely many sampling times changes the counting function
by at most a constant. The following lemma shows that a bounded
operator with closed range and finite dimensional kernel can be
made bounded below by adding finitely many integer observations.

\begin{lem}\label{lem:finite-completion}
Let $(A,C)$ be an observable system, let $Z$ be a Hilbert space,
and let $W:H\to Z$ be a bounded operator with closed range
and finite dimensional kernel.
There is a finite set $F\subset\mathbb N_0$ such that
$x\mapsto(Wx,(CA^nx)_{n\in F})$ is bounded below.
\end{lem}
\begin{proof}
Set $E=\ker W$. The analysis operator $\mathcal O$ is bounded below on $E$.
The partial sums of its squared norm converge uniformly on the unit
sphere of $E$, since $E$ is finite dimensional. Thus finitely many
integer observations define an operator $V$ that is bounded below
on $E$. If the conclusion were false,
there would be unit vectors $x_k=u_k+v_k$, $k\ge1$, with $u_k\in E$ and
$v_k\in E^\perp$, such that $Wx_k\to0$ and $Vx_k\to0$.
Closedness of the range of $W$ gives $v_k\to0$.
Passing to a subsequence, we obtain a unit vector $u\in E$ with
$u_k\to u$ and $Vu=0$. This contradicts the lower bound for $V$ on $E$.
\end{proof}

For a positive contraction $A$, let $P_0$ be its spectral projection on
$[0,e^{-1/2})$, and, for $j\ge1$, let $P_j$ correspond to
$2^{-j-1}<-\log\lambda\le2^{-j}$.
The endpoint $e^{-1/2}$ belongs to the band for $P_1$.
Write $q_j=j\log2$ for $j\ge0$ and
$\mathcal J=\{0\}\cup\{j\ge1:P_j\ne0\}$. For $u\in\mathbb R$ and a nonempty set $E\subset\mathbb R$, write
$\operatorname{dist}(u,E)=\inf_{v\in E}|u-v|$.

\begin{lem}\label{lem:log-localization}
Let $(A,C)$ be an observable system with $A$ positive.
For $u\ge0$, set $F(u)=e^{u/2}CA^{e^u}$.
There are absolute constants $L,c>0$ such that the following hold.
\begin{enumerate}[(i)]
\item For every $j\ge0$ and $u\ge0$,
$\|F(u)P_j\|+\|F'(u)P_j\|
\le L\sqrt b\,e^{-c|u-q_j|}$.

\item For every $x\in H$,
$\int_0^\infty\|F'(u)x\|^2\,du\le L^2b\|x\|^2$.

\item For $0<h\le1$ and $D\ge0$, set
$\mathcal K_{h,D}=\{k\in\mathbb N_0:
\operatorname{dist}(kh,\{q_j:j\in\mathcal J\})>D\}$.
Then, for every $x\in H$,
\[
 h\sum_{k\in\mathcal K_{h,D}}\|F(kh)x\|^2
 \le L^2b\,e^{-cD}\|x\|^2.
\]
\end{enumerate}
\end{lem}

\begin{proof}
\emph{(i)}
For $j\ge0$, write $A_j=A|_{P_jH}$, $B_{jj}=P_jB|_{P_jH}$ and
$S_{jj}=P_jS|_{P_jH}$. Compressing the identity
$B=S-A^*SA$ from Lemma~\ref{lem:stein} gives, for $j\ge1$,
$B_{jj}=S_{jj}-A_jS_{jj}A_j$ and hence
$\|B_{jj}\|\le2b\|I-A_j\|\le2b2^{-j}$.
Thus $\|CP_j\|\le\sqrt{2b2^{-j}}$.
For $j\ge1$, set $X_j=-\log A_j$ and $t=e^u$.
The estimate
\[
 \sqrt t\,\|CP_j e^{-tX_j}\|
 \le\sqrt{2b}\,(t2^{-j})^{1/2}e^{-t2^{-j-1}}
\]
and the identity
$F'(u)P_j=\sqrt t\,CP_j(\tfrac12-tX_j)e^{-tX_j}$
give absolute constants $L_0,c_0>0$ such that
\begin{equation}\label{eq:log-localization-band}
 \|F(u)P_j\|+\|F'(u)P_j\|
 \le L_0\sqrt b\,e^{-c_0|u-q_j|}
 \qquad j\ge1,\quad u\ge0.
\end{equation}
On $P_0H$, use $\|CP_0\|\le\sqrt b$,
$\|A^tP_0\|\le e^{-t/2}$ and
$\sup_{s\ge1/2}ts e^{-ts}\le L_2e^{-t/4}$ for $t\ge1$.
Both $F(u)$ and $F'(u)$ vanish on $\ker A$.
Since $q_0=0$, these estimates show that
\eqref{eq:log-localization-band} also holds for $j=0$,
after increasing $L_0$ if necessary.

\emph{(ii)}
Let $\nu$ be either Lebesgue measure on $[0,\infty)$ or
$h\sum_{k\ge0}\delta_{kh}$ with $0<h\le1$.
For $\rho>0$, direct integration or summation of a geometric
series gives
$\int_0^\infty e^{-\rho|u-q_i|/2}\,d\nu(u)\le2+4/\rho$.
Since $|u-q_i|+|u-q_j|\ge|q_i-q_j|$ and
$|q_i-q_j|=|i-j|\log2$, it follows that
\begin{equation}\label{eq:log-localization-kernel}
\begin{aligned}
 &\int_0^\infty
 e^{-\rho(|u-q_i|+|u-q_j|)}\,d\nu(u)
 \le e^{-\rho|q_i-q_j|/2}
 \int_0^\infty e^{-\rho|u-q_i|/2}\,d\nu(u)
 \le K e^{-\eta|i-j|}
\end{aligned}
\end{equation}
for $\rho\in\{c_0,c_0/2\}$, where
$K=2+8/c_0$ and $\eta=c_0\log2/4$.

For $\gamma>0$, set
$M_\gamma=\sum_{\ell\in\mathbb Z}e^{-\gamma|\ell|}<\infty$.
Using $2a_i a_j\le a_i^2+a_j^2$, we obtain
\begin{equation}\label{eq:log-localization-sequence}
\begin{aligned}
 \sum_{i,j\ge0}e^{-\gamma|i-j|}a_i a_j
 \le \frac12\sum_{i,j\ge0}
 e^{-\gamma|i-j|}(a_i^2+a_j^2)
 \le M_\gamma\sum_{j\ge0}a_j^2
\end{aligned}
\end{equation}
for every finitely supported nonnegative sequence $(a_j)_{j\ge0}$.

Write $x_j=P_jx$ and first suppose that only finitely many $x_j$
are nonzero. Applying \eqref{eq:log-localization-band},
then \eqref{eq:log-localization-kernel} with $\rho=c_0$ and
Lebesgue measure, and finally \eqref{eq:log-localization-sequence}
with $\gamma=\eta$ and $a_j=\|x_j\|$, gives
\[
\begin{aligned}
 \int_0^\infty\|F'(u)x\|^2\,du
 &\le L_0^2b\sum_{i,j\ge0}\|x_i\|\,\|x_j\|
       \int_0^\infty
       e^{-c_0(|u-q_i|+|u-q_j|)}\,du\\
 &\le L_0^2bK\sum_{i,j\ge0}
       e^{-\eta|i-j|}\|x_i\|\,\|x_j\|\\
 &\le L_0^2bKM_\eta\|x\|^2.
\end{aligned}
\]

\emph{(iii)}
Fix $0<h\le1$ and $D\ge0$.
If $k\in\mathcal K_{h,D}$ and $i,j\in\mathcal J$, then
$|kh-q_i|>D$ and $|kh-q_j|>D$. Hence
\begin{equation}\label{eq:log-localization-tail-factor}
 e^{-c_0(|kh-q_i|+|kh-q_j|)}
 \le e^{-c_0D}
      e^{-\frac{c_0}{2}(|kh-q_i|+|kh-q_j|)}.
\end{equation}
For $x$ with finitely many nonzero spectral components,
we again write $x_j=P_jx$.
Since $x_j=0$ for $j\notin\mathcal J$,
\eqref{eq:log-localization-band} and
\eqref{eq:log-localization-tail-factor} give
\[
\begin{aligned}
 h\sum_{k\in\mathcal K_{h,D}}\|F(kh)x\|^2
 &\le L_0^2b\sum_{i,j\in\mathcal J}\|x_i\|\,\|x_j\|\,
       h\sum_{k\in\mathcal K_{h,D}}
       e^{-c_0(|kh-q_i|+|kh-q_j|)}\\
 &\le L_0^2b e^{-c_0D}
       \sum_{i,j\in\mathcal J}\|x_i\|\,\|x_j\|\,
       h\sum_{k\ge0}
       e^{-\frac{c_0}{2}(|kh-q_i|+|kh-q_j|)}\\
 &\le L_0^2bK e^{-c_0D}
       \sum_{i,j\in\mathcal J}
       e^{-\eta|i-j|}\|x_i\|\,\|x_j\|\\
 &\le L_0^2bKM_\eta e^{-c_0D}\|x\|^2.
\end{aligned}
\]
The third inequality uses \eqref{eq:log-localization-kernel}
with $\rho=c_0/2$ and
$\nu=h\sum_{k\ge0}\delta_{kh}$.
The last inequality follows from
\eqref{eq:log-localization-sequence} with
$\gamma=\eta$ and $a_j=\|x_j\|$.

By Lemma~\ref{lem:stein}, $E_A(\{1\})=0$.
Thus $x^{(N)}=\sum_{j=0}^N P_jx\to x$ for every $x\in H$.
Applying the estimates obtained in (ii) and (iii) to $x^{(N)}$
and then using Fatou's lemma extends both bounds to every $x\in H$.
Taking $L=L_0\sqrt{KM_\eta}$ and $c=c_0$ proves (i)--(iii).
\end{proof}

\begin{prop}\label{prop:reference-grid}
Let $(A,C)$ be an observable system with $A$ positive.
There exist $0<h\le1$, $D>0$, $M\ge2$ and $N\in\mathbb N_0$,
depending only on $a,b$, such that, writing
$t_k=e^{kh}$ and
$\mathcal I_{h,D}=\mathbb N_0\setminus\mathcal K_{h,D}$,
the following hold.
\begin{enumerate}[(i)]
\item The real sampling design
\(
 \mu=\delta_0+\sum_{k\in\mathcal I_{h,D}}ht_k\delta_{t_k}
\)
is admissible.

\item The integer sampling design
\(
 \mu_{\mathbb Z}
 =\sum_{n=0}^{N}\delta_n+
 \sum_{\substack{k\in\mathcal I_{h,D}\\t_k\ge M}}
 ht_k\delta_{\lfloor t_k\rfloor}
\)
is admissible.
Here $\lfloor t\rfloor$ denotes the greatest integer not exceeding $t$.
\end{enumerate}
\end{prop}

\begin{proof}
\emph{(i)}
Summing the Stein identity from Lemma~\ref{lem:stein} and using
$A^n\to0$ strongly gives
$\int_1^\infty A^tBA^t\,dt=\int_1^2A^tSA^t\,dt$
as quadratic forms.
Set
$Q(x)=\|Cx\|^2+\int_0^\infty\|F(u)x\|^2\,du$.
The substitution $t=e^u$ therefore gives
\[
 Q(x)=\|Cx\|^2+
 \int_1^2\langle SA^tx,A^tx\rangle\,dt
 \le2b\|x\|^2.
\]
For a unit vector $x$, either $\|Cx\|^2\ge a/2$, or
the Stein identity gives $\|Ax\|^2\ge a/(2b)$.
In the latter case, Cauchy--Schwarz yields
$\|A^2x\|^2\ge\|Ax\|^4\ge a^2/(4b^2)$.
Since $A$ is a positive contraction,
$\int_1^2\|A^tx\|^2\,dt\ge\|A^2x\|^2$.
Thus, setting $\gamma=a^{3/2}/(2b)$, we obtain
$\gamma^2\|x\|^2\le Q(x)\le2b\|x\|^2$ for every $x\in H$.

For $k\ge0$, replace $F(u)$ on $[kh,(k+1)h)$ by $F(kh)$.
The fundamental theorem of calculus, Cauchy--Schwarz and
Lemma~\ref{lem:log-localization}(ii) give
\[
 \sum_{k\ge0}\int_{kh}^{(k+1)h}
 \|(F(u)-F(kh))x\|^2\,du
 \le h^2\int_0^\infty\|F'(u)x\|^2\,du
 \le h^2L^2b\|x\|^2.
\]
Choose $0<h\le1$ so that $hL\sqrt b\le\gamma/4$.
The bounds for $Q(x)$ and the discretization error,
together with the triangle inequality, give
\[
 \frac{3\gamma}{4}\|x\|
 \le
 \left(\|Cx\|^2+h\sum_{k\ge0}\|F(kh)x\|^2\right)^{1/2}
 \le\left(\sqrt{2b}+\frac{\gamma}{4}\right)\|x\|.
\]
Next choose $D>0$ so that
$L\sqrt b\,e^{-cD/2}\le\gamma/4$.
Lemma~\ref{lem:log-localization}(iii) then gives
\[
 \left(
 h\sum_{k\in\mathcal K_{h,D}}\|F(kh)x\|^2
 \right)^{1/2}
 \le \frac{\gamma}{4}\|x\|.
\]
Removing these observations and applying the triangle inequality
once more yields
\begin{equation}\label{eq:reference-grid-real-bounds}
 \frac{\gamma}{2}\|x\|
 \le\|\mathcal O_\mu x\|
 \le\left(\sqrt{2b}+\frac{\gamma}{4}\right)\|x\|.
\end{equation}
This proves (i).

\emph{(ii)}
Set $n_k=\lfloor t_k\rfloor$.
We first estimate the error from replacing $t_k$ by $n_k$
when $t_k\ge M\ge2$.
For $j\ge1$, write $X_j=-\log(A|_{P_jH})$.
Since $n_k\ge t_k/2$, integration over $[n_k,t_k]$ gives,
with $y=t_k2^{-j}$,
\[
 \bigl\|\sqrt{ht_k}\,C(A^{t_k}-A^{n_k})P_j\bigr\|
 \le\sqrt{ht_k}
 \int_{n_k}^{t_k}\|CP_jX_je^{-rX_j}\|\,dr
 \le\frac{\sqrt{2hb}}{t_k}\,y^{3/2}e^{-y/4}.
\]
The function $y^{3/2}e^{-y/4}$ is bounded by an exponential
in $-|\log y|$.
For $j=0$, use $\|CP_0\|\le\sqrt b$ and
$\sup_{s\ge1/2}s e^{-rs}\le L_3r^{-1}e^{-r/4}$ for $r\ge1$,
where $L_3>0$ is an absolute constant.
These estimates give absolute constants $L_1,c_1>0$ such that
\[
 \bigl\|\sqrt{ht_k}\,C(A^{t_k}-A^{n_k})P_j\bigr\|
 \le\frac{L_1\sqrt{hb}}{M}
 e^{-c_1|kh-q_j|}
 \qquad j\ge0,\quad t_k\ge M.
\]

Since $q_j=j\log2$, geometric series estimates give absolute
constants $K_1,K_2>0$ such that
\[
 \sup_{k\ge0}\sum_{j\ge0}e^{-c_1|kh-q_j|}\le K_1,
 \qquad
 \sup_{j\ge0}h\sum_{k\ge0}e^{-c_1|kh-q_j|}\le K_2.
\]
Combining these bounds with the spectral band estimate for the
rounding error and applying the Schur test gives an absolute
constant $L_2>0$ such that
\begin{equation}\label{eq:reference-grid-rounding-error}
 \sum_{\substack{k\in\mathcal I_{h,D}\\t_k\ge M}}
 ht_k\|C(A^{t_k}-A^{n_k})x\|^2
 \le\frac{L_2^2b}{M^2}\|x\|^2.
\end{equation}
Choose $M\ge2$ so that
$L_2\sqrt b/M\le\gamma/8$ and $M(e^h-1)\ge1$.
By \eqref{eq:reference-grid-real-bounds} and
\eqref{eq:reference-grid-rounding-error}, rounding these times
leaves a lower norm bound of at least $3\gamma/8$.

It remains to replace the positive sampling times below $M$.
Let $\mathcal E_M=\{k\in\mathbb N_0:t_k<M\}$.
This set is finite.
For each $k\in\mathcal E_M$, choose a polynomial $p_k$ such that
$b\sum_{k\in\mathcal E_M}ht_k
\sup_{0\le\lambda\le1}|\lambda^{t_k}-p_k(\lambda)|^2
\le\gamma^2/64$.
Such polynomials exist by uniform polynomial approximation on
$[0,1]$.
By the spectral theorem and $\|C\|\le\sqrt b$, these polynomial
approximations satisfy
$\sum_{k\in\mathcal I_{h,D}\cap\mathcal E_M}
ht_k\|C(A^{t_k}-p_k(A))x\|^2
\le(\gamma^2/64)\|x\|^2$.
Consequently, replacing the remaining real observations by
$\sqrt{ht_k}\,Cp_k(A)$ leaves a lower norm bound of at least
$\gamma/4$.

Choose $N$ so that all the polynomials have degree at most $N$,
and write $p_k(z)=\sum_{n=0}^{N}d_{kn}z^n$.
Set
$\Lambda=1+\sum_{k\in\mathcal E_M}ht_k
\sum_{n=0}^{N}|d_{kn}|^2$.
Cauchy--Schwarz gives
\[
 \|Cx\|^2+
 \sum_{k\in\mathcal I_{h,D}\cap\mathcal E_M}
 ht_k\|Cp_k(A)x\|^2
 \le\Lambda\sum_{n=0}^{N}\|CA^nx\|^2.
\]
Since $\Lambda\ge1$, the lower bound just obtained implies
\[
 \frac{\gamma^2}{16}\|x\|^2
 \le \|Cx\|^2+
 \sum_{k\in\mathcal I_{h,D}\cap\mathcal E_M}
 ht_k\|Cp_k(A)x\|^2
 +\sum_{\substack{k\in\mathcal I_{h,D}\\t_k\ge M}}
 ht_k\|CA^{n_k}x\|^2
 \le\Lambda\langle S_{\mu_{\mathbb Z}}x,x\rangle.
\]
Thus $S_{\mu_{\mathbb Z}}$ has a positive lower bound.
Its upper bound follows from
$\sum_{n=0}^{N}\|CA^nx\|^2\le b\|x\|^2$,
\eqref{eq:reference-grid-real-bounds} and
\eqref{eq:reference-grid-rounding-error}.
Hence $\mu_{\mathbb Z}$ is admissible.

The choice of the polynomials, and therefore of $N$, depends only
on $a,b$. The condition $M(e^h-1)\ge1$ ensures that the rounded
times with $t_k\ge M$ are distinct. If any of them coincides with
an integer in $\{0,\ldots,N\}$, their weights are added.
\end{proof}
\section{Single orbits of positive operators}
\label{sec:positive-core}

We first consider a single orbit of a positive operator.
We relate the counting function to the distribution of the
eigenvalues. When the spectral counting function grows logarithmically,
we determine the optimal coefficient of $\log R$ in the sampling count.
Under additional spectral assumptions, we construct sampling sets whose
counting functions differ from the spectral counting function by at most
a constant. We then characterize the positive weights that preserve
the frame on these sets.

Write $\mathbb D=\{z\in\mathbb C:|z|<1\}$ and define the
pseudohyperbolic distance by
$\rho(z,w)=|(z-w)/(1-\overline w z)|$ for $z,w\in\mathbb D$.
Assume that $I$ is nonempty.
A sequence of distinct points $(\lambda_k)_{k\in I}\subset\mathbb D$
is called uniformly separated, or a Carleson sequence, if
$\inf_{k\in I}\prod_{\ell\in I\setminus\{k\}}
\rho(\lambda_k,\lambda_\ell)>0$;
see \cite{Carleson1958}.
The empty product is taken to be one.
A real sequence $(x_k)_{k\in I}$ is uniformly separated in the
Euclidean distance if $|x_k-x_\ell|\ge c$ for some $c>0$ and all
$k,\ell\in I$ with $k\ne\ell$.
For an increasing sequence $(x_j)_{j\ge1}$, this is equivalent to
$\inf_{j\ge1}(x_{j+1}-x_j)>0$.

Let $(\lambda_k)_{k\in I}\subset(0,1)$ be a Carleson sequence,
and set $x_k=\log(1/(-\log\lambda_k))$ for $k\in I$.
The product condition gives $0<\delta<1$ such that
$\rho(\lambda_k,\lambda_\ell)\ge\delta$ for all $k,\ell\in I$
with $k\ne\ell$.
For $0<s<u$, we have
$\rho(e^{-s},e^{-u})
=\frac{\sinh((u-s)/2)}{\sinh((u+s)/2)}
\le\frac{u-s}{u+s}$.
It follows that
$|x_k-x_\ell|\ge\log((1+\delta)/(1-\delta))>0$
for all $k,\ell\in I$ with $k\ne\ell$.
Thus $(x_k)_{k\in I}$ is uniformly separated in the Euclidean distance.

Throughout this section, let $(A,C)$ be an observable system with
$A$ positive and $Cx=\langle x,f\rangle$ for some $f\in H$.
In the infinite dimensional case,
\cite[Theorem~5.7]{AldroubiEtAl2017} gives $\dim\ker A\le1$.
We now restrict $A$ and $C$ to $(\ker A)^\perp$ and replace $f$
by its orthogonal projection onto this subspace.
We retain the notation $H,A,C,f$ for the resulting space,
operators, and vector.
There is an orthonormal basis $(e_j)_{j\ge1}$ of $H$ such that
\[
 Ae_j=e^{-s_j}e_j,\qquad j\ge1,\qquad
 s_1>s_2>\cdots>0,\qquad s_j\longrightarrow0
 \quad\text{as }j\to\infty.
\]
The positive eigenvalues are indexed in increasing order,
without loss of generality.
Write the generating vector as $f=\sum_{j\ge1}f_je_j$, with
$f_j=b_j\sqrt{s_j}$, and set $x_j=\log(1/s_j)$ for $j\ge1$.
Since $1-e^{-2s_j}\asymp s_j$, the coefficient bounds and the Carleson
condition in the same theorem give constants $b_-,b_+,\delta>0$ such that
$b_-\le |b_j|\le b_+$ and $x_{j+1}-x_j\ge\delta$ for every $j\ge1$.
By multiplying each vector in the orthonormal basis $(e_j)_{j\ge1}$
by a scalar of modulus one, we may assume $b_j>0$.

Set
\[
 m_A(R)=\#\{j\ge1:s_j\ge R^{-1}\},\qquad
 \delta_A=\limsup_{R\to\infty}\frac{m_A(R)}{\log R}.
\]
The spectral counting function \(m_A(R)\) counts the indices \(j\)
with \(x_j\le\log R\). Since \(x_{j+1}-x_j\ge\delta\),
\(m_A(R)=O(\log R)\). We call
\(\limsup_{R\to\infty}N_\mu(R)/\log R\) the logarithmic counting
coefficient. The critical count condition $N_\mu(R)=m_A(R)+O(1)$
is stronger than attaining the logarithmic counting
coefficient $\delta_A$.

A sampling design that is admissible after omitting the zero eigenspace
and finitely many initial positive eigenvalues can be completed to a
frame on the full space by adding finitely many integer observations.
These finite deletions and additions change the spectral and sampling
counting functions by at most a constant.
The completion argument is given in the proof of
Theorem~\ref{thm:positive-count}.

\subsection{Spectral counting and the optimal logarithmic coefficient}

We first derive a lower bound for the counting function in terms
of $m_A$. We then construct weighted integer sampling designs and
show that $\delta_A$ is the infimum of the logarithmic counting
coefficient for both real and integer sampling.

\begin{lem}\label{lem:positive-lower}
Let $\mu=\sum_{t\in T}w_t\delta_t$ be a sampling design such that
$\{\sqrt{w_t}A^tf:t\in T\}$ is a frame with bounds
$\alpha,\beta>0$.
There exists $K\in\mathbb N_0$ such that
$N_\mu(R)\ge m_A(R)-K$ for every $R>0$.
\end{lem}

\begin{proof}
Set \(m=m_A(R)\). It suffices to consider \(m>K\). The sampled upper frame bound gives
\(\sum_{\substack{t\in T\\t>R}}w_t|f_m|^2e^{-2s_mt}\le\beta\).
For \(1\le j\le m-K\), let \(r=s_j/s_m\). Since \(s_mR\ge1\) and
\(r\ge e^{\delta(m-j)}\),
\[
 \frac{|f_j|^2e^{-2s_jt}}{|f_m|^2e^{-2s_mt}}
 \le\frac{b_+^2}{b_-^2}r e^{-2(r-1)},\qquad t>R.
\]
The function \(r\mapsto r e^{-2(r-1)}\) is decreasing for \(r\ge1\).
Consequently, the part of the sampled frame operator contributed by times
larger than \(R\), compressed to
\(E=\operatorname{span}\{e_1,\ldots,e_{m-K}\}\), has trace at most
\[
 \beta\frac{b_+^2}{b_-^2}
 \sum_{\ell\ge K}e^{\delta\ell}
       \exp\bigl(-2(e^{\delta\ell}-1)\bigr).
\]
Choose $K$ so that this bound is smaller than $\alpha$.
Since the compressed operator is positive, its operator norm
is also smaller than $\alpha$.
If $N_\mu(R)<m-K=\dim E$, the scalar observation functionals
at the times in $T\cap[0,R]$ have a nontrivial common kernel in $E$.
Choose a unit vector $x$ in this kernel.
Its observations at times $t\le R$ vanish, while its sampled
energy at times $t>R$ is smaller than $\alpha$.
This contradicts the lower frame inequality.
\end{proof}

Together with the upper frame bound, the lower bound on the vector
norms in the next lemma will be used to obtain a uniform upper bound
on the coefficients obtained when applying
Lemma~\ref{lem:bss-range}.

\begin{lem}\label{lem:positive-grid}
There exist $a,D,C,c>0$ such that, with
$\mathcal J=\{k\in\mathbb Z:
\operatorname{dist}(ka,\{x_j:j\ge1\})\le D\}$,
the vectors $\psi_k=\sqrt{ae^{ka}}\,A^{e^{ka}}f$,
$k\in\mathcal J$, form a frame for $H$ and satisfy
$\inf_{k\in\mathcal J}\|\psi_k\|>0$.
Its analysis operator $R:H\to\ell^2(\mathcal J)$ has matrix entries
$R_{kj}=\langle e_j,\psi_k\rangle$ satisfying
$|R_{kj}|\le Ce^{-c|ka-x_j|}$
for every $k\in\mathcal J$ and $j\ge1$.
\end{lem}

\begin{proof}
Define $g(u)=\exp(u/2-e^u)$ and, for
$v=\sum_{j\ge1}v_je_j\in H$, put
$Fv(y)=\sum_{j\ge1}b_jv_jg(y-x_j)$.
Both $g$ and $g'$ satisfy a bound of the form $Ce^{-c|u|}$.
Since the points $x_j$ are uniformly separated and
$b_-\le b_j\le b_+$, the Schur test shows that the maps
$v\mapsto Fv$ and $v\mapsto(Fv)'$ are bounded from $H$
to $L^2(\mathbb R)$.

The operators $Q=F^*F$ and $S$ have the following matrix entries
in the basis $(e_j)_{j\ge1}$:
\[
 Q_{ij}=\frac{f_i\overline{f_j}}{s_i+s_j},
 \qquad
 S_{ij}=\frac{f_i\overline{f_j}}{1-e^{-(s_i+s_j)}},
 \qquad i,j\ge1.
\]
The function $u\mapsto(1-e^{-u})^{-1}-u^{-1}$ extends continuously
to $u=0$ and is bounded on $[0,2s_1]$.
Thus $|(S-Q)_{ij}|\le C|f_i||f_j|$ for all $i,j\ge1$. Since $\sum_{j\ge1}|f_j|^2<\infty$,
the operator $S-Q$ is Hilbert--Schmidt.
As $S$ is invertible, $Q$ is Fredholm.
Here Fredholm means that the operator has closed range and
finite dimensional kernel and cokernel.

Since $Q$ is self-adjoint, it remains to show that its kernel
is trivial. Suppose that $Qv=0$. Then $Fv=0$ in $L^2(\mathbb R)$.
The identity
$Fv(\log t)=\sqrt t\sum_{j\ge1}v_jf_je^{-s_jt}$
and continuity give
$\sum_{j\ge1}v_jf_je^{-s_jt}=0$ for every $t>0$.
By Cauchy--Schwarz,
$\sum_{j\ge1}|v_jf_j|\le\|v\|\|f\|<\infty$.
Thus $\nu=\sum_{j\ge1}v_jf_j\delta_{s_j}$ is a finite complex
measure on $[0,s_1]$, whose Laplace transform vanishes for $t>0$.
Since $\nu$ has compact support, differentiation under the
integral at zero gives
$\int_{[0,s_1]}s^n\,d\nu(s)=0$ for every $n\in\mathbb N_0$.
Polynomials are dense in $C([0,s_1])$, so $\nu=0$.
As the points $s_j$ are distinct, this implies
$v_jf_j=0$ for every $j\ge1$.
Since $f_j\ne0$ for every $j\ge1$, we obtain $v=0$.
Thus $Q$ is positive and boundedly invertible, and $F$ is bounded below.

For $k\in\mathbb Z$ and $y\in[ka,(k+1)a)$, replace $Fv(y)$
by $Fv(ka)$. The fundamental theorem of calculus and
Cauchy--Schwarz give
\[
 \sum_{k\in\mathbb Z}\int_{ka}^{(k+1)a}
 |Fv(ka)-Fv(y)|^2\,dy
 \le a^2\|(Fv)'\|_2^2.
\]
Since $F$ is bounded below and $v\mapsto(Fv)'$ is bounded,
the discrete analysis operator
$v\mapsto(\sqrt a\,Fv(ka))_{k\in\mathbb Z}$
is bounded and bounded below when $a>0$ is sufficiently small.
Its matrix entries are $\sqrt a\,b_jg(ka-x_j)$.

Fix $L\ge0$. Let $E=(E_{kj})$ agree with the matrix of the
discrete analysis operator on any chosen collection of entries
satisfying $|ka-x_j|>L$, and be zero elsewhere.
Since $|g(u)|\le C_0e^{-c_g|u|}$ for every $u\in\mathbb R$,
each chosen entry satisfies
\[
 |E_{kj}|
 =\sqrt a\,b_j|g(ka-x_j)|
 \le \sqrt a\,b_+C_0e^{-c_g|ka-x_j|}
 \le Ce^{-cL}e^{-c|ka-x_j|},
\]
where $c=c_g/2$ and $C=\sqrt a\,b_+C_0$.
The same bound holds for the zero entries.
Since both $a\mathbb Z$ and $\{x_j:j\ge1\}$ are uniformly separated,
we have
$\sup_{k\in\mathbb Z}\sum_{j\ge1}e^{-c|ka-x_j|}<\infty$
and
$\sup_{j\ge1}\sum_{k\in\mathbb Z}e^{-c|ka-x_j|}<\infty$.
After increasing $C$ if necessary, the Schur test gives
\[
 \|E\|
 \le
 \left(
 \sup_{k\in\mathbb Z}\sum_{j\ge1}|E_{kj}|
 \right)^{1/2}
 \left(
 \sup_{j\ge1}\sum_{k\in\mathbb Z}|E_{kj}|
 \right)^{1/2}
 \le Ce^{-cL}.
\]
Thus deleting any collection of entries with $|ka-x_j|>L$
changes the discrete analysis operator by at most $Ce^{-cL}$
in operator norm.

Deleting all rows at distance greater than $D$ from
$\{x_j:j\ge1\}$ therefore changes the operator by an arbitrarily
small amount in norm when $D$ is sufficiently large.
The remaining operator $R:H\to\ell^2(\mathcal J)$ is consequently
bounded and bounded below.
Since
$R_{kj}=\sqrt a\,b_jg(ka-x_j)=\langle e_j,\psi_k\rangle$,
the vectors $(\psi_k)_{k\in\mathcal J}$ form a frame for $H$,
and
$|R_{kj}|=\sqrt a\,b_j|g(ka-x_j)|
\le \sqrt a\,b_+C_0e^{-c_g|ka-x_j|}
\le Ce^{-c|ka-x_j|}$.

For every $k\in\mathcal J$, there is a point $x_j$ with
$|ka-x_j|\le D$. Hence
$\|\psi_k\|\ge |R_{kj}|
\ge\sqrt a\,b_-\min_{|u|\le D}g(u)=:c_0>0$.
Finally, $k\in\mathcal J$ implies $ka\ge x_1-D$,
so the retained sampling times are bounded away from zero
and form a locally finite set.
\end{proof}

\begin{lem}\label{lem:bss-range}
Let $v_1,\ldots,v_q$ be real vectors spanning a space of dimension $r$,
and put $L=\sum_{i=1}^q v_i v_i^*$.
Here $v_i v_i^*$ denotes the operator
$x\mapsto\langle x,v_i\rangle v_i$.
For every $\tau>1$, there are nonnegative numbers
$\theta_i$, $1\le i\le q$, at most $\lceil\tau r\rceil$
of which are nonzero, such that
\[
 L\le\sum_{i=1}^q\theta_i v_i v_i^*\le\kappa_\tau L,
 \qquad
 \kappa_\tau=
 \left(\frac{\sqrt\tau+1}{\sqrt\tau-1}\right)^2.
\]
\end{lem}

\begin{proof}
If $r=0$, take $\theta_i=0$ for every $i$.
Suppose that $r\ge1$, and set
$E=\operatorname{span}_{\mathbb R}\{v_1,\ldots,v_q\}$
and $L_E=L|_E$.
For every nonzero $x\in E$,
$\langle L_Ex,x\rangle
=\sum_{i=1}^q|\langle x,v_i\rangle|^2>0$,
because the vectors $v_i$ span $E$.
Thus $L_E$ is positive and invertible.

Put $u_i=L_E^{-1/2}v_i$ for $1\le i\le q$.
Then
\(
 \sum_{i=1}^q u_i u_i^*
 =L_E^{-1/2}L_E L_E^{-1/2}
 =I_E,
\)
where $I_E$ is the identity operator on $E$.
Applying the Batson--Spielman--Srivastava theorem
\cite[Theorem~3.1]{BSS2012} on the $r$ dimensional space $E$
gives nonnegative coefficients $\theta_i$, at most
$\lceil\tau r\rceil$ of which are nonzero, such that
\(
 I_E\le\sum_{i=1}^q\theta_i u_i u_i^*
 \le\kappa_\tau I_E.
\)
Since $L_E^{1/2}u_i=v_i$, conjugating these inequalities
by $L_E^{1/2}$ gives
\[
 L_E
 \le
 L_E^{1/2}\left(\sum_{i=1}^q\theta_i u_i u_i^*\right)L_E^{1/2}
 =
 \left.\left(\sum_{i=1}^q\theta_i v_i v_i^*\right)\right|_E
 \le\kappa_\tau L_E.
\]
Both $L$ and $\sum_{i=1}^q\theta_i v_i v_i^*$ vanish on $E^\perp$,
so the stated inequalities hold on the full real space.

Finally, for any real symmetric matrix $M$ and
$z=x+\mathrm{i}y$ with real vectors $x,y$, we have
$\langle Mz,z\rangle
=\langle Mx,x\rangle+\langle My,y\rangle$.
Thus the inequalities also hold for complex vectors
with the same coefficients.
\end{proof}

\begin{lem}
\label{lem:positive-rounding}
Suppose that a sampled orbit of a positive operator is a frame with
sampling set \(T\subset\{e^{ka}:k\in\mathbb Z\}\) and weights
\(w_t\le Cat\) for some \(a,C>0\).
Keeping the same index set \(T\) after rounding, define
\(R,\widetilde R:H\to\ell^2(T)\) by
\(Rx=(\sqrt{w_t}\langle x,A^tf\rangle)_{t\in T}\) and
\(\widetilde Rx=(\sqrt{w_t}\langle x,A^{\lceil t\rceil}f\rangle)_{t\in T}\).
Then \(\widetilde R-R\) is Hilbert--Schmidt. The rounded family can be completed
to a frame by adding finitely many observations at nonnegative integer times.
\end{lem}

\begin{proof}
For \(t\in T\), write \(n=\lceil t\rceil\). For every \(j\ge1\), we have
\(|e^{-s_jn}-e^{-s_jt}|\le s_j e^{-s_jt}\).
The squared Hilbert--Schmidt norm of \(\widetilde R-R\) is at most
\[
 C\sum_{j\ge1} |f_j|^2s_j^2
     \sum_{k\in\mathbb Z} a e^{ka}e^{-2s_je^{ka}}
 \le C_a\sum_{j\ge1} |f_j|^2s_j<\infty.
\]
For every $y\in\mathbb R$ and $\ell\in\mathbb N_0$,
at most two points of $a\mathbb Z$ have distance from $y$
in $[\ell a,(\ell+1)a)$.
Since $|g(u)|\le C_0e^{-c_g|u|}$, we obtain
\[
 \sup_{y\in\mathbb R}\sum_{k\in\mathbb Z}g(ka-y)^2
 \le 2C_0^2\sum_{\ell\ge0}e^{-2c_g a\ell}
 =\frac{2C_0^2}{1-e^{-2c_g a}}<\infty.
\]
Together with the identity
$\sum_{k\in\mathbb Z}ae^{ka}e^{-2s_je^{ka}}
=(a/s_j)\sum_{k\in\mathbb Z}g(ka-x_j)^2$,
this proves the second inequality. Moreover,
$\sum_{j\ge1}|f_j|^2s_j
\le s_1\sum_{j\ge1}|f_j|^2
=s_1\|f\|^2<\infty$.

A compact perturbation of an operator that is bounded below has closed range and
finite dimensional kernel. Lemma~\ref{lem:finite-completion} therefore
gives a frame after adding finitely many observations at nonnegative integer times.
If several times round to the same integer, we add their weights.
\end{proof}

\begin{thm}
\label{thm:positive-count}
Let \(\mu=\sum_{t\in T}w_t\delta_t\) be a sampling design.
If \(\{\sqrt{w_t}A^tf:t\in T\}\) is a frame, then
\(N_\mu(R)\ge m_A(R)-O_\mu(1)\).
Conversely, for every \(\eta>0\), there is such a frame at integer times
with \(N_\mu(R)\le(1+\eta)m_A(R)+O_\eta(1)\), with weights
\(w_n=n\) at positive sampling times and \(w_0=1\) if \(0\in T\).
\end{thm}

\begin{proof}
The lower bound follows from Lemma~\ref{lem:positive-lower}.
We now prove the upper bound.
Let $R$ be the analysis operator of the frame in
Lemma~\ref{lem:positive-grid}.
Choose $\gamma>0$ such that
$\|Rx\|\ge\gamma\|x\|$ for every $x\in H$.
Let $c_0>0$ be the constant obtained in the proof of
Lemma~\ref{lem:positive-grid}.
For every $k\in\mathcal J$, there is an index $j(k)\ge1$ such that
$|ka-x_{j(k)}|\le D$ and $|R_{k,j(k)}|\ge c_0$.
In particular, $\|\psi_k\|\ge c_0$.
Fix $1<\tau<1+\eta$. We will first choose a truncation distance
\(L\) and then the number \(M\) of spectral points in each block.

Partition the logarithmic time axis into intervals, each containing
exactly $M$ successive points of the sequence $(x_j)_{j\ge1}$.
More precisely, put
\[
 d_\ell=\tfrac12(x_{\ell M}+x_{\ell M+1}),\qquad \ell\ge1,\qquad
 I_1=(-\infty,d_1),\qquad
 I_\ell=[d_{\ell-1},d_\ell),\qquad \ell\ge2.
\]
For each \(\ell\ge1\), let
\(J_\ell=\{j\ge1:\operatorname{dist}(x_j,I_\ell)\le L\}\).
Separation gives \(|J_\ell|\le M+2\lceil L/\delta\rceil+2\).
For each row of \(R\) whose logarithmic time lies in \(I_\ell\), retain
only the coordinates in \(J_\ell\). Denote the resulting matrix by
\(\widetilde R\).
An omitted entry has \(|ka-x_j|>L\), and the proof of
Lemma~\ref{lem:positive-grid} gives \(\|R-\widetilde R\|\le\varepsilon_L\),
where \(\varepsilon_L\to0\) as \(L\to\infty\), uniformly in \(M\).
If \(L\ge D\), every retained row still has norm at
least \(c_0\), because its closest spectral point belongs to \(J_\ell\).
Each block has finitely many rows, since its retained grid points lie
within distance \(D\) of finitely many spectral points.

Since $b_j>0$, the matrix $R$ and its restricted rows have real entries.
We therefore apply Lemma~\ref{lem:bss-range} to the restricted rows
in each block.
This does not require a positive lower eigenvalue bound for the block.
Let $(\theta_k)_{k\in\mathcal J}$ denote the resulting coefficients.
Form the matrix $\widetilde U$ by multiplying row $k$ of
$\widetilde R$ by $\sqrt{\theta_k}$ and deleting the rows
for which $\theta_k=0$.

For each $x\in H$, sum the block inequalities over the first $N$
blocks and let $N\to\infty$.
Monotone convergence of the nonnegative sums gives
\(
 \|\widetilde R x\|^2
 \le \sum_{k\in\mathcal J}
       \theta_k|(\widetilde R x)_k|^2
 \le \kappa_\tau\|\widetilde R x\|^2.
\)
The upper bound shows that $\widetilde U$ defines a bounded
analysis operator, and these inequalities give
$\widetilde R^*\widetilde R
\le\widetilde U^*\widetilde U
\le\kappa_\tau\widetilde R^*\widetilde R$.

Assume $L\ge D$ and $\varepsilon_L\le1$.
For every $k\in\mathcal J$ with $\theta_k>0$, we have
\[
 \theta_k c_0^2
 \le \theta_k\sum_{j\ge1}|\widetilde R_{kj}|^2
 \le \|\widetilde U\|^2
 \le \kappa_\tau\|\widetilde R\|^2
 \le \kappa_\tau(\|R\|+1)^2.
\]
Consequently,
$0<\theta_k\le Q_\tau:=\kappa_\tau(\|R\|+1)^2/c_0^2$.

Let $\mathcal K=\{k\in\mathcal J:\theta_k>0\}$.
Restore the omitted coordinates in the rows of $\widetilde U$,
retaining the factors $\sqrt{\theta_k}$, to obtain
$Uv=(\sqrt{\theta_k}\langle v,\psi_k\rangle)_{k\in\mathcal K}$.
The coefficient bound gives
$\|U-\widetilde U\|\le\sqrt{Q_\tau}\,\varepsilon_L$.
Hence
$\|Uv\|\ge[\gamma-(1+\sqrt{Q_\tau})\varepsilon_L]\|v\|$.

Since $\varepsilon_L\to0$ and $Q_\tau$ is independent of $L$ and $M$,
choose $L\ge D$ so that $\varepsilon_L\le1$ and
$d:=\gamma-(1+\sqrt{Q_\tau})\varepsilon_L>0$.
Since $\theta_k\le Q_\tau$ and $R$ is bounded,
\[
 \frac{d^2}{Q_\tau}\|v\|^2
 \le \sum_{k\in\mathcal K}|\langle v,\psi_k\rangle|^2
 \le \|R\|^2\|v\|^2.
\]
Thus $(\psi_k)_{k\in\mathcal K}$ is a frame.
Since $a^{-1/2}\psi_k=\sqrt{e^{ka}}\,A^{e^{ka}}f$,
multiplying each vector by $a^{-1/2}$ gives a frame on
the same sampling set with weights
$w_{e^{ka}}=e^{ka}$, $k\in\mathcal K$.

Since $x_{j+1}-x_j\ge\delta$, each restricted block has at most
$M+2\lceil L/\delta\rceil+2$ nonzero columns.
Lemma~\ref{lem:bss-range} therefore selects at most
$K_M:=\tau(M+2\lceil L/\delta\rceil+2)+1$
sampling times in each block.
Choose $M$ so large that $K_M\le(1+\eta)M$.
If $\log R\in I_\ell$, all $(\ell-1)M$ spectral points in
earlier intervals lie below $\log R$. Hence
\(
 N_\mu(R)\le \ell K_M
 \le(1+\eta)m_A(R)+K_M.
\)
This proves the estimate for real sampling times.
By Lemma~\ref{lem:positive-rounding}, rounding the times and
adding finitely many observations at nonnegative integer times
gives an integer sampling frame.
Since $\lceil t\rceil\le R$ implies $t\le R$, rounding does not
increase the counting function.
Finite completion adds only finitely many sampling times.
Their contribution is absorbed into $O_\eta(1)$,
so the resulting integer design satisfies the stated upper bound.

We now justify restoring any spectral subspaces omitted earlier.
Fix $m\in\mathbb N$ and set
$E_m=\operatorname{span}\{e_1,\ldots,e_m\}$ and $H_m=E_m^\perp$.
Suppose that $\nu=\sum_{t\in T}w_t\delta_t$ is admissible for the
system restricted to $H_m$, with upper frame bound $\beta$.
For $1\le i\le m$, we have $s_i>s_{m+1}$, so
\[
 \sum_{t\in T}w_t|f_i|^2e^{-2s_it}
 \le\frac{|f_i|^2}{|f_{m+1}|^2}
       \sum_{t\in T}w_t|f_{m+1}|^2e^{-2s_{m+1}t}
 \le\beta\frac{|f_i|^2}{|f_{m+1}|^2}<\infty.
\]
The preceding estimate gives boundedness of the weighted analysis
operator on the finite dimensional space $E_m$.
Its restriction to $H_m$ is bounded by admissibility.
Thus $W=\mathcal O_\nu:H\to\ell^2(T)$ is bounded.
Its restriction to $H_m$ is bounded below, so $W(H_m)$ is closed.
As $W(E_m)$ is finite dimensional,
$\operatorname{ran}W=W(H_m)+W(E_m)$ is closed.
The orthogonal projection onto $E_m$ is injective on $\ker W$,
because $W|_{H_m}$ is injective. Hence $\dim\ker W\le m$.
Lemma~\ref{lem:finite-completion} therefore gives a frame on $H$
after adding finitely many observations at nonnegative integer times.

The zero eigenspace removed earlier can be restored in the same way,
since it is finite dimensional and contributes only at time zero.
Restoring finitely many eigenvalues and adding finitely many sampling
times change the corresponding counting functions by at most a constant.
Thus both the real and integer constructions extend to the space
before restriction with the same upper counting bounds up to
an additive constant.

Let $v_n>0$ denote the weight of the completed integer design
at a sampling time $n$.
Since consecutive grid times are more than one apart for all
sufficiently large indices, only finitely many times coincide after
rounding. Outside a finite set, each sampling time \(n\) therefore has
weight \(v_n=t\), where \(n=\lceil t\rceil\), so
\(n-1<v_n\le n\). The ratios \(n/v_n\) at positive sampling times
are consequently bounded above and below by positive constants.
Comparison of the weighted frame inequalities shows that replacing
\(v_n\) by \(n\), and taking \(w_0=1\) if zero belongs to the sampling set,
preserves the frame property. The sampling set and its counting
function are unchanged.
\end{proof}

\begin{cor}\label{cor:positive-coefficient}
For both real and integer sampling, the infimum of
$\limsup_{R\to\infty}N_\mu(R)/\log R$
over all admissible designs $\mu$ equals $\delta_A$.
If $\delta_A=0$, this infimum is attained.
\end{cor}

\begin{proof}
The lower bound follows from Lemma~\ref{lem:positive-lower}. For the
upper bound, apply Theorem~\ref{thm:positive-count} and let
\(\eta\downarrow0\). If \(\delta_A=0\), each sampling design constructed
in that theorem has logarithmic counting coefficient zero.
\end{proof}

\subsection{Critical sampling sets}
\label{sec:critical}

When \(\delta_A>0\), Corollary~\ref{cor:positive-coefficient} gives the
value of the infimum but does not show that it is attained. We now give
three spectral conditions under which a sampling design attains the
critical count \(N_\mu(R)=m_A(R)+O(1)\). Each construction gives a frame
of finite excess, meaning that finitely many vectors can be removed
to obtain a Riesz basis; see Section~\ref{sec:finite-excess}.

For the geometric model in
\cite[Proposition~3.5]{WuUniversalWeighted2026},
real and integer designs attaining the critical count were
constructed in \cite[Theorem~B.2]{WuUniversalWeighted2026}.
We extend this conclusion to the spectral conditions below.

The proof uses total positivity. A function \(g:\mathbb R\to\mathbb R\)
is totally positive if \(\det(g(u_i-v_j))_{i,j=1}^m\ge0\) for every
\(m\ge1\) and all strictly increasing real sequences
\(u_1<\cdots<u_m\) and \(v_1<\cdots<v_m\).

\begin{lem}\label{lem:logarithmic-kernel}
Let $g(u)=\exp(u/2-e^u)$ and $D=\operatorname{diag}((b_j)_{j\ge1})$.
For a finite or countable family $(t_k)_{k\in I}$ of positive times, let $V$
be the matrix whose $k$th column is the coordinate vector of
$\sqrt{t_k}A^{t_k}f$ in the basis $(e_j)_{j\ge1}$.
Then $V=DG$, where $G_{jk}=g(\log t_k-x_j)$ for $j\ge1$ and $k\in I$.
The operator $D$ is boundedly invertible. The function $g$ is continuous,
integrable, and totally positive, and
$|g(u)|+|g'(u)|\le C\exp(-c|u|)$ for some $C,c>0$.
\end{lem}

\begin{proof}
The bounds on $b_j$ imply that $D$ is boundedly invertible. The identity
$\sqrt{t_k}f_j e^{-s_jt_k}=b_jg(\log t_k-x_j)$ proves the matrix formula.

For $\operatorname{Re}z<1/2$, the substitution $v=e^u$ gives
\[
\int_{\mathbb R}e^{-zu}g(u)\,du=\Gamma(1/2-z).
\]
The normalized reciprocal of this transform has the product
\[
\frac{\Gamma(1/2)}{\Gamma(1/2-z)}
=e^{\psi(1/2)z}
\prod_{n=0}^{\infty}
\left(1-\frac{z}{n+1/2}\right)e^{z/(n+1/2)}.
\]
Here $\Gamma$ is the gamma function and $\psi=\Gamma'/\Gamma$.
The zeros of this product have summable squared reciprocals. Thus
Schoenberg's characterization \cite{Schoenberg1951} implies that $g$ is
totally positive. The product has no Gaussian
factor $e^{-\gamma z^2}$ with $\gamma>0$. The decay follows from the formula
for $g$ and the identity $g'(u)=(1/2-e^u)g(u)$.
\end{proof}

We identify \(\ell^2(\mathbb N)\) with the subspace of \(\ell^2(\mathbb Z)\)
whose coordinates vanish for \(j\le0\), and let \(P\) be the orthogonal
projection onto this subspace. For a bounded operator \(L\) on
\(\ell^2(\mathbb Z)\), its compression to \(\ell^2(\mathbb N)\) is
\(L_+=PL|_{\ell^2(\mathbb N)}\), acting on \(\ell^2(\mathbb N)\).
Its matrix is \((L_{jk})_{j,k\ge1}\).

\begin{lem}\label{lem:critical-compression}
The following statements hold.
\begin{enumerate}[(i)]
\item If a bounded invertible matrix $L$ on $\ell^2(\mathbb Z)$
satisfies $|L_{jk}|\le C e^{-c|j-k|}$ for some $C,c>0$
and all $j,k\in\mathbb Z$, then its compression $L_+$ is Fredholm.

\item Suppose that a matrix $E$ on $\ell^2(\mathbb N)$ satisfies
$|E_{jk}|\le C e^{-c|j-k|}$ for some $C,c>0$ and all $j,k\ge1$.
If $E_{j,j+d}\to0$ as $j\to\infty$ for every fixed integer $d$,
then $E$ is compact.
\end{enumerate}
\end{lem}

\begin{proof}
For (i), both $PL(I-P)$ and $(I-P)LP$ are Hilbert--Schmidt, since the sum of
$e^{-2c|j-k|}$ over either pair of opposite half axes is finite.  On $\operatorname{ran}P$, multiplying $PLP$ by $PL^{-1}P$ in either
order gives the identity up to a compact error. These errors are
\(
-PL(I-P)L^{-1}P, -PL^{-1}(I-P)LP.
\)
Thus $L_+$ is Fredholm.

For (ii), let $E^{(q)}$ be obtained from $E$ by setting its entries
with $|j-k|>q$ to zero. The Schur test gives
$\|E-E^{(q)}\|\le C\sum_{|d|>q}e^{-c|d|}\to0$ as $q\to\infty$.  Within a
fixed band, each diagonal is a shift composed with a diagonal operator
whose entries tend to zero.  Each such operator is compact.  Taking
$q\to\infty$ proves the assertion.
\end{proof}

\begin{lem}\label{lem:finite-orbit-completion}
If a Bessel sequence $(\sqrt{v_k}A^{t_k}f)_{k\ge1}$, with $v_k>0$, has a
Fredholm synthesis operator, it can be completed to a frame of finite
excess by adding finitely many integer orbit vectors at new times.
\end{lem}

\begin{proof}
Let $V:\ell^2(\mathbb N)\to H$ be the synthesis operator.
Since $V$ is Fredholm, $V^*$ has closed range and finite
dimensional kernel. Apply Lemma~\ref{lem:finite-completion}
to $W=V^*$ to obtain a frame by finite completion.

For $x\in\ker V^*$, we have
$(V^*x)_k=\sqrt{v_k}\langle x,A^{t_k}f\rangle=0$
for every $k\ge1$. Since $v_k>0$, the observations at existing
times vanish on $\ker V^*$.
In the finite completion argument, the added observations may
therefore be chosen at distinct times
$n_1,\ldots,n_m\in\mathbb N_0\setminus\{t_k:k\ge1\}$.

The completed frame has synthesis operator
$\widetilde V:\ell^2(\mathbb N)\oplus\mathbb C^m\to H$ given by
\(
 \widetilde V(c,d)
 =Vc+\sum_{\ell=1}^m d_\ell A^{n_\ell}f.
\)
Define $\pi:\ker\widetilde V\to\mathbb C^m$ by $\pi(c,d)=d$.
Then $\ker\pi=\ker V\oplus\{0\}$, and hence
\[
 \dim\ker\widetilde V
 =\dim\ker V+\dim\operatorname{ran}\pi
 \le\dim\ker V+m<\infty.
\]
Thus the completed frame has finite excess.
\end{proof}

The proof of the next theorem uses the Zak transform
$Z\phi(x,\omega)=\sum_{k\in\mathbb Z}\phi(x-k)e^{2\pi i k\omega}$.
For the functions used below, the exponential decay of $\phi$
ensures absolute and locally uniform convergence of this series
on $\mathbb R^2$.

\begin{thm}\label{thm:critical-regular-spectra}
Let $(A,C)$ be an observable system with $A$ positive and
$Cx=\langle x,f\rangle$. Write the nonzero eigenvalues as
$e^{-s_j}$, $j\ge1$, with $e^{-s_j}\uparrow1$, and set $x_j=\log(1/s_j)$.
Suppose that one of the following conditions holds for some $h>0$:
\begin{enumerate}[(i)]
\item $x_{j+1}-x_j\to h$ as $j\to\infty$;
\item $x_j=jh+\eta_j$ for $j\ge1$, and for some $a<b$ with $b-a<h$ one has
$a\leq\eta_j\leq b$ for all sufficiently large $j$.
\end{enumerate}
For both real and integer sampling times, there is a weighted orbit
frame of finite excess satisfying \(N_T(R)=m_A(R)+O(1)\) and
\(\lim_{R\to\infty}N_T(R)/\log R=1/h\).
Its weights equal the sampling times outside a finite set.
\end{thm}

\begin{proof}
We first choose real sampling times and prove that the corresponding
synthesis operator is Fredholm. We then round the times and add finitely
many observations at nonnegative integer times to obtain the required frames.

Under (i), apply the theorem on the unique zero of the Zak
transform to $u\mapsto g(hu)$ \cite{Kloos2015}. By
Lemma~\ref{lem:logarithmic-kernel}, this function is continuous,
integrable, and totally positive, and the reciprocal of its Laplace
transform has no Gaussian factor. There is a shift $z\in\mathbb R$ for which
$p_z(\theta)=\sum_{n\in\mathbb Z}g(z+nh)e^{in\theta}$ has no zero. Modulo $h$, only the shift corresponding to the unique Zak
zero is excluded. The series converges absolutely and uniformly, so
$p_z$ is continuous and bounded away from zero. Under the Fourier
transform, the matrix $L=(g(z+(k-j)h))_{j,k\in\mathbb Z}$ on $\ell^2(\mathbb Z)$
acts by multiplication by $p_z$. Hence $L$ is bounded and invertible.

Choose $t_k=e^{z+x_k}$ for $k\ge1$. There are constants $0<d_-\leq d_+<\infty$ such
that $d_-\leq x_{j+1}-x_j\leq d_+$ for all $j\ge1$. The entries of
$G=(g(z+x_k-x_j))_{j,k\ge1}$ and of the compression $L_+$ of $L$ therefore
decay exponentially with the distance from the diagonal, with uniform
constants. For each fixed $r\in\mathbb Z$, we have
$x_{j+r}-x_j\to rh$ and $G_{j,j+r}-(L_+)_{j,j+r}\to0$ as $j\to\infty$.
Lemma~\ref{lem:critical-compression} shows that $G-L_+$ is compact and
that $L_+$ is Fredholm.  Thus $G$, and then $V=DG$, are Fredholm.

Under (ii), use the reflected kernel $\phi(u)=g(-hu)$, which
is also continuous, integrable, and totally positive.  Write
$(\xi_0,1/2)$ for the unique zero of $Z\phi$ in $[0,1)^2$.
Set $\delta_0=\xi_0-1/2$, $m=(a+b)/2$, and $z=m-h\delta_0$.
Replace the finitely many exceptional values of $\eta_j$ by $m$, and
extend this modified sequence to $j\leq0$ by the same value.  Denote
the resulting sequence on $\mathbb Z$ by $\widetilde\eta_j$, and set
$\delta_j=(\widetilde\eta_j-z)/h$ for $j\in\mathbb Z$.
Then $\sup_{j\in\mathbb Z}|\delta_j-\delta_0|\leq(b-a)/(2h)<1/2$.
The sharp perturbation theorem for totally positive generators
\cite[Theorem~1.2]{deDiosPontGrochenigLiehrShafkulovskaTaylor2026}
implies that the matrix $Q_{jk}=\phi(j+\delta_j-k)$, indexed by
$j,k\in\mathbb Z$, is invertible on $\ell^2(\mathbb Z)$. The perturbations
$\delta_j$ are bounded, so $Q$ has exponential decay away from the
diagonal.  Its compression to positive indices is Fredholm by
Lemma~\ref{lem:critical-compression}.

Choose $t_k=e^{z+kh}$ for $k\ge1$. Except in finitely many rows, the matrix in
Lemma~\ref{lem:logarithmic-kernel} satisfies
$G_{jk}=g(z+kh-x_j)=\phi(j+\delta_j-k)=Q_{jk}$ for $j,k\ge1$.
Since $g$ decays exponentially, each row of $G$ belongs to
$\ell^2(\mathbb N)$.
Thus $G$ differs from the compression of $Q$ by a finite rank operator,
and $V=DG$ is Fredholm.

For either construction, put $n_k=\lceil t_k\rceil$ for $k\ge1$.
The gaps between consecutive $t_k$ tend to infinity, so only finitely
many of the $n_k$ are repeated. Also
$\rho_k:=\log n_k-\log t_k\to0$ as $k\to\infty$.
The synthesis matrix for weights $n_k$ is $V^{\mathrm{int}}=DG^{\mathrm{int}}$,
where $G^{\mathrm{int}}_{jk}=g(\log n_k-x_j)=g(\log t_k+\rho_k-x_j)$ for $j,k\ge1$.
In both cases, there are constants $c_1,C_1>0$ such that
$|\log t_k-x_j|\ge c_1|j-k|-C_1$ for all $j,k\ge1$.
Since $(\rho_k)_{k\ge1}$ is bounded, the mean value theorem
and the exponential decay of $g'$ give
$|G^{\mathrm{int}}_{jk}-G_{jk}|
\le C\rho_k e^{-c|j-k|}$
for some $C,c>0$ and all $j,k\ge1$.
Thus the entries satisfy a uniform exponential bound in
the distance from the diagonal and tend to zero along
each fixed diagonal.
Lemma~\ref{lem:critical-compression} shows that $G^{\mathrm{int}}-G$ is
compact. Since $D$ is bounded, $V^{\mathrm{int}}-V=D(G^{\mathrm{int}}-G)$
is also compact. Thus $V^{\mathrm{int}}$ is Fredholm; deleting finitely
many columns to remove repeated times preserves this property.
Lemma~\ref{lem:finite-orbit-completion} now completes
both the real and the integer families to frames of finite excess.

Under (i), the lower bound on the gaps of $(x_j)_{j\ge1}$ gives
$\#\{k\ge1:x_k\leq\log R-z\}=m_A(R)+O(1)$.
Also $x_j/j\to h$, which gives $m_A(R)/\log R\to1/h$.
Under (ii), boundedness of $\eta_j$ gives both
$m_A(R)=h^{-1}\log R+O(1)$ and the same estimate for the regular
logarithmic sampling set. Since $n_k=\lceil t_k\rceil$, counting by indices gives
\[
 \#\{k:n_k\le R\}
 =\#\{k:t_k\le R\}
 -\#\{k:\lfloor R\rfloor<t_k\le R\}.
\]
The last term is uniformly bounded because consecutive sampling
times eventually differ by more than one.
Since only finitely many repetitions occur, passing to distinct
rounded times and adding finitely many observations changes
the count by at most a constant.
This proves both counting estimates. The weights equal the sampling
times except at the finitely many added observations.
\end{proof}

\begin{thm}
\label{thm:positive-separated-critical}
Suppose that \(s_j/s_{j+1}\ge9\) for all \(j\ge1\). Set
\(t_j=(2s_j)^{-1}\) for \(j\ge1\) and \(\mu=\sum_{j\ge1}t_j\delta_{t_j}\).
Then \(\{\sqrt{t_j}A^{t_j}f\}_{j\ge1}\) is a Riesz basis with
\(N_\mu(R)=m_A(R)+O(1)\).
There is also an admissible integer sampling design $\nu$
such that $N_\nu(R)=m_A(R)+O(1)$ and the corresponding
sampled frame has finite excess.
\end{thm}

\begin{proof}
After multiplication of the columns by the invertible diagonal matrix
\(\operatorname{diag}((b_j^{-1})_{j\ge1})\), and multiplication by \(\sqrt{2e}\),
the analysis matrix has entries indexed by \(j,k\ge1\):
\[
 H_{jk}=h(s_k/s_j),\qquad h(r)=\sqrt r\,e^{(1-r)/2},\qquad H_{jj}=1.
\]
For \(k>j\), we have \(|H_{jk}|\le\sqrt e\,3^{-(k-j)}\).
For \(k<j\), monotonicity of \(h\) on \([1,\infty)\) gives
\(|H_{jk}|\le h(9^{j-k})\).
Both the row sums and the column sums of \(|H-I|\) are therefore bounded
by
\[
 \rho=\frac{\sqrt e}{2}+\sum_{\ell\ge1}h(9^\ell)<1.
\]
Indeed,
\(h(9^{\ell+1})/h(9^\ell)=3e^{-4\cdot9^\ell}\le3e^{-36}\).
Thus the series is at most \(3e^{-4}/(1-3e^{-36})\), and
\(\sqrt e/2+3e^{-4}/(1-3e^{-36})<1\). The Schur test gives
\(\|H-I\|\le\rho<1\), so the analysis operator is invertible.

The counting function is \(m_A(2R)\). The separation assumption gives
\(m_A(2R)-m_A(R)=O(1)\), proving the assertion.

To round these times, apply the Hilbert--Schmidt estimate from
Lemma~\ref{lem:positive-rounding} to \(\log t_j=x_j-\log2\). These
logarithmic times are separated, so the sums of translates of \(g^2\)
are uniformly bounded as required. If \(n_j=\lceil t_j\rceil\),
replacing the reference weight \(t_j\) by \(n_j\) is another compact
perturbation: the diagonal multipliers \(\sqrt{n_j/t_j}-1\) tend to zero.
The rounded analysis operator is a compact perturbation of an invertible
operator, hence Fredholm. Lemma~\ref{lem:finite-orbit-completion}
produces a frame with finite excess.
Since the logarithmic times are separated, only finitely many times can
round to the same integer. Combining their weights preserves finite
excess. Let $\nu$ denote the resulting integer sampling design. Since \(\lceil t_j\rceil\le R\) implies \(t_j\le R\), rounding
does not increase the count. Finite completion changes the counting function by at most a constant.
Together with Lemma~\ref{lem:positive-lower}, this gives
$N_\nu(R)=m_A(R)+O(1)$.
\end{proof}

The following construction uses the shift excluded in case~(i) of the
proof of Theorem~\ref{thm:critical-regular-spectra}.

\begin{prop}
Let $s_j=e^{-jh}$ for $j\ge1$ and some $h>0$, and suppose that the integer orbit
of the corresponding positive operator is a frame.
There is $z\in\mathbb R$ such that $T=\{e^{z+kh}:k\ge1\}$ satisfies
$N_T(R)=m_A(R)+O(1)$, but no positive weights give a sampling frame,
even after adding finitely many sampling times.
\end{prop}
\begin{proof}
Choose the shift $z$ corresponding to the unique zero of the Zak
transform of $g(hu)$. Then the continuous symbol
$p(\theta)=\sum_{n\in\mathbb Z}g(z+nh)e^{in\theta}$
satisfies $p(\pi)=0$.
Let $L_+$ be the operator on $\ell^2(\mathbb N)$ with matrix entries
$g(z+(k-j)h)$, $j,k\ge1$.

For $N\ge1$, define $y_N\in\ell^2(\mathbb N)$ by
\[
 (y_N)_j=
 \begin{cases}
 N^{-1/2}(-1)^j, & N<j\le2N,\\
 0, & \text{otherwise}.
 \end{cases}
\]
Then $\|y_N\|=1$. For every $x\in\ell^2(\mathbb N)$,
$|\langle y_N,x\rangle|
\le(\sum_{j>N}|x_j|^2)^{1/2}\to0$,
so $y_N\rightharpoonup0$.

Write $a_n=g(z+nh)$ for $n\in\mathbb Z$.
Exponential decay gives
$M:=\sum_{n\in\mathbb Z}|a_n|<\infty$ and
$\varepsilon_q:=\sum_{|n|>q}|a_n|\to0$ as $q\to\infty$.
Let $L_{+,q}$ be obtained from $L_+$ by setting its entries
with $|j-k|>q$ to zero.
The Schur test gives
$\|L_+^*-L_{+,q}^*\|\le\varepsilon_q$.
Moreover, since $p(\pi)=0$,
$p_q(\pi):=\sum_{|n|\le q}(-1)^n a_n$
satisfies $|p_q(\pi)|\le\varepsilon_q$.

For $N>2q$ and $N+q<j\le2N-q$, we have
\[
 (L_{+,q}^*y_N)_j
 =\frac{(-1)^j}{\sqrt N}
   \sum_{|n|\le q}(-1)^n a_n
 =p_q(\pi)(y_N)_j.
\]
The vector $L_{+,q}^*y_N$ is supported on
$\{N+1-q,\ldots,2N+q\}$.
Outside the interior interval above, there are at most $4q$
coordinates, each bounded in modulus by $M/\sqrt N$.
Consequently,
$\|L_{+,q}^*y_N\|^2
\le |p_q(\pi)|^2+4qM^2/N$.
It follows that
$\|L_+^*y_N\|
\le \varepsilon_q+\|L_{+,q}^*y_N\|
\le 2\varepsilon_q+2M\sqrt{q/N}$.
For each fixed $q$, letting $N\to\infty$ gives
$\limsup_{N\to\infty}\|L_+^*y_N\|\le2\varepsilon_q$.
Letting $q\to\infty$, we obtain $L_+^*y_N\to0$.

Set $t_k=e^{z+kh}$ and $\psi_k=\sqrt{t_k}A^{t_k}f$ for $k\ge1$.
By Lemma~\ref{lem:logarithmic-kernel}, the synthesis operator $V$
of $(\psi_k)_{k\ge1}$ has matrix $DL_+$, where $D$ is boundedly
invertible.
Set $c_N=\|(D^*)^{-1}y_N\|_{\ell^2}$, and let $x_N\in H$
have coordinate sequence $c_N^{-1}(D^*)^{-1}y_N$ in the
basis $(e_j)_{j\ge1}$.
Since $\|D\|^{-1}\le c_N\le\|D^{-1}\|$ and
$y_N\rightharpoonup0$, we have $\|x_N\|=1$ and
$x_N\rightharpoonup0$.
Moreover,
$\|V^*x_N\|^2=c_N^{-2}\|L_+^*y_N\|^2\to0$.

The $k$th coordinate of $\psi_k$ equals $b_kg(z)$, so
$\|\psi_k\|\ge b_-g(z)>0$ for every $k\ge1$.
Suppose that positive weights $(w_k)_{k\ge1}$ give a Bessel
sequence $(\sqrt{w_k}A^{t_k}f)_{k\ge1}$ with bound $\beta$.
Then
$(w_k/t_k)b_-^2g(z)^2
\le\|\sqrt{w_k}A^{t_k}f\|^2\le\beta$,
and hence
$w_k/t_k\le C_\beta:=\beta/(b_-^2g(z)^2)$.
Consequently,
\[
 \sum_{k\ge1}w_k|\langle x_N,A^{t_k}f\rangle|^2
 \le C_\beta\sum_{k\ge1}t_k|\langle x_N,A^{t_k}f\rangle|^2
 =C_\beta\|V^*x_N\|^2\longrightarrow0.
\]
For any fixed finite family of additional observation vectors
$u_1,\ldots,u_m\in H$, weak convergence gives
$\sum_{\ell=1}^m|\langle x_N,u_\ell\rangle|^2\to0$.
Thus the combined family cannot satisfy a positive lower frame bound.

Finally, for all sufficiently large $R$,
$N_T(R)=\lfloor(\log R-z)/h\rfloor$ and
$m_A(R)=\lfloor\log R/h\rfloor$.
Hence $N_T(R)=m_A(R)+O(1)$.
\end{proof}

\subsection{Admissible weights and finite excess}
\label{sec:finite-excess}

We characterize the positive weights that preserve the frame
property for frames of finite excess.

Let $(u_n)_{n\ge1}$ be a frame in $H$.
Its analysis operator is
$U:H\to\ell^2(\mathbb N)$, $Ux=(\langle x,u_n\rangle)_{n\ge1}$.
The excess of the frame is $\dim\ker U^*$, or equivalently the
codimension of $\operatorname{ran}U$.
Standard results on excess and Riesz bases can be found in
\cite{BalanCasazzaHeilLandau2003,Christensen2016}.
All scalar weights considered below are finite and strictly positive.

\begin{lem}[{\cite[p.~94 and Lemma~4.1(c)]{BalanCasazzaHeilLandau2003}}]
\label{lem:finite-deletion-weights}
A frame of finite excess $r$ contains a Riesz basis obtained
by deleting exactly $r$ vectors.
\end{lem}

The following proposition follows from the preceding lemma
and standard properties of Riesz bases.

\begin{prop}\label{prop:finite-excess-weights}
Let $(u_n)_{n\geq1}$ be a frame of finite excess and let
$a_n>0$ for every $n\geq1$. Then $(\sqrt{a_n}u_n)_{n\geq1}$
is a frame if and only if
$0<\inf_{n\geq1}a_n\leq\sup_{n\geq1}a_n<\infty$.
\end{prop}

\begin{proof}
Sufficiency follows directly from the frame inequalities.
For necessity, suppose that $(\sqrt{a_n}u_n)_{n\ge1}$ is a frame.
By Lemma~\ref{lem:finite-deletion-weights}, there is a finite set
$F\subset\mathbb N$ such that $(u_n)_{n\in I}$ is a Riesz basis,
where $I=\mathbb N\setminus F$.

Put $v_n=\sqrt{a_n}u_n$ for $n\in I$.
Multiplication by nonzero scalars preserves the closed span of
every subfamily. Thus $(v_n)_{n\in I}$ is complete and minimal,
where minimality means that
$v_n\notin\overline{\operatorname{span}}\{v_m:m\in I\setminus\{n\}\}$
for every $n\in I$.
A complete family obtained by deleting finitely many vectors
from a frame remains a frame
\cite[p.~95]{BalanCasazzaHeilLandau2003}.
Hence $(v_n)_{n\in I}$ is a minimal frame and therefore a Riesz basis.

The vector norms in both Riesz bases are bounded above and
bounded away from zero. Thus there are constants
$c_0,C_0,c_1,C_1>0$ such that $c_0\le\|u_n\|^2\le C_0$ and
$c_1\le\|v_n\|^2\le C_1$ for every $n\in I$.
Consequently,
$c_1/C_0\le a_n=\|v_n\|^2/\|u_n\|^2\le C_1/c_0$
for every $n\in I$.
Since $F$ is finite and each $a_n$ is finite and strictly positive,
these bounds extend to all indices after adjusting the constants.
Therefore $0<\inf_{n\ge1}a_n\le\sup_{n\ge1}a_n<\infty$.
\end{proof}

To apply these facts to sampling, let
$T=\{t_n:n\ge1\}\subset[0,\infty)$ be locally finite and put
$q_n=A^{t_n}f$ for $n\ge1$.
Suppose that positive reference weights $(v_n)_{n\ge1}$ are
chosen so that $(u_n)_{n\ge1}$, where $u_n=\sqrt{v_n}q_n$,
is a frame of finite excess.

\begin{cor}\label{cor:critical-orbit-weights}
Suppose that $(\sqrt{v_n}A^{t_n}f)_{n\ge1}$ is a frame of finite excess.
Then $(\sqrt{w_n}A^{t_n}f)_{n\ge1}$ is a frame if and only if
$0<\inf_{n\ge1}(w_n/v_n)\le\sup_{n\ge1}(w_n/v_n)<\infty$.
\end{cor}

\begin{proof}
Apply Proposition~\ref{prop:finite-excess-weights} with $a_n=w_n/v_n$.
\end{proof}

For the critical sampling frames constructed above, $v_n=t_n$
except at finitely many times. Thus the new weights preserve
the frame exactly when $w_n\asymp t_n$ outside that finite set.
Lemma~\ref{lem:finite-deletion-weights} also gives a Riesz basis
after finitely many deletions.
These deletions change the counting function by at most a constant,
so both the critical count condition and the logarithmic counting
coefficient are preserved.

\begin{example}
Suppose that $T$ carries a weighted orbit frame with bounds $a,b>0$ and
$N_T(R)=O(\log R)$, where $N_T(R)=\#(T\cap[0,R])$.
Assume also that $A^nf\ne0$ for every nonnegative integer $n$.
Choose a strictly increasing sequence $(r_j)_{j\geq1}$ of integers
outside $T$ such that
$r_j\geq\exp(j^2)$. This is possible because $T$ has at most logarithmic
growth.  Given $\eta>0$, put
$\varepsilon_j=\eta 2^{-j}/(1+\|A^{r_j}f\|^2)$.
The positive operator contributed by these additional observations has
norm at most $\eta$, since
$\sum_{j\geq1}\varepsilon_j\|A^{r_j}f\|^2\leq\eta$.
The enlarged family is therefore a frame with bounds $a,b+\eta$.
Its additional counting function is at most $\sqrt{\log R}$ for $R\geq e$,
and consequently
\[
\limsup_{R\to\infty}\frac{N_{T\cup\{r_j:j\geq1\}}(R)}{\log R}
=\limsup_{R\to\infty}\frac{N_T(R)}{\log R}.
\]
However, $\varepsilon_j/r_j\to0$ as $j\to\infty$.
If the design on $T$ attains the optimal logarithmic counting
coefficient, then the enlarged design also attains it, although
its weights are not bounded below by a positive multiple of
the sampling times.
Deleting all added vectors leaves the frame on $T$, so the
enlarged frame has infinite excess.
This explains why the construction lies outside the finite
excess setting of Corollary~\ref{cor:critical-orbit-weights}.
If a bound $\Gamma>b/a$ on the condition number after sampling is
prescribed, choosing $\eta<\Gamma a-b$ ensures that the enlarged frame
also satisfies this bound.
\end{example}
\section{Sampling of multiple operator orbits}
\label{sec:multiple}

We now turn to finite and countable families of operator orbits.
We first characterize when finitely many sampling times suffice.
For positive operators, we then show that the number of nonempty
spectral bands determines the optimal order of the sampling count. We also study common sampling designs for systems with a fixed spectral set and unweighted sampling with a compact observation operator.

\begin{thm}\label{thm:finite-times}
Let $(A,C)$ be an observable system with $A$ normal.
An admissible sampling design with finite support exists
if and only if $\|A\|<1$.
In this case, the sampling times may be chosen to be
nonnegative integers.
\end{thm}
\begin{proof}
Let $a,b$ be frame bounds for $(A,C)$.
Suppose that an admissible sampling design has finite support $T$
and lower frame bound $\alpha>0$.
For every $x\in H$ and $N\in\mathbb N_0$, the frame bounds give
\[
 \alpha\sum_{n=0}^N\|A^nx\|^2
 \le \sum_{t\in T}w_t\sum_{n=0}^N\|CA^nA^tx\|^2
 \le b\left(\sum_{t\in T}w_t\|A^t\|^2\right)\|x\|^2.
\]
Thus the operators $\sum_{n=0}^N(A^n)^*A^n$ are uniformly bounded
in norm.
By Lemma~\ref{lem:stein}, $\|A\|\le1$.
If $\|A\|=1$, normality and the spectral theorem give
$\|\sum_{n=0}^N(A^n)^*A^n\|=N+1$,
a contradiction. Hence $\|A\|<1$.

Conversely, suppose that $\|A\|<1$.
Lemma~\ref{lem:stein} gives
$(a-b\|A\|^{2N+2})I\le\sum_{n=0}^N(A^n)^*BA^n\le bI$.
Choose $N$ large enough that $b\|A\|^{2N+2}<a$.
Then the sampling times $0,1,\ldots,N$, with all weights equal
to one, form an admissible sampling design.
\end{proof}

\begin{remark}
Suppose that $H$ is infinite dimensional.
If $C$ is compact and a sampling design has finite support, then
$S_\mu$ is a finite sum of compact operators and hence cannot
be bounded below.
Thus an admissible design with finite support requires $C$
to be noncompact.
In particular, finitely many operator orbits cannot form a frame
at finitely many sampling times.
If $(A,C)$ is observable with $A$ normal and
$\operatorname{rank}C<\infty$, then $\|A\|=1$ by
Theorem~\ref{thm:finite-times} and Lemma~\ref{lem:stein}.

Infinite rank alone is insufficient.
For example, on $\ell^2(\mathbb N)$, let
$Ae_j=\lambda_je_j$ with $0<\lambda_j\uparrow1$, and set
$C=(I-A^2)^{1/2}$.
Then $S=I$, but $C$ is compact and has infinite rank,
so no sampling design with finite support is admissible.
For $A=\tfrac12I$ and $C=I$, time zero alone suffices.
\end{remark}

If $\dim H=d<\infty$, a finite sampling set $T$ is admissible
if and only if the map $x\mapsto(CA^tx)_{t\in T}$ is injective.
Positive weights do not affect this condition.
Consequently, if $\operatorname{rank}C=r$, admissibility
requires $r\,\#T\ge d$.

For a single orbit, suppose that $A$ is positive with $d$
distinct strictly positive eigenvalues and that $(A^nf)_{n\ge0}$
is a frame for $H$.
Then any $d$ distinct nonnegative sampling times give a basis.
Indeed, after factoring out the nonzero coordinates of $f$
in an eigenbasis, the sampling matrix reduces to a real
exponential matrix.
A nonzero real linear combination of $d$ exponential functions
with distinct real exponents has at most $d-1$ distinct real zeros,
as follows by induction using Rolle's theorem.
The matrix is therefore invertible over both the real and
complex fields.
If zero is allowed among the eigenvalues, time zero must be included.

In finite dimensions, the frame operator can also be preserved
exactly with a number of sampling times bounded in terms of
the dimension.

\begin{prop}\label{prop:finite-dimensional-exact}
Let $(A,C)$ be an observable system with $\dim H=d<\infty$.
There exists an integer sampling design $\mu$ with at most
$d^2$ sampling times such that $S_\mu=S$.
\end{prop}

\begin{proof}
Let $\mathcal V$ be the real vector space of Hermitian
$d\times d$ matrices.
Set $Q_n=(A^n)^*BA^n$ and define
$\nu=\sum_{n\in\mathbb N_0:\,Q_n\ne0}
\operatorname{tr}Q_n\,\delta_{Q_n/\operatorname{tr}Q_n}$.
Then $\nu(\mathcal V)=\operatorname{tr}S<\infty$.
Since $\|Q_n\|\le\operatorname{tr}Q_n$, we have
$\int_{\mathcal V}\|M\|\,d\nu(M)\le\operatorname{tr}S$
and $\int_{\mathcal V}M\,d\nu(M)=S$.
As $\dim_{\mathbb R}\mathcal V=d^2$,
\cite[Corollary~1]{BayerTeichmann2006} gives
$S=\sum_{i=1}^m a_iQ_{n_i}/\operatorname{tr}Q_{n_i}$
with $1\le m\le d^2$,
where the indices $n_1,\ldots,n_m$ are distinct,
$Q_{n_i}\ne0$, and $a_i>0$.
Set $w_{n_i}=a_i/\operatorname{tr}Q_{n_i}$ and
$\mu=\sum_{i=1}^m w_{n_i}\delta_{n_i}$.
Then $S_\mu=\sum_{i=1}^m w_{n_i}Q_{n_i}=S$.
\end{proof}

\subsection{Spectral bands and sampling counts}
Use the spectral projections $P_j$ from
Lemma~\ref{lem:log-localization} and define
$K_A(R)=\#\{j\ge1:2^j\le R,\ P_j\ne0\}$ for $R\ge2$.
Here a band is called nonempty when its spectral projection $P_j$
is nonzero. Each such band contributes one to $K_A(R)$ if $2^j\le R$. Changing finitely many initial bands alters $K_A$ by a bounded amount. A different choice of band endpoints changes only the constants in the estimates below. Neither the spectral projections $P_j$ nor the observation operator $C$ is required to have finite rank.

\begin{lem}\label{lem:scalar-spectral-tests}
Let $(A,C)$ be an observable system with frame bounds $a,b$,
where $A$ is positive and $\|A\|=1$.
If $\alpha I\le S_\mu\le\beta I$ for some $0<\alpha\le\beta<\infty$,
then $\alpha/b\le(1-\lambda^2)\sum_{t\in T}w_t\lambda^{2t}\le\beta/a$
for every $\lambda\in\sigma(A)\cap(0,1)$.
\end{lem}
\begin{proof}
Fix $\lambda\in\sigma(A)\cap(0,1)$ and choose $r\in(\lambda,1)$.
The spectral theorem gives unit vectors $x_k\in E_A([0,r])H$, $k\ge1$,
such that $\|(A-\lambda I)x_k\|\to0$ as $k\to\infty$.
After passing to a subsequence, assume that
$\langle Sx_k,x_k\rangle\to\sigma\in[a,b]$ as $k\to\infty$.
The Stein identity and continuous functional calculus give
$\|CA^t x_k\|^2\to\lambda^{2t}(1-\lambda^2)\sigma$
as $k\to\infty$ for every $t\ge0$.
For each finite $F\subset T$, the sampled upper frame bound therefore gives
$(1-\lambda^2)\sigma\sum_{t\in F}w_t\lambda^{2t}\le\beta$.
Taking the supremum over all finite $F\subset T$ and using $\sigma\ge a$
proves $(1-\lambda^2)\sum_{t\in T}w_t\lambda^{2t}\le\beta/a$.

Since $\|A\|=1$ and $E_A(\{1\})=0$ by Lemma~\ref{lem:stein},
$1$ is an accumulation point of $\sigma(A)$.
Choose $\gamma\in\sigma(A)\cap(r,1)$.
The upper bound just proved, applied at $\gamma$, gives
$\sum_{t\in T}w_t\gamma^{2t}<\infty$.
For every $t\in T$ and $k\ge1$, we have
$\|CA^t x_k\|^2\le\|C\|^2r^{2t}\le\|C\|^2\gamma^{2t}$.
Dominated convergence therefore yields
$\lim_{k\to\infty}\langle S_\mu x_k,x_k\rangle
=(1-\lambda^2)\sigma\sum_{t\in T}w_t\lambda^{2t}$.
The lower frame bound and $\sigma\le b$ give the required lower estimate.
\end{proof}

\begin{thm}\label{thm:band-count}
Let $(A,C)$ be an observable system with frame bounds $a,b$,
where $A$ is positive.
For every admissible sampling design $\mu$, there exist
$c_\mu>0$ and $C_\mu\ge0$ such that
$N_\mu(R)\ge c_\mu K_A(R)-C_\mu$ for all $R\ge2$.
There is an admissible integer sampling design $\mu$ with
$N_\mu(R)\le L K_A(R)+L$ for all $R\ge2$,
where $L>0$ depends only on $a,b$.
\end{thm}

\begin{proof}
Let $\alpha,\beta$ be the sampled frame bounds for an admissible design $\mu$.
If $\|A\|<1$, then $K_A$ is bounded, so the lower bound holds
with $c_\mu=1$ and $C_\mu=\sup_{R\ge2}K_A(R)$. We may therefore assume
$\|A\|=1$ when proving that bound.
Write $s=-\log\lambda$. Lemma~\ref{lem:scalar-spectral-tests} implies
$l\le s\sum_{t\in T}w_te^{-2st}\le M$
whenever $0<s\le\tfrac12$ and $e^{-s}\in\sigma(A)$,
for some constants $0<l\le M<\infty$.
Fix $R\ge2$ with $K_A(R)>0$; the lower bound is trivial when $K_A(R)=0$.
Choose one spectral point from each band counted by $K_A(R)$.
Among the band indices, one residue class modulo three contains at
least $K_A(R)/3$ of these points. Denote their decay rates by
$s_1>\cdots>s_p$. They satisfy $s_i/s_{i+1}>4$ for
$1\le i<p$ and $Rs_p>1/2$.

For a finite sequence of positive numbers with successive ratio at
least four, there is a numerical constant $L_0$ such that
$\sum_{i=1}^p s_i e^{-2s_it}
\le L_0\max_{1\le i\le p}s_i e^{-2s_it}$
for all $t\ge0$.
For $t>0$, let $i_0=\min\{i:s_it\le1\}$ when this set is nonempty.
Then
\[
 \sum_{s_it\le1}s_ie^{-2s_it}
 \le\sum_{i=i_0}^p s_i
 \le\frac43s_{i_0}
 \le\frac{4e^2}{3}\max_{1\le i\le p}s_ie^{-2s_it}.
\]
If the set is empty, the sum is zero.
For $s_it>1$, successive terms, ordered by increasing $s_i$,
decrease by a factor at most $4e^{-6}$.
The case $t=0$ is an ordinary geometric sum.
Each individual atom satisfies $w_ts_i e^{-2s_it}\le M$ for
$t\in T$ and $1\le i\le p$.
Consequently,
$\sum_{i=1}^p\sum_{t\in T:\,t\le R}w_ts_ie^{-2s_it}
\le L_0MN_\mu(R)$.

To bound the contribution from $t>R$, apply the scalar upper bound
at the smallest of these decay rates, $s_p$:
\[
 s_i\sum_{\substack{t\in T\\t>R}}w_te^{-2s_it}
 \le M\frac{s_i}{s_p}e^{-2(s_i-s_p)R}
 \le M\frac{s_i}{s_p}e^{-(s_i/s_p-1)}
 \qquad 1\le i\le p.
\]
Since $s_i/s_p\ge4^{p-i}$ and $u\mapsto ue^{1-u}$ is decreasing
on $[1,\infty)$,
\[
 \sum_{i=1}^p\sum_{\substack{t\in T\\t>R}}w_ts_ie^{-2s_it}
 \le M\sum_{i=1}^p\frac{s_i}{s_p}e^{1-s_i/s_p}
 \le M\sum_{k=0}^{p-1}4^ke^{1-4^k}
 \le L_1M,
\]
where $L_1=\sum_{k=0}^\infty4^ke^{1-4^k}<\infty$.
Summing the scalar lower bounds yields
$\frac{l}{3}K_A(R)\le pl\le L_0MN_\mu(R)+L_1M$,
and hence
$N_\mu(R)\ge\frac{l}{3L_0M}K_A(R)-\frac{L_1}{L_0}$.

For the upper bound, let $\nu$ be the real sampling design in
Proposition~\ref{prop:reference-grid}, with parameters $h,D$. A retained grid point $kh\le\log R$, $k\ge0$, lies
within distance $D$ of some $q_j\le\log R+D$, $j\in\mathcal J$.
Each interval $[q_j-D,q_j+D]$ contains at most $2D/h+2$ grid points.
Since $q_j=j\log2$, there are at most
$K_A(R)+\lfloor D/\log2\rfloor+2$ such indices $j$, including $j=0$.
Counting also the sampling time zero gives
\[
 N_\nu(R)\le1+(2D/h+2)
 \bigl(K_A(R)+\lfloor D/\log2\rfloor+2\bigr)
 \qquad R\ge2.
\]
The integer construction in that proposition adds at most $C_0$
initial sampling times.
The resulting integer design $\mu$ satisfies
$N_\mu(R)\le N_\nu(R+1)+C_0$ for $R\ge2$.
Since $K_A(R+1)\le K_A(R)+1$, the required upper bound follows.
\end{proof}

Let $(A,C)$ be an observable system with $A$ positive and
$\operatorname{rank}C=r<\infty$.
An orthonormal basis of $\operatorname{ran}C$ gives a representation
by $r$ generating vectors.
By \cite[Theorem~2.3]{CabrelliMolterPaternostroPhilipp2020},
$A$ has an orthonormal basis of eigenvectors, every eigenspace has
dimension at most $r$, and the distinct positive eigenvalues form
a union of at most $r$ Carleson sequences.
When $r>1$, the union need not itself be a Carleson sequence.

For each constituent sequence, the logarithmic parameters defined in
Section~\ref{sec:positive-core} are uniformly separated in the Euclidean distance.
Each spectral band corresponds to an interval of length $\log2$
in these parameters, so each sequence contributes a uniformly
bounded number of eigenvalues to the band.
Since there are at most $r$ sequences and each eigenspace has
dimension at most $r$, we obtain $\sup_{j\ge1}\dim P_jH<\infty$.
These dimensions need not be bounded by $r$, since a band may
contain several distinct eigenvalues.

Define the eigenvalue counting function with multiplicity by
$n_A(R)=\operatorname{rank}E_A((0,e^{-1/R}])$.
There are constants $M,D>0$ such that
$K_A(R)-D\le n_A(R)\le M K_A(R)+D$ for all $R\ge2$.
Thus Theorem~\ref{thm:band-count} determines the optimal order
of the sampling count in terms of $n_A$.

\begin{prop}
\label{prop:finite-channel-count}
Let $(A,C)$ be an observable system with $A$ positive and
$r=\operatorname{rank}C<\infty$. Every admissible real design
satisfies $rN_\mu(R)\ge n_A(R)-K_\mu$ for all $R>0$, with a
constant $K_\mu\ge0$.
\end{prop}

\begin{proof}
If $\dim H<\infty$, take $K_\mu=\dim H$.
Assume now that $H$ is infinite dimensional.
The zero eigenspace does not contribute to $n_A$.
On $(\ker A)^\perp$, choose an orthonormal eigenbasis $(e_j)_{j\ge1}$
such that $Ae_j=e^{-s_j}e_j$, where
$s_1\ge s_2\ge\cdots>0$ and $s_j\to0$.
Eigenvalues are counted with multiplicity.
The separation and multiplicity bounds give an integer $q\ge1$ such that
$\sup_{v\in\mathbb R}\#\{j:\log(1/s_j)\in[v,v+\log2]\}\le q$.
Consequently, $s_j/s_{j+q}\ge2$ for every $j\ge1$.
The frame inequalities for $(A,C)$ give
$a(1-e^{-2s_j})\le\|Ce_j\|^2\le b(1-e^{-2s_j})$.
Since $0<s_j\le s_1<\infty$, there are constants $a_0,b_0>0$ such that
$a_0s_j\le\|Ce_j\|^2\le b_0s_j$ for every $j\ge1$.

Let $\alpha,\beta$ be sampled frame bounds, and let $K\ge1$
be an integer to be chosen independently of $R$.
For $R>0$, put $m=n_A(R)=\#\{j\ge1:s_j\ge R^{-1}\}$.
If $m\le K$, then $rN_\mu(R)\ge m-K$ is immediate.
Assume $m>K$.
For $1\le j\le m-K$, set $u=s_j/s_m$.
Since $s_mR\ge1$, the sampled upper frame bound gives
\[
 \sum_{\substack{t\in T\\t>R}}w_t\|CA^te_j\|^2
 \le
 \frac{\|Ce_j\|^2}{\|Ce_m\|^2}
 e^{-2(s_j-s_m)R}
 \sum_{\substack{t\in T\\t>R}}w_t\|CA^te_m\|^2
 \le \beta\frac{b_0}{a_0}u e^{-2(u-1)}.
\]
Let $E=\operatorname{span}\{e_1,\ldots,e_{m-K}\}$,
let $P_E$ be the orthogonal projection onto $E$, and set
$S_{\mu,>R}=\sum_{t\in T\cap(R,\infty)}w_t(A^t)^*C^*CA^t$.
Iteration of $s_j/s_{j+q}\ge2$ gives
$s_j/s_m\ge2^{\lfloor(m-j)/q\rfloor}$ for $1\le j\le m$.
Since $u\mapsto ue^{-2(u-1)}$ decreases on $[1,\infty)$,
\[
 \operatorname{tr}(P_ES_{\mu,>R}P_E)
 =\sum_{j=1}^{m-K}
   \sum_{\substack{t\in T\\t>R}}w_t\|CA^te_j\|^2
 \le\beta\frac{b_0}{a_0}
 \sum_{\ell=K}^\infty
 2^{\lfloor\ell/q\rfloor}
 \exp\bigl[-2(2^{\lfloor\ell/q\rfloor}-1)\bigr].
\]
The series converges.
Choose $K$ so that the right side is smaller than $\alpha$.

Define $V_Rx=(CA^tx)_{t\in T\cap[0,R]}$.
Then $\operatorname{rank}V_R\le rN_\mu(R)$.
If $rN_\mu(R)<m-K$, then
\[
 \dim\ker(V_R|_E)
 =\dim E-\operatorname{rank}(V_R|_E)
 \ge m-K-rN_\mu(R)>0.
\]
Choose a unit vector $x\in E\cap\ker V_R$.
Then $\alpha\le\langle S_\mu x,x\rangle
=\langle S_{\mu,>R}x,x\rangle
\le\operatorname{tr}(P_ES_{\mu,>R}P_E)<\alpha$,
a contradiction.
Hence $rN_\mu(R)\ge m-K=n_A(R)-K$, and $K_\mu=K$
gives the assertion.
\end{proof}

The coefficient $1/r$ in
$N_\mu(R)\ge n_A(R)/r-K_\mu/r$ cannot be increased in a bound
valid for all observable systems $(A,C)$ with $A$ positive and
$\operatorname{rank}C=r$. For some fixed systems, a larger
coefficient is possible.

\begin{example}
Let $A_0$ be a positive operator on an infinite dimensional Hilbert
space $H_0$ such that $(A_0^nf_0)_{n\ge0}$ is a frame. Set $H=H_0^{\oplus r}$ and $A=A_0^{\oplus r}$.
Take generating vectors $f_1,\ldots,f_r$, where $f_i$ has
$i$th component $f_0$ and all other components zero.
Then $n_A(R)=r m_{A_0}(R)$, and a design is admissible
exactly when it is admissible for $(A_0,f_0)$. By Theorem~\ref{thm:positive-count}, for every $\eta>0$ there is
an admissible integer sampling design $\mu$, using the same times
and weights in each summand, such that
$N_\mu(R)\le(1+\eta)n_A(R)/r+O_\eta(1)$.

In contrast, let $A=A_0\oplus\lambda_*I$ on
$H_0\oplus\mathbb C^{r-1}$, where $0<\lambda_*<1$ and $r\ge2$,
and set $C(x,z)=(\langle x,f_0\rangle,z)$.
The corresponding integer orbits form a frame and $\operatorname{rank}C=r$.
For all sufficiently large $R$, one has
$n_A(R)=m_{A_0}(R)+r-1$. Restriction to $H_0$ shows that every
admissible design satisfies
$N_\mu(R)\ge n_A(R)-O_\mu(1)$.
An upper bound with coefficient $(1+\eta)/r<1$ is therefore
impossible for this fixed orbit frame.
\end{example}

For infinitely many generating vectors, the spectral restrictions above can
fail. Let $H=Y=L^2((0,1),d\lambda)$, $(Ax)(\lambda)=\lambda x(\lambda)$,
and $C=(I-A^2)^{1/2}$. An orthonormal basis $(e_j)_{j\ge1}$ of $H$
gives generating vectors $f_j=Ce_j$ for $j\ge1$.
The geometric series gives $S=I$.
The positive operator $A$ has no eigenvalues, and
$\dim P_jH=\infty$ for every $j\ge1$.
Theorem~\ref{thm:band-count} still applies, but
$\operatorname{rank}E_A((0,e^{-1/R}])=\infty$ for every $R>0$,
so the formula defining $n_A$ no longer gives a finite counting function.

\begin{remark}
For an observable system $(A,C)$ with $A$ normal and $SA=AS$,
Lemma~\ref{lem:stein} gives $B=(I-|A|^2)S$.
Whenever $S_\mu$ is bounded, spectral calculus gives
\[
 S^{-1/2}S_\mu S^{-1/2}
 =\sum_{t\in T}w_t(I-|A|^2)|A|^{2t}.
\]
Put $\varphi_\mu(\rho)=(1-\rho^2)\sum_{t\in T}w_t\rho^{2t}$
for $0\le\rho<1$, and let $E_{|A|}$ be the spectral measure of $|A|$.
Thus $\mu$ is admissible if and only if there are constants
$0<c\le d<\infty$ such that $c\le\varphi_\mu(\rho)\le d$
for almost every $\rho\in[0,1)$ with respect to $E_{|A|}$.

In the multiplication example above, take $T=\mathbb N_0$,
$w_{2n}=1$, and $w_{2n+1}=2^{-n-1}$ for $n\in\mathbb N_0$.
Then $\tfrac12\le\varphi_\mu(\rho)\le1$ for $0\le\rho<1$.
The unit weights $v_t=1$ for $t\in\mathbb N_0$ also give a frame,
but $\inf_{t\in\mathbb N_0}w_t/v_t=0$.
Thus the weight comparison in Proposition~\ref{prop:finite-excess-weights}
can fail without the finite excess assumption.
\end{remark}

\subsection{Common sampling designs for a fixed positive spectral set}
\label{sec:common-positive}
We show that sampling designs for a single canonical orbit
characterize common sampling designs for all observable positive
systems whose spectral measures are concentrated on the same
Carleson set.

Let $(\lambda_j)_{j\ge1}$ be a sequence in $(0,1)$ satisfying
the Carleson condition.
After reindexing, we may assume that $\lambda_1<\lambda_2<\cdots$.
Put $\Lambda=\{\lambda_j:j\ge1\}$ and
$s_j=-\log\lambda_j$ for $j\ge1$.
Then $s_1>s_2>\cdots>0$ and $s_j\to0$.
Since $1-e^{-s}\sim s$ as $s\downarrow0$,
\cite[Proposition~1.2(i)--(ii)]{ChristensenEtAl2024} shows that
the Carleson condition is equivalent to
$\inf_{j\ge1}\log(s_j/s_{j+1})>0$.

Fix an orthonormal basis $(e_j)_{j\ge1}$ of $H$ and define
$A_\Lambda e_j=\lambda_je_j$ for $j\ge1$,
$f_\Lambda=\sum_{j\ge1}\sqrt{s_j}\,e_j$, and
$C_\Lambda:H\to\mathbb C$ by
$C_\Lambda x=\langle x,f_\Lambda\rangle$ for $x\in H$.
The separation condition gives $\sum_{j\ge1}s_j<\infty$,
so $f_\Lambda\in H$.
Set $q_\Lambda=(1-e^{-2s_1})/s_1$.
Since $0<q_\Lambda\le(1-e^{-2s_j})/s_j\le2$ for every $j\ge1$,
the frame characterization in \cite{ACMT2017} shows that
$(A_\Lambda^nf_\Lambda)_{n\in\mathbb N_0}$ is a frame for $H$.
We call this family the canonical orbit associated with $\Lambda$.

Let $\mathcal C(\Lambda)$ consist of observable systems $(A,C)$ with
$A$ positive and $E_A(\mathbb R\setminus\Lambda)=0$.
The eigenspaces and the separable observation space $Y$ may have finite
or infinite dimension, and $E_A(\{e^{-s_j}\})$ may be zero for some $j\ge1$.
The spectral condition permits $1\in\sigma(A)$ but implies $E_A(\{1\})=0$.
A sampling design is common to $\mathcal C(\Lambda)$ if the same times
and weights are admissible for every system in this class,
with sampled frame bounds that may depend on the system.

\begin{thm}
\label{thm:scalar-universal-spectrum}
A sampling design $\mu$ is common to $\mathcal C(\Lambda)$ if and
only if it is admissible for the canonical orbit.
If $0<\alpha_0\le\beta_0<\infty$ are sampled frame bounds for this orbit,
then $a q_\Lambda\alpha_0 I\le S_\mu\le2b\beta_0 I$ for every
$(A,C)\in\mathcal C(\Lambda)$ with frame bounds $a,b$.
\end{thm}

\begin{proof}
Suppose that $\mu$ is admissible for the canonical orbit with
sampled frame bounds $\alpha_0,\beta_0$.
Fix $(A,C)\in\mathcal C(\Lambda)$ with frame bounds $a,b$.
For $j\ge1$, put $P_j=E_A(\{e^{-s_j}\})$.
These projections are mutually orthogonal and $\sum_{j\ge1}P_j=I$
in the strong operator topology.
By Lemma~\ref{lem:stein},
$P_jC^*CP_j=(1-e^{-2s_j})P_jSP_j$ for every $j\ge1$.
Thus the map
\[
 Z:H\longrightarrow\ell^2(\mathbb N;Y),\qquad
 Zx=\left(\frac{CP_jx}{\sqrt{s_j}}\right)_{j\ge1},\qquad x\in H,
\]
satisfies $a q_\Lambda\|x\|^2\le\|Zx\|^2\le2b\|x\|^2$ for every $x\in H$,
using $aI\le S\le bI$ and
$q_\Lambda\le(1-e^{-2s_j})/s_j\le2$ for every $j\ge1$.

In the orthonormal basis defining the canonical orbit, its sampled
analysis operator is represented by
$H_\mu:\ell^2(\mathbb N)\to\ell^2(T)$, with entries
$(H_\mu)_{t,j}=\sqrt{w_t s_j}\,e^{-ts_j}$ for $t\in T$ and $j\ge1$.
Expanding in an orthonormal basis of $Y$ gives
$\alpha_0\|y\|^2\le\|(H_\mu\otimes I_Y)y\|^2\le\beta_0\|y\|^2$
for every $y\in\ell^2(\mathbb N;Y)$.
For $x\in H$ and $t\in T$, spectral expansion gives
\[
 \bigl((H_\mu\otimes I_Y)Zx\bigr)_t
 =\sqrt{w_t}\sum_{j\ge1}e^{-ts_j}CP_jx
 =\sqrt{w_t}CA^tx.
\]
The series converges in $Y$, since
$\sum_{j\ge1}e^{-ts_j}P_jx=A^tx$ in $H$ and $C$ is bounded.
Thus the sampled analysis operator is the bounded operator
$\mathcal O_\mu=(H_\mu\otimes I_Y)Z$.
The bounds for the two factors give
$a q_\Lambda\alpha_0\|x\|^2\le\|\mathcal O_\mu x\|^2
\le2b\beta_0\|x\|^2$ for every $x\in H$.
The reverse implication holds because
$(A_\Lambda,C_\Lambda)\in\mathcal C(\Lambda)$.
\end{proof}

\begin{cor}
Let $(A,C)$ be an observable system with $A$ positive and
$r=\operatorname{rank}C<\infty$.
Assume that all but finitely many eigenspaces have dimension $r$.
Then every admissible real design satisfies
$N_\mu(R)\ge n_A(R)/r-O_\mu(1)$ as $R\to\infty$.
For every $\eta>0$, there is an admissible integer design with
$N_\mu(R)\le(1+\eta)n_A(R)/r+O_\eta(1)$ as $R\to\infty$.
\end{cor}

\begin{proof}
The lower bound follows from Proposition~\ref{prop:finite-channel-count}.
Fix $\eta>0$.
If $\dim H<\infty$, then $\|A\|<1$ by Lemma~\ref{lem:stein},
and Theorem~\ref{thm:finite-times} gives an admissible design at
finitely many integer times, proving the upper bound.

Assume that $H$ is infinite dimensional.
An orthonormal basis of $\operatorname{ran}C$ represents the system
by $r$ generating vectors. By
\cite[Corollary~2.9]{CabrelliMolterPaternostroPhilipp2020},
the distinct positive eigenvalues form a Carleson sequence.
Let $\Lambda_A$ be the set of these eigenvalues and put
$H_+=(\ker A)^\perp$.
Their logarithmic parameters are uniformly separated, as shown in
Section~\ref{sec:positive-core}.
The canonical orbit associated with $\Lambda_A$ has spectral counting
function $n_A(R)/r+O(1)$ as $R\to\infty$, because only finitely many
positive eigenvalues have multiplicity different from $r$.
Theorem~\ref{thm:positive-count} gives an integer design $\nu$ with the
required upper counting bound. By
Theorem~\ref{thm:scalar-universal-spectrum}, it is admissible for
$(A|_{H_+},C|_{H_+})$.

The space $\ker A$ has dimension at most $r$, and its observations vanish
at every positive time. Thus the sampled analysis operator extends
boundedly to $H$.
Its range is the sum of the closed range from $H_+$ and the finite
dimensional image of $\ker A$, so it is closed.
Its kernel is finite dimensional because its restriction to $H_+$
is injective and $H_+$ has finite codimension.
Lemma~\ref{lem:finite-completion} gives an admissible design $\mu$
by adding finitely many integer observations to $\nu$.
Since $N_\mu(R)\le N_\nu(R)+O(1)$ as $R\to\infty$, the upper bound follows.
\end{proof}

Put $m_\Lambda(R)=\#\{j\ge1:s_j\ge R^{-1}\}$ and
$\delta_\Lambda=\limsup_{R\to\infty}m_\Lambda(R)/\log R$.

\begin{cor}
Every common design for $\mathcal C(\Lambda)$ satisfies
$N_\mu(R)\ge m_\Lambda(R)-O_\mu(1)$. For every \(\eta>0\),
there is a common integer design such that
\(
 N_\mu(R)\le(1+\eta)m_\Lambda(R)+O_\eta(1).
\)
The optimal logarithmic counting coefficient is $\delta_\Lambda$
for both real and integer times. It is attained if $\delta_\Lambda=0$ or if
the hypotheses of Theorem~\ref{thm:critical-regular-spectra} or
Theorem~\ref{thm:positive-separated-critical} hold.
\end{cor}

\begin{proof}
Every common design $\mu$ is admissible for the canonical orbit.
For every $\eta>0$, Theorem~\ref{thm:positive-count} also gives
an integer design $\mu_\eta$ admissible for that orbit, which is
common by Theorem~\ref{thm:scalar-universal-spectrum}.
Thus
\[
 N_\mu(R)\ge m_\Lambda(R)-O_\mu(1),\qquad
 N_{\mu_\eta}(R)\le(1+\eta)m_\Lambda(R)+O_\eta(1).
\]
Consequently,
\[
 \limsup_{R\to\infty}\frac{N_\mu(R)}{\log R}\ge\delta_\Lambda,
 \qquad
 \limsup_{R\to\infty}\frac{N_{\mu_\eta}(R)}{\log R}
 \le(1+\eta)\delta_\Lambda.
\]
Letting $\eta\downarrow0$ proves that the optimal coefficient
is $\delta_\Lambda$ for both real and integer times.

If $\delta_\Lambda=0$, the design $\mu_1$ attains this coefficient.
Under either of the stated spectral hypotheses, the corresponding
theorem gives real and integer designs $\mu_*$ on the canonical
orbit satisfying
$N_{\mu_*}(R)=m_\Lambda(R)+O(1)$ and
$\limsup_{R\to\infty}N_{\mu_*}(R)/\log R=\delta_\Lambda$.
These designs are common by
Theorem~\ref{thm:scalar-universal-spectrum}.

For integer times, rounding and any necessary finite completion
are performed on the canonical orbit before applying the factorization
$\mathcal O_\mu=(H_\mu\otimes I_Y)Z$.
Thus finite completion is used only where the relevant kernels
are finite dimensional.
\end{proof}

\begin{prop}
Suppose that $s_j/s_{j+1}\ge9$ for all sufficiently large $j$.
Then the following hold.
\begin{enumerate}[(i)]
\item There exist common real and integer sampling designs for
$\mathcal C(\Lambda)$ satisfying
$N_\mu(R)=m_\Lambda(R)+O(1)$ as $R\to\infty$.

\item These designs may be chosen so that $w_t=t$ outside a finite set
and the sampled families on the canonical orbit have finite excess.

\item If $s_j/s_{j+1}\ge9$ for every $j\ge1$, the real design may be chosen
as $\mu=\sum_{j\ge1}t_j\delta_{t_j}$, where $t_j=(2s_j)^{-1}$.
\end{enumerate}
\end{prop}

\begin{proof}
\emph{(i)--(ii).}
Choose $J\ge1$ such that $s_j/s_{j+1}\ge9$ for all $j\ge J$, and put
$H_J=\overline{\operatorname{span}}\{e_j:j\ge J\}$.
Apply Theorem~\ref{thm:positive-separated-critical} to the canonical
orbit restricted to $H_J$.
For either resulting design $\nu=\sum_{t\in T}v_t\delta_t$,
the synthesis operator $V_J:\ell^2(T)\to H_J$ is Fredholm,
and $v_t=t$ outside a finite set.

The sampled upper frame bound gives
$\sum_{t\in T}v_t e^{-2s_Jt}<\infty$.
For $1\le j<J$, we have
\[
 \sum_{t\in T}v_t s_j e^{-2s_jt}
 \le s_j\sum_{t\in T}v_t e^{-2s_Jt}<\infty.
\]
Regard $V_J$ as an operator into $H$.
The synthesis operator $V$ of the full sampled family therefore satisfies
\(
 V=V_J+F, \operatorname{rank}F\le J-1,
\)
where $F$ is bounded.
Since $\operatorname{codim}_H H_J=J-1$, the operator $V_J$
remains Fredholm.
Hence $V$ is Fredholm by stability under finite rank perturbations;
see \cite[Chapter~1]{BottcherSilbermann2006}.

Lemma~\ref{lem:finite-orbit-completion} gives a completed design $\mu$
whose sampled family on the canonical orbit is a frame of finite excess.
Only finitely many integer times are added, so $w_t=t$ outside a finite set.
Moreover,
\[
 N_\mu(R)=N_\nu(R)+O(1)
 =\#\{j\ge J:s_j\ge R^{-1}\}+O(1)
 =m_\Lambda(R)+O(1)
 \qquad R\to\infty.
\]
Theorem~\ref{thm:scalar-universal-spectrum} makes $\mu$ common to
$\mathcal C(\Lambda)$.

\emph{(iii).}
Taking $J=1$, Theorem~\ref{thm:positive-separated-critical} gives
$\mu=\sum_{j\ge1}t_j\delta_{t_j}$ with $t_j=(2s_j)^{-1}$.
Its sampled family on the canonical orbit is a Riesz basis,
so no finite completion is needed.
Theorem~\ref{thm:scalar-universal-spectrum} makes this design
common to $\mathcal C(\Lambda)$.
\end{proof}

\begin{remark}
Let $\nu=\sum_{t\in T}v_t\delta_t$ be a common design for
$\mathcal C(\Lambda)$, and suppose that its sampled family
on the canonical orbit has finite excess.
By Theorem~\ref{thm:scalar-universal-spectrum} and
Corollary~\ref{cor:critical-orbit-weights}, positive weights
$w_t$, $t\in T$, give a common design
$\mu=\sum_{t\in T}w_t\delta_t$ if and only if
\[
 0<\inf_{t\in T}\frac{w_t}{v_t}
 \le\sup_{t\in T}\frac{w_t}{v_t}<\infty.
\]
For either design constructed above, $v_t=t$ outside a finite set,
so this condition is equivalent to $w_t\asymp t$ outside that set.
\end{remark}

\subsection{Unweighted sampling with a compact observation operator}
\label{sec:unweighted}

For Carleson frames with positive spectrum, Christensen et al.\
proved that sampling on $N\mathbb N_0$ preserves the frame
property for every integer $N\ge1$. They also
conjectured that choosing one integer time from each interval
$[Nk,N(k+1))$, $k\in\mathbb N_0$, preserves the frame property
for every fixed integer $N\ge1$
\cite[Theorem~2.1 and the discussion after Theorem~2.5]{ChristensenEtAl2024}.
Krishtal and Miller proved this conjecture and extended it to real
sampling times \cite[Theorem~2.3 and Corollary~2.4]{KrishtalMiller2026}.
The BDC theorem for operators with positive spectrum
\cite[Theorem~1.3]{KrishtalMillerBDC2026} extends this conclusion
to sampling sets with positive finite natural density, including
sets whose gaps are unbounded, and also allows nonnormal block
diagonal generators.
For invertible systems with finitely many generators, bounded
temporal perturbations are treated under suitable spectral
assumptions in
\cite[Theorem~4.1]{KrishtalMashreghiMillerSectorial2026}.

We study bounded additive perturbations of the sampling times
$Nk$, $k\in\mathbb N_0$, in an observation operator formulation
for observable systems $(A,C)$ with $A$ positive and $C$ compact.
Finite dimensional observation spaces correspond to finitely many
generators, while infinite dimensional observation spaces allow
countably many observation coordinates.
When $A$ is injective, for each fixed system $(A,C)$, integer
$N\ge1$, and $h\ge0$, we obtain frame bounds uniform over all
sequences $(t_k)_{k\in\mathbb N_0}$ satisfying
$t_k\ge0$ and $|t_k-Nk|\le h$. When $A$ is not injective, the sampled observations satisfy
frame inequalities if and only if the sampling sequence
contains zero.

Throughout this subsection, let $(A,C)$ be an observable system on
$\mathcal H$ with observation space $\mathcal Y$, positive $A$,
and frame bounds $a_0,b_0$. By Lemma~\ref{lem:stein},
$A$ is a contraction. The space $\mathcal Y$ may be infinite dimensional.

For each integer $N\ge1$, define
$U_N:\mathcal H\to\ell^2(\mathbb N_0;\mathcal Y)$ by
$U_Nx=(CA^{Nk}x)_{k\in\mathbb N_0}$ for $x\in\mathcal H$.

\begin{lem}
\label{lem:arithmetic-sampling}
For each integer $N\ge1$, there exists $c_N>0$ such that
$c_N\|x\|\le\|U_Nx\|\le\sqrt{b_0}\|x\|$
for every $x\in\mathcal H$.
\end{lem}

\begin{proof}
The lower bound follows from
\cite[Proposition~5.3]{WuUniversalNonuniform2026}.
The upper bound follows from
$N\mathbb N_0\subset\mathbb N_0$.
\end{proof}

For an integer $N\ge1$ and $h\ge0$, let $\mathcal T_{N,h}$ be
the set of sequences $t=(t_k)_{k\in\mathbb N_0}$ satisfying
$t_k\ge0$ and $|t_k-Nk|\le h$ for every $k\in\mathbb N_0$.
Repeated times are allowed, and each occurrence contributes
separately to the observation sum.

\begin{lem}
\label{lem:compact-timing-error}
Fix an integer $N\ge1$ and $h\ge0$.
For each $t\in\mathcal T_{N,h}$, define
$E_t:\mathcal H\to\ell^2(\mathbb N_0;\mathcal Y)$ by
$E_tx=(C(A^{t_k}-A^{Nk})x)_{k\in\mathbb N_0}$
for $x\in\mathcal H$.
Then the following hold.
\begin{enumerate}[(i)]
\item The family $(E_t)_{t\in\mathcal T_{N,h}}$ is uniformly bounded.
\item If $C$ is compact, each $E_t$ is compact.
\item If $C$ is compact and $A$ is injective, the set
$\{E_t:t\in\mathcal T_{N,h}\}$ is compact in operator norm.
\end{enumerate}
\end{lem}

\begin{proof}
For $h=0$, we have $E_t=0$ for every $t\in\mathcal T_{N,h}$.
Assume $h>0$ and let $K\in\mathbb N$ satisfy $NK>h$.
For every $t\in\mathcal T_{N,h}$, $0<\lambda<1$, and integer $k\ge K$,
$|\lambda^{t_k}-\lambda^{Nk}|\leq\lambda^{Nk-h}(1-\lambda^{2h})$.
Consequently,
\[
 \sum_{k=K}^{\infty}|\lambda^{t_k}-\lambda^{Nk}|^2
 \leq q_K(\lambda)
 :=\lambda^{2(NK-h)}
       \frac{(1-\lambda^{2h})^2}{1-\lambda^{2N}}.
\]
We extend $q_K$ continuously by zero at both endpoints.
The quotient is bounded on $[0,1)$ and tends to zero as $\lambda\uparrow1$.
For each $\delta\in(0,1)$, the factor
$\lambda^{2(NK-h)}$ tends to zero uniformly as $K\to\infty$
on $[0,1-\delta]$. Since this factor is at most one on $[0,1]$, it follows that
$\eta_K:=\sup_{0\leq\lambda\leq1}q_K(\lambda)\to0$ as $K\to\infty$.
The scalar estimate also holds at $\lambda=0,1$.
For every $x\in\mathcal H$, the spectral theorem therefore gives
\[
 \sum_{k=K}^{\infty}\|C(A^{t_k}-A^{Nk})x\|^2
 \leq\|C\|^2\eta_K\|x\|^2
 \qquad t\in\mathcal T_{N,h}.
\]
The coordinates with $0\le k<K$ are uniformly bounded as well,
which proves (i).

Assume that $C$ is compact. For $K\in\mathbb N$, define
\(
 E_t^{(K)}x
 =\bigl(\mathbf 1_{\{k<K\}}C(A^{t_k}-A^{Nk})x\bigr)_{k\in\mathbb N_0}.
\)
Each operator $C(A^{t_k}-A^{Nk})$ is compact, so $E_t^{(K)}$
is compact. For $NK>h$, the tail estimate gives
$\sup_{t\in\mathcal T_{N,h}}\|E_t-E_t^{(K)}\|
\le \|C\|\sqrt{\eta_K}\to0$ as $K\to\infty$.
Thus each $E_t$ is an operator norm limit of compact operators
and is therefore compact. This proves (ii).

Now assume that $C$ is compact and $A$ is injective.
The spectral calculus gives strong continuity of $s\mapsto A^s$
on $[0,\infty)$. Approximating $C^*$ in norm by finite rank operators
and using $\|A^s\|\le1$ gives
$\|CA^s-CA^{s_0}\|=\|(A^s-A^{s_0})C^*\|\to0$
as $s\to s_0$ in $[0,\infty)$, for every $s_0\ge0$.
A diagonal subsequence argument shows that $\mathcal T_{N,h}$
is compact for coordinatewise convergence.
Let $t^{(m)}\to t$ coordinatewise in $\mathcal T_{N,h}$.
For each fixed $K$, norm continuity of $s\mapsto CA^s$ gives
\[
 \|E_{t^{(m)}}^{(K)}-E_t^{(K)}\|
 \le
 \left(\sum_{k=0}^{K-1}
 \|C(A^{t_k^{(m)}}-A^{t_k})\|^2\right)^{1/2}
 \longrightarrow0.
\]
For $NK>h$, the uniform tail estimate yields
$\|E_{t^{(m)}}-E_t\|
\le \|E_{t^{(m)}}^{(K)}-E_t^{(K)}\|
+2\|C\|\sqrt{\eta_K}$.
Letting first $m\to\infty$ and then $K\to\infty$ gives
$\|E_{t^{(m)}}-E_t\|\to0$.
Thus $t\mapsto E_t$ is continuous in operator norm.
Since $\mathcal T_{N,h}$ is compact for coordinatewise convergence,
its image $\{E_t:t\in\mathcal T_{N,h}\}$ is compact in operator norm.
This proves (iii).
\end{proof}

\begin{thm}
\label{thm:compact-observation-jitter}
Assume that $A$ is injective and $C$ is compact. For each integer
$N\ge1$ and each $h\ge0$, there are constants
$\alpha_{N,h},\beta_{N,h}>0$ such that
every $t\in\mathcal T_{N,h}$ satisfies
\[
 \alpha_{N,h}\|x\|^2
 \leq\sum_{k\geq0}\|CA^{t_k}x\|^2
 \leq\beta_{N,h}\|x\|^2,
 \qquad x\in\mathcal H.
\]
\end{thm}

\begin{proof}
Write $U_tx=(CA^{t_k}x)_{k\ge0}$.
Lemmas~\ref{lem:arithmetic-sampling}
and~\ref{lem:compact-timing-error} give
$U_t=U_N+E_t$, where $U_N$ is bounded below and $E_t$ is compact.
Hence $U_t^*U_t$ is a compact perturbation of the positive invertible
operator $U_N^*U_N$. It is therefore Fredholm, so it suffices to show
that $\ker U_t=\{0\}$.

Suppose that $U_tx=0$. For $y\in\mathcal Y$, the function
$F_y(z)=\langle CA^z x,y\rangle$ is bounded and analytic on
$\operatorname{Re}z>0$.
Since $A$ is injective, its spectral projection at zero is zero.
Set $\Lambda_t=\{t_k:k\ge0,\ t_k>0\}$.
Then $F_y(s)=0$ for every $s\in\Lambda_t$.
Since $t_k=Nk+O(1)$, the multiplicities of the sampling times are
uniformly bounded, and hence
$\sum_{s\in\Lambda_t}\frac{s}{1+s^2}=\infty$.
On the other hand, the Blaschke condition \cite{Garnett2007} gives
$F_y\not\equiv0 \Longrightarrow
\sum_{s\in\Lambda_t}\frac{s}{1+s^2}<\infty$.
Thus $F_y\equiv0$. This holds for every $y\in\mathcal Y$,
so $CA^s x=0$ for all $s>0$.
By strong convergence $A^s x\to x$ as $s\downarrow0$,
we also have $Cx=0$. The lower frame bound for $(A,C)$
yields $x=0$.

It follows that $U_t^*U_t$ is invertible for every
$t\in\mathcal T_{N,h}$.
By Lemma~\ref{lem:compact-timing-error}, the set
$\{U_t:t\in\mathcal T_{N,h}\}$ is compact in operator norm.
For bounded operators $V,W:\mathcal H\to
\ell^2(\mathbb N_0;\mathcal Y)$, put
$m(V)=\inf_{\|x\|=1}\|Vx\|$. Then $|m(V)-m(W)|\le\|V-W\|$.
Since $m(U_t)>0$ for every $t\in\mathcal T_{N,h}$,
compactness gives
$\alpha_{N,h}:=\min_{t\in\mathcal T_{N,h}}m(U_t)^2>0$
and
$\beta_{N,h}:=\max_{t\in\mathcal T_{N,h}}\|U_t\|^2<\infty$.
Consequently, for every $t\in\mathcal T_{N,h}$ and
$x\in\mathcal H$,
\[
 \alpha_{N,h}\|x\|^2
 \le \|U_tx\|^2
 =\sum_{k\ge0}\|CA^{t_k}x\|^2
 \le \beta_{N,h}\|x\|^2.
\]
\end{proof}

The proof uses a compact perturbation argument.
Related methods appear in
\cite[Proposition~5.4]{KrishtalMillerBDC2026} and
\cite[Theorem~2.7]{KrishtalMashreghiMillerSectorial2026}.

\begin{cor}
Suppose that \(A\) is injective and \(C\) is compact.
Fix integers \(N,M\ge1\). For every set \(T\subset[0,\infty)\)
satisfying \(1\le\#(T\cap[Nk,N(k+1)))\le M\) for all \(k\in\mathbb N_0\),
the family \((CA^t)_{t\in T}\) satisfies frame
inequalities with bounds independent of \(T\).
\end{cor}

\begin{proof}
For each $k\in\mathbb N_0$, write
$T\cap[Nk,N(k+1))=\{s_{k,1},\ldots,s_{k,m_k}\}$,
where $s_{k,1}<\cdots<s_{k,m_k}$ and $1\le m_k\le M$.
Set $s_{k,r}=s_{k,1}$ for $m_k<r\le M$.
For each $1\le r\le M$, we have
$0\le s_{k,r}-Nk<N$, so
$(s_{k,r})_{k\ge0}\in\mathcal T_{N,N}$.
Theorem~\ref{thm:compact-observation-jitter} therefore gives,
for every $x\in\mathcal H$,
\[
 \alpha_{N,N}\|x\|^2
 \le \sum_{k\ge0}\|CA^{s_{k,1}}x\|^2
 \le \sum_{s\in T}\|CA^s x\|^2\le\sum_{r=1}^{M}\sum_{k\ge0}\|CA^{s_{k,r}}x\|^2
 \le M\beta_{N,N}\|x\|^2.
\]
\end{proof}

\begin{remark}
Assume that $C$ is compact.
If $A$ is injective, a set $T\subset\mathbb N_0$ is
admissible with weight one whenever it meets $[Nk,N(k+1))$ for
every integer $k\ge k_0$, where $N\ge1$ is an integer and
$k_0\in\mathbb N_0$. Choose
$t_j\in T\cap[N(j+k_0),N(j+k_0+1))$ for $j\in\mathbb N_0$.
Then $0\le t_j-Nj<N(k_0+1)$, so
Theorem~\ref{thm:compact-observation-jitter} applies with $h=N(k_0+1)$.
The upper bound $b_0$ follows from $T\subset\mathbb N_0$.
Consequently, $T\setminus F$ remains admissible with weight one
for every finite set $F\subset T$.

If $A$ is not injective, a sequence $t\in\mathcal T_{N,h}$ gives
a frame if and only if $t_k=0$ for some $k\in\mathbb N_0$.
If zero is included, the analytic argument in the proof of
Theorem~\ref{thm:compact-observation-jitter}, with $0^z=0$ for
$\operatorname{Re}z>0$, gives
$\sum_{k\ge0}\|CA^{t_k}x\|^2=0
\Longrightarrow CA^nx=0$ for every $n\in\mathbb N$.
The observation at zero also gives $Cx=0$. Hence
$a_0\|x\|^2\le\sum_{n\ge0}\|CA^nx\|^2=0$, so $x=0$.
The compact perturbation argument then gives the lower frame bound.

If zero is omitted, choose $0\ne x\in\ker A$. Since $t_k>0$ for every $k\in\mathbb N_0$, we have
$A^{t_k}x=0$ for every $k\in\mathbb N_0$ and
$\sum_{k\ge0}\|CA^{t_k}x\|^2=0$.
Thus no positive lower frame bound exists.
\end{remark}

\begin{example}
On $\ell^2(\mathbb N)$, let $Ae_j=\lambda_j e_j$,
where $\lambda_j=(j+1)^{-1}$, and let
$Ce_j=\sqrt{1-\lambda_j^2}\,e_j$ for $j\ge1$.
Then $A$ is positive and injective, $C$ is bounded but not compact, and the observations at all nonnegative integers form a Parseval frame:
$\sum_{n\ge0}\|CA^n x\|^2=\|x\|^2$ for every $x\in\ell^2(\mathbb N)$.
However, the sequence $t_k=2k+1$ for $k\ge0$, which chooses one point
in every interval $[2k,2(k+1))$, satisfies
\[
 \sum_{k\geq0}\|CA^{2k+1}e_j\|^2
 =\frac{\lambda_j^2}{1+\lambda_j^2}\longrightarrow0.
\]
Even sampling at every positive integer gives no positive lower
frame bound,
since $\sum_{n\geq1}\|CA^n e_j\|^2=\lambda_j^2\to0$.
Thus, for a general noncompact observation operator, selecting one
time from each interval need not preserve the frame property.
Deleting time zero may also destroy that property.
\end{example}

We conclude with necessary counting estimates for unweighted
integer sampling of a single orbit.

\begin{prop}
\label{prop:unweighted-spectral-count}
Let $A\in\mathcal B(\mathcal H)$ be positive and injective,
and suppose that $(A^nf)_{n\in\mathbb N_0}$ is a frame
with upper frame bound $b_0$.
Let $(e_j)$ and $(s_j)$ be as in
Section~\ref{sec:positive-core}.

Suppose that $T\subset\mathbb N_0$ and $(A^nf)_{n\in T}$ is a
frame with lower frame bound $\alpha_T>0$.
Set $L=\frac12\log\frac{2b_0}{\alpha_T}$.
Then $L>0$ and, for every $j\ge1$,
\[
 N_T(L/s_j)
 \ge \frac{\alpha_T}{2b_0(1-e^{-2s_j})}
 \ge \frac{\alpha_T}{4b_0s_j}.
\]
Consequently,
$\limsup_{R\to\infty}\frac{N_T(R)}{R}
\ge \frac{\alpha_T}{4b_0L}>0$.
In particular, no such unweighted sampling set satisfies
$N_T(R)=o(R)$.
\end{prop}

\begin{proof} Since $(A^nf)_{n\in\mathbb N_0}$ has upper frame bound $b_0$, we have
$\frac{|f_j|^2}{1-e^{-2s_j}}
=\sum_{n\ge0}|\langle e_j,A^nf\rangle|^2\le b_0$.
The sampled lower frame inequality therefore yields, for every
$R\ge0$,
\begin{align*}
 \alpha_T
 &\le |f_j|^2\sum_{n\in T}e^{-2s_jn}
 \le b_0(1-e^{-2s_j})N_T(R)
   +b_0(1-e^{-2s_j})
      \sum_{\substack{n\in\mathbb N_0\\n>R}}e^{-2s_jn}\\
 &=b_0(1-e^{-2s_j})N_T(R)
   +b_0e^{-2s_j(\lfloor R\rfloor+1)}
 \le b_0(1-e^{-2s_j})N_T(R)+b_0e^{-2s_jR}.
\end{align*}
Since the sampled family is a subfamily of the full orbit frame,
$\alpha_T\le b_0$, and hence $L>0$.
Taking $R=L/s_j$ and using $b_0e^{-2L}=\alpha_T/2$, we obtain
$N_T(L/s_j)\ge\frac{\alpha_T}{2b_0(1-e^{-2s_j})}$.
The second inequality follows from $1-e^{-2s_j}\le2s_j$.
Finally, $L/s_j\to\infty$, and division by $L/s_j$ gives the
asserted lower bound for the limit superior.
\end{proof}

\begin{cor}
\label{cor:unweighted-linear-count}
Let $A$, $f$, and $(s_j)_{j\ge1}$ be as in
Proposition~\ref{prop:unweighted-spectral-count}, and assume that
$Q:=\sup_{j\ge1}\frac{s_j}{s_{j+1}}<\infty$.
Then the following hold.
\begin{enumerate}[(i)]
\item Every set $T\subset\mathbb N_0$ for which
$(A^nf)_{n\in T}$ is a frame satisfies
$N_T(R)\asymp R$ as $R\to\infty$.

\item For every integer $M\ge1$, the set $T_M=M\mathbb N_0$
gives an unweighted frame and satisfies
$N_{T_M}(R)=\left\lfloor\frac{R}{M}\right\rfloor+1$
for every $R\ge0$.
\end{enumerate}
\end{cor}

\begin{proof}
\emph{(i).}
Fix $T\subset\mathbb N_0$ such that $(A^nf)_{n\in T}$ is a frame,
let $\alpha_T>0$ be a lower frame bound, and choose $L$ as in
Proposition~\ref{prop:unweighted-spectral-count}.
For $R\ge L/s_1$, choose $j\ge1$ such that
$\frac{L}{s_j}\le R<\frac{L}{s_{j+1}}$.
Since $s_j/s_{j+1}\le Q$, monotonicity of $N_T$ and
Proposition~\ref{prop:unweighted-spectral-count} give
\[
 N_T(R)
 \ge N_T(L/s_j)
 \ge \frac{\alpha_T}{4b_0s_j}
 \ge \frac{\alpha_T}{4b_0LQ}\,R.
\]
On the other hand, $T\subset\mathbb N_0$ implies
$N_T(R)\le\lfloor R\rfloor+1$.
Thus $N_T(R)\asymp R$ as $R\to\infty$.

\emph{(ii).}
Take $Cx=\langle x,f\rangle$.
The operator $C$ is compact, so
Theorem~\ref{thm:compact-observation-jitter}, applied with
$N=M$ and $h=0$, shows that $(A^{Mk}f)_{k\in\mathbb N_0}$
is a frame.
The counting formula follows from $T_M=M\mathbb N_0$.
\end{proof}

For the single orbit systems of Section~\ref{sec:positive-core},
Theorem~\ref{thm:positive-count} provides admissible integer
designs with $N_\mu(R)=O(\log R)$.
Their supports have natural density zero and therefore do not
yield frames with unit weights, by
Proposition~\ref{prop:unweighted-spectral-count};
see also \cite[Theorem~1.3]{KrishtalMillerBDC2026}.

\section{Optimal condition numbers}
\label{sec:conditioning}

We first study attainment of the infimum of the frame condition
number for integer sampling under counting bounds and for positive
weights on a fixed Riesz basis. We then compute the minimum for
a particular positive operator orbit with geometrically spaced
sampling times.

\subsection{Attainment of optimal condition numbers}
\label{sec:counting-conditioning}

For a fixed observable system $(A,C)$, we study whether
the infimum of $\kappa(S_\mu)$ is attained among admissible
integer sampling designs supported in a prescribed set
and satisfying a fixed bound on $N_\mu(R)$.
Both the sampling support and the positive weights may vary.

For a fixed Riesz basis $(u_n)$, we study whether the infimum
of the frame condition number of $(\sqrt{a_n}u_n)$ is attained
over all positive sequences $(a_n)$ for which this family
is a frame.

\begin{thm}
\label{thm:counting-condition-attainment}
Let $(A,C)$ be an observable system with $A$ normal and $\|A\|=1$.
Fix $T_0\subset\mathbb N_0$ and a function
$F:[0,\infty)\to[0,\infty)$.
If the class of admissible sampling designs supported on $T_0$ and satisfying
$N_\mu(R)\le F(R)$ for every $R\ge0$ is nonempty, then
$\kappa(S_\mu)$ attains its minimum on this class.
\end{thm}

\begin{proof}
Let $\Gamma_*$ be the infimum of the condition numbers over this class.
Choose designs $(\mu_j)_{j\in\mathbb N}$ in the class such that
$\Gamma_j:=\kappa(S_{\mu_j})\to\Gamma_*$ as $j\to\infty$.
Normalize the weights of each design so that its optimal
lower frame bound is one.
Writing $\mu_j=\sum_{n=0}^{\infty}w_n^{(j)}\delta_n$, with zero
weights outside its support, we have
\[
 I\le S_j:=S_{\mu_j}
 =\sum_{n=0}^{\infty}w_n^{(j)}(A^n)^*C^*CA^n
 \le\Gamma_j I
 \qquad j\in\mathbb N.
\]
Set $M=\sup_{j\in\mathbb N}\Gamma_j<\infty$.

Let $E_A$ denote the spectral measure of $A$. For $0<r_1<1$, set
$P_1=E_A(\{z\in\mathbb C:|z|\ge r_1\})$.
Since $\|A\|=1$, we have $P_1\ne0$. Also $CP_1\ne0$: otherwise
$CA^nP_1=CP_1A^n=0$ for every $n\in\mathbb N_0$, contradicting
observability. The restriction of $A$ to $P_1H$ is invertible,
and its inverse powers have norm at most $r_1^{-n}$ for
$n\in\mathbb N_0$. Hence
$\|CA^n\|\ge\|CA^nP_1\|\ge r_1^n\|CP_1\|>0$ for every
$n\in\mathbb N_0$. Positivity of the summands in $S_j$ gives
$0\le w_n^{(j)}\le M/\|CA^n\|^2$ for
$n\in\mathbb N_0$ and $j\in\mathbb N$.
Passing to a diagonal subsequence, we may assume that
$w_n^{(j)}\to w_n\ge0$ as $j\to\infty$ for every $n\in\mathbb N_0$.
Taking limits in finite sums gives
$0\le\sum_{n=0}^Nw_n(A^n)^*C^*CA^n\le\Gamma_*I$
for every $N\in\mathbb N_0$.
Since these partial sums are increasing and bounded above,
there exists an operator $S_*$ with $0\le S_*\le\Gamma_*I$ such that
\[
 \lim_{N\to\infty}
 \left\|\sum_{n=0}^Nw_n(A^n)^*C^*CA^n x-S_*x\right\|=0,
 \qquad x\in H.
\]

Fix $0<r<1$ and $x\in E_A(\{z\in\mathbb C:|z|\le r\})H$.
Choose $r<r_1<1$ and use the projection $P_1$ above.
For $N\in\mathbb N_0$ and $j\in\mathbb N$, normality and the
weight estimate give
\[
 \sum_{n=N+1}^{\infty}w_n^{(j)}\|CA^nx\|^2
 \le
 \frac{M\|C\|^2}{\|CP_1\|^2}\|x\|^2
 \sum_{n=N+1}^{\infty}\left(\frac r{r_1}\right)^{2n}.
\]
The right hand side tends to zero as $N\to\infty$, uniformly in
$j\in\mathbb N$, and the same bound holds for the limiting weights.
Convergence of the finite sums therefore gives
$\langle S_*x,x\rangle=\lim_{j\to\infty}\langle S_jx,x\rangle\ge\|x\|^2$.
By Lemma~\ref{lem:stein}, the spectral projection on the unit circle is
zero. The union of these spectral subspaces over $0<r<1$ is therefore
dense in $H$, so $I\le S_*\le\Gamma_*I$.

Set $T=\{n\in\mathbb N_0:w_n>0\}$ and
$\mu=\sum_{n\in T}w_n\delta_n$.
Weights outside $T_0$ remain zero, so $T\subset T_0$.
For each $R\ge0$, the finite set $T\cap[0,R]$ is contained in
$\operatorname{supp}\mu_j$ for all sufficiently large $j$.
Thus $N_\mu(R)\le N_{\mu_j}(R)\le F(R)$ for such $j$.
Since $S_\mu=S_*$, the design is admissible and
$\kappa(S_\mu)\le\Gamma_*$. The definition of $\Gamma_*$ gives equality.
\end{proof}

Taking $T_0=\mathbb N_0$ and $F(R)=\lfloor R\rfloor+1$
gives an admissible integer sampling design of minimum
condition number.

Let $\{A^n f:n\in\mathbb N_0\}$ be a frame for $H$, with $A$ positive,
and let $m_A$ be as in Section~\ref{sec:positive-core}.
Fix $K\ge0$ and $R_0\ge2$. In the following definitions, $\mu$ ranges
over admissible integer sampling designs. Set
\[
 \Gamma_K^{\mathrm{crit}}
 =\inf\bigl\{\kappa(S_\mu):
 N_\mu(R)\le m_A(R)+K\text{ for every }R\ge R_0\bigr\}.
\]
For $\Gamma\ge1$, define $c_{\Gamma,K}$ as the infimum of all $c\ge0$
such that some $\mu$ satisfies $\kappa(S_\mu)\le\Gamma$ and
$N_\mu(R)\le c\,m_A(R)+K$ for every $R\ge R_0$.

\begin{cor}
\label{cor:optimal-counting-envelope}
Fix $K\ge0$, $R_0\ge2$, and $\Gamma\ge1$.
\begin{enumerate}[(i)]
\item If the class defining $\Gamma_K^{\mathrm{crit}}$ is nonempty,
the infimum is attained, and every minimizer satisfies
$N_\mu(R)=m_A(R)+O(1)$ as $R\to\infty$.
\item If the set defining $c_{\Gamma,K}$ is nonempty,
its infimum is attained and $c_{\Gamma,K}\ge1$.
\end{enumerate}
\end{cor}

\begin{proof}
The spectral description in Section~\ref{sec:positive-core} gives
$\|A\|=1$ and $m_A(R)\to\infty$ as $R\to\infty$.
For (i), apply Theorem~\ref{thm:counting-condition-attainment} with
$T_0=\mathbb N_0$, $F(R)=m_A(R)+K$ for $R\ge R_0$, and
$F(R)=\lfloor R\rfloor+1$ for $0\le R<R_0$.
For every minimizer, Theorem~\ref{thm:positive-count} gives
$N_\mu(R)\ge m_A(R)-O_\mu(1)$ as $R\to\infty$.
Together with the defining upper bound, this proves (i).

For (ii), choose $(c_j)_{j\in\mathbb N}$ in the set defining
$c_{\Gamma,K}$ with $c_j\to c_{\Gamma,K}$ as $j\to\infty$, and
corresponding admissible integer designs $(\mu_j)_{j\in\mathbb N}$.
Normalize the weights of each design so that
$I\le S_j:=S_{\mu_j}\le\Gamma I$ for every $j\in\mathbb N$.
Write $\mu_j=\sum_{n=0}^{\infty}w_n^{(j)}\delta_n$, with zero
weights outside its support. The diagonal subsequence argument
and uniform tail estimates in the proof of
Theorem~\ref{thm:counting-condition-attainment} give, after passing
to a subsequence and relabeling,
$w_n^{(j)}\to w_n\ge0$ for every $n\in\mathbb N_0$, and
$\mu:=\sum_{\substack{n\in\mathbb N_0\\w_n>0}}w_n\delta_n$
satisfies $I\le S_\mu\le\Gamma I$.
For each fixed $R\ge R_0$, the counting argument in that proof gives
$N_\mu(R)\le\liminf_{j\to\infty}N_{\mu_j}(R)\le c_{\Gamma,K}m_A(R)+K$.
Thus the infimum is attained. The lower bound in
Theorem~\ref{thm:positive-count} and $m_A(R)\to\infty$ as $R\to\infty$
then imply $c_{\Gamma,K}\ge1$.
\end{proof}

\begin{remark}
If the class defining $\Gamma_K^{\mathrm{crit}}$ is nonempty,
part~(i) of Corollary~\ref{cor:optimal-counting-envelope} shows
that it contains a design with $\kappa(S_\mu)\le\Gamma$
if and only if $\Gamma\ge\Gamma_K^{\mathrm{crit}}$.
The integer constructions in
Theorems~\ref{thm:critical-regular-spectra}
and~\ref{thm:positive-separated-critical} imply that this class
is nonempty for some $K\ge0$ under their respective assumptions.
\end{remark}

\begin{example}
\label{ex:finite-conditioning-nonattainment}
Let $H=\mathbb C^2$, $A=\operatorname{diag}(1/4,1/2)$, and
$Cx=\langle x,(1,1)\rangle$ for $x\in H$.
The observation vectors at times zero and one are linearly independent,
and $\|A\|<1$, so $(A,C)$ is observable.
Let $\mathcal D_{\mathrm r}$ and $\mathcal D_{\mathrm i}$ denote
the classes of admissible real and integer sampling designs,
respectively. Then
\[
 \inf_{\mu\in\mathcal D_{\mathrm r}}\kappa(S_\mu)
 =
 \inf_{\mu\in\mathcal D_{\mathrm i}}\kappa(S_\mu)
 =(1+\sqrt2)^2.
\]
Neither infimum is attained.

For an admissible design $\mu=\sum_{t\in T}w_t\delta_t$, set
$q_t=2^{-t}$ and $a_t=w_t4^{-t}$ for $t\in T$. Its sampled frame operator is
\[
 S_\mu=\begin{pmatrix}X&Y\\Y&Z\end{pmatrix},
\]
where $X=\sum_{t\in T}a_tq_t^2$, $Y=\sum_{t\in T}a_tq_t$, and
$Z=\sum_{t\in T}a_t$.
Set $\tau=X+Z$ and $\Delta=((Z-X)^2+4Y^2)^{1/2}$.
The eigenvalues are $(\tau\pm\Delta)/2$, with $0\le\Delta<\tau$.
An admissible design contains at least one positive sampling time,
so $0<q_t<1$ for some $t\in T$. Hence
$Z-X+2Y-\tau=2\sum_{t\in T}a_tq_t(1-q_t)>0$.
Cauchy--Schwarz gives
$\Delta\ge(Z-X+2Y)/\sqrt2>\tau/\sqrt2$, and hence
\[
 \kappa(S_\mu)=\frac{\tau+\Delta}{\tau-\Delta}
 >(1+\sqrt2)^2.
\]

For $m\in\mathbb N$, define
$\mu_m=\tfrac12\delta_0+\frac{4^m}{1+4^{-m}}\delta_m$.
Its weighted observation vectors are $u_0=(1,1)/\sqrt2$ and
$u_m=(2^{-m},1)/\sqrt{1+4^{-m}}$.
They are linearly independent unit vectors, so $\mu_m$ is admissible.
Their inner product is $d_m=(1+2^{-m})/\sqrt{2(1+4^{-m})}$,
and $S_{\mu_m}$ has eigenvalues $1\pm d_m$.
Since $d_m\to1/\sqrt2$ as $m\to\infty$, we obtain
$\kappa(S_{\mu_m})=(1+d_m)/(1-d_m)\to(1+\sqrt2)^2$,
proving both infimum formulas.
\end{example}

\begin{remark}
The conclusion of Theorem~\ref{thm:counting-condition-attainment}
remains valid for an observable system $(A,C)$ with $A$ normal,
$0<\rho=\|A\|\le1$, and $E_A(\{z\in\mathbb C:|z|=\rho\})=0$.
Indeed, using the notation of that proof, take $0<r<r_1<\rho$
and set $P_1=E_A(\{z\in\mathbb C:|z|\ge r_1\})$.
For $x\in E_A(\{z\in\mathbb C:|z|\le r\})H$, we have
\[
 \sup_{j\in\mathbb N}
 \sum_{n=N+1}^{\infty}w_n^{(j)}\|CA^nx\|^2
 \le
 \frac{M\|C\|^2}{\|CP_1\|^2}\|x\|^2
 \sum_{n=N+1}^{\infty}\left(\frac r{r_1}\right)^{2n}
 \longrightarrow0
 \qquad\text{as }N\to\infty.
\]
Moreover, the spectral assumption gives
\[
 \overline{\bigcup_{0<r<\rho}
 E_A(\{z\in\mathbb C:|z|\le r\})H}=H.
\]

For $\rho=1$, the spectral condition follows from
Lemma~\ref{lem:stein}. In Example~\ref{ex:finite-conditioning-nonattainment},
this condition fails because $\rho=1/2$ is an eigenvalue.
\end{remark}

Let $(u_n)_{n\in\mathbb N}$ be a Riesz basis for $H$, with synthesis
operator $V:\ell^2(\mathbb N)\to H$, and set $K=V^*V$.
The operator $K$ is positive, bounded, and invertible.
For a bounded real sequence $a=(a_n)_{n\in\mathbb N}$, define
$D_ae_n=a_ne_n$ for $n\in\mathbb N$, where $(e_n)_{n\in\mathbb N}$
is the standard orthonormal basis of $\ell^2(\mathbb N)$.
For $0<\alpha\le\beta<\infty$, the identity
$V^{-1}(V^*)^{-1}=K^{-1}$ gives
$\alpha I\le VD_aV^*\le\beta I
\Longleftrightarrow
\alpha K^{-1}\le D_a\le\beta K^{-1}$.
If $a_n>0$ for every $n\in\mathbb N$, these inequalities show that
$(\sqrt{a_n}u_n)_{n\in\mathbb N}$ is a frame if and only if
$\inf_{n\in\mathbb N}a_n>0$. Its frame operator is then $VD_aV^*$.

If $\inf_{n\in\mathbb N}a_n>0$ and $\Gamma\ge1$, then
$\kappa(VD_aV^*)\le\Gamma$ if and only if
$K^{-1}\le cD_a\le\Gamma K^{-1}$ for some $c>0$.
Taking $c$ to be the reciprocal of the optimal lower frame bound
normalizes that bound to one. The infimum of the condition numbers
over all positive weights that preserve the frame property is
\[
 \Gamma_*=\inf\bigl\{\Gamma\ge1:
 K^{-1}\le D_a\le\Gamma K^{-1}
 \text{ for some }a\in\ell^\infty(\mathbb N;\mathbb R)\bigr\}.
\]

\begin{prop}
\label{prop:optimal-riesz-weights}
The infimum $\Gamma_*$ is attained. The optimal weight sequences,
normalized so that the optimal lower frame bound is one, form a
convex set that is compact for coordinatewise convergence.
\end{prop}

\begin{proof}
Choose $0<\alpha_0\le\beta_0<\infty$ with
$\alpha_0I\le K\le\beta_0I$.
The constant weights $a_n=1/\alpha_0$ for $n\in\mathbb N$ show that
the set defining $\Gamma_*$ is nonempty and
$\Gamma_*\le\beta_0/\alpha_0$.
Choose $(\Gamma_j,a^{(j)})_{j\in\mathbb N}$, with
$a^{(j)}=(a_n^{(j)})_{n\in\mathbb N}$, such that
$K^{-1}\le D_{a^{(j)}}\le\Gamma_jK^{-1}$ and
$\Gamma_j\le\beta_0/\alpha_0+1$ for every $j\in\mathbb N$, and
$\Gamma_j\to\Gamma_*$ as $j\to\infty$.
Taking diagonal entries gives
$1/\beta_0\le a_n^{(j)}\le\Gamma_j/\alpha_0$ for every
$n,j\in\mathbb N$.
Passing to a diagonal subsequence, we may assume that
$a_n^{(j)}\to a_n$ as $j\to\infty$ for every $n\in\mathbb N$.
The uniform bounds imply that $a=(a_n)_{n\in\mathbb N}$ is bounded
and $a_n\ge1/\beta_0$ for every $n\in\mathbb N$.
For every finitely supported $c\in\ell^2(\mathbb N)$, taking limits gives
$\langle K^{-1}c,c\rangle\le\langle D_ac,c\rangle
\le\Gamma_*\langle K^{-1}c,c\rangle$.
By density, $K^{-1}\le D_a\le\Gamma_*K^{-1}$.
Thus $\kappa(VD_aV^*)\le\Gamma_*$, and equality follows from the
definition of $\Gamma_*$.

The inequalities $I\le VD_aV^*\le\Gamma_*I$ force the optimal frame
bounds to be $1$ and $\Gamma_*$; otherwise the condition number would
be smaller than $\Gamma_*$.
Hence the normalized optimal weight sequences are exactly those satisfying
$K^{-1}\le D_a\le\Gamma_*K^{-1}$.
These inequalities are linear in $a$, so the set is convex.
The same argument with finitely supported vectors shows that it is
closed in $\prod_{n\in\mathbb N}[1/\beta_0,\Gamma_*/\alpha_0]$.
This product is compact for coordinatewise convergence, so the set
of normalized optimal weight sequences is compact as well.
\end{proof}

For a frame as in Corollary~\ref{cor:critical-orbit-weights},
Lemma~\ref{lem:finite-deletion-weights} gives a finite set of sampling
times whose deletion leaves a Riesz basis.
Enumerate the remaining times and their reference weights as
$(t_n)_{n\in\mathbb N}$ and $(v_n)_{n\in\mathbb N}$, and set
$u_n=\sqrt{v_n}A^{t_n}f$ for $n\in\mathbb N$.
Proposition~\ref{prop:optimal-riesz-weights}, applied with
$a_n=w_n/v_n$ for $n\in\mathbb N$, then gives positive sampling
weights $(w_n)_{n\in\mathbb N}$ attaining the minimum condition
number among those preserving the frame property on this fixed set.
Finite deletion changes the counting function by at most a constant,
so the logarithmic counting coefficient is unchanged.

\begin{example}
Let $(e_1,e_2)$ be the standard orthonormal basis of $\mathbb C^2$,
and set $u_1=e_1$, $u_2=e_2$, and $u_3=e_1+e_2$.
The family $(u_i)_{i=1}^3$ is a frame of excess one.
For $a_i>0$, $i=1,2,3$, the frame operator of
$(\sqrt{a_i}u_i)_{i=1}^3$ is
\[
 S_a=\begin{pmatrix}
 a_1+a_3&a_3\\
 a_3&a_2+a_3
 \end{pmatrix}.
\]
Since $a_3>0$, the operator $S_a$ is not a scalar multiple of $I$,
so $\kappa(S_a)>1$.
For $\varepsilon>0$, the weights $a_1=a_2=1$ and $a_3=\varepsilon$
give eigenvalues $1$ and $1+2\varepsilon$.
Hence $\kappa(S_a)=1+2\varepsilon\to1$ as $\varepsilon\downarrow0$.
The infimum over strictly positive weights is therefore one and is
not attained. If zero weights are allowed, $a_1=a_2=1$ and $a_3=0$
give $S_a=I$.

The same failure occurs in an infinite dimensional $H$ after embedding
$\mathbb C^2$ in $H$ and adding an orthonormal basis of its orthogonal
complement. For the approximating weights above, assign unit weights
to the additional vectors. The resulting frame still has excess one.
Thus finite excess cannot replace the Riesz basis assumption in
Proposition~\ref{prop:optimal-riesz-weights} when every weight must
remain strictly positive.
\end{example}

\subsection{Optimal condition numbers for geometric sampling}
\label{sec:geometric-conditioning}

We now compute the minimum condition number for the positive operator
orbit defined below. The eigenvalues are $e^{-s_j}$ for $j\in\mathbb N$,
where the decay rates $(s_j)_{j\in\mathbb N}$ form a geometric
progression. The real sampling times also form a geometric progression,
with finitely many additions or deletions allowed. For the full
progression, weights equal to the sampling times give a Toeplitz
matrix for the analysis operator in the orthonormal eigenbasis of $A$.

A real sequence $(a_n)_{n\in\mathbb Z}$ is a P\'olya frequency sequence
if $\det(a_{i_r-j_s})_{r,s=1}^m\ge0$ for every $m\in\mathbb N$
and all integers $i_1<\cdots<i_m$ and $j_1<\cdots<j_m$;
see \cite{Edrei1953}.
Let $\mathbb T=\{z\in\mathbb C:|z|=1\}$.
For a continuous function $p:\mathbb T\to\mathbb C\setminus\{0\}$,
its winding number about zero is
$(\varphi(2\pi)-\varphi(0))/(2\pi)$, where
$\varphi:[0,2\pi]\to\mathbb R$ is continuous and satisfies
$p(e^{i\theta})=|p(e^{i\theta})|e^{i\varphi(\theta)}$ for
$\theta\in[0,2\pi]$; see \cite{BottcherSilbermann2006}.

\begin{lem}
\label{lem:pf-toeplitz-conditioning}
Let $(a_n)_{n\in\mathbb Z}$ be a nonzero P\'olya frequency sequence
whose Laurent series $p(z)=\sum_{n\in\mathbb Z}a_nz^n$ converges in
an open annulus containing $\mathbb T$. Then the following hold.
\begin{enumerate}[(i)]
\item The only possible zero of $p$ on $\mathbb T$ is $-1$, and
$\max_{z\in\mathbb T}|p(z)|=p(1)$ and
$\min_{z\in\mathbb T}|p(z)|=|p(-1)|$.
\item If $p(-1)\ne0$ and $p$ has winding number zero, then
$T(p)=(a_{k-j})_{k,j\in\mathbb N}$ is bounded and invertible
on $\ell^2(\mathbb N)$, with $\|T(p)\|=p(1)$ and
$\|T(p)^{-1}\|=|p(-1)|^{-1}$.
\end{enumerate}
\end{lem}

\begin{proof}
For (i), Edrei's representation \cite{Edrei1953} gives
\[
 p(z)=Cz^m e^{\gamma z+\delta z^{-1}}
 \prod_{\nu=1}^{\infty}
 \frac{(1+\alpha_\nu z)(1+\widetilde\alpha_\nu z^{-1})}
      {(1-\beta_\nu z)(1-\widetilde\beta_\nu z^{-1})},
\]
where $C>0$, $m\in\mathbb Z$, $\gamma,\delta\ge0$, and the sequences
$(\alpha_\nu)_{\nu\in\mathbb N}$,
$(\widetilde\alpha_\nu)_{\nu\in\mathbb N}$,
$(\beta_\nu)_{\nu\in\mathbb N}$, and
$(\widetilde\beta_\nu)_{\nu\in\mathbb N}$ are nonnegative and summable.
The annulus of convergence contains $\mathbb T$, so
$\beta_\nu,\widetilde\beta_\nu<1$ for every $\nu\in\mathbb N$.
For $\alpha\ge0$ and $0\le\beta<1$, the modulus of
$(1+\alpha z)/(1-\beta z)$ on $\mathbb T$ is largest at $1$
and smallest at $-1$.
For $z=e^{i\theta}\in\mathbb T$, we also have
$e^{-(\gamma+\delta)}
\le |e^{\gamma z+\delta z^{-1}}|
=e^{(\gamma+\delta)\cos\theta}
\le e^{\gamma+\delta}$.
Passing to the limit in the finite products gives
$|p(-1)|\le |p(z)|\le p(1)$ for every $z\in\mathbb T$,
which proves the two extremal formulas.
The representation also shows that a zero on $\mathbb T$
can occur only at $-1$.

For (ii), assume $p(-1)\ne0$ and that $p$ has winding number zero.
The condition $p(-1)\ne0$ excludes numerator parameters equal to one.
Summability gives
$\#\{\nu\in\mathbb N:\alpha_\nu>1\}
+\#\{\nu\in\mathbb N:\widetilde\alpha_\nu>1\}<\infty$.
For these parameters, use
\[
 1+\alpha_\nu z=\alpha_\nu z(1+\alpha_\nu^{-1}z^{-1}),
 \qquad
 1+\widetilde\alpha_\nu z^{-1}
 =\widetilde\alpha_\nu z^{-1}(1+\widetilde\alpha_\nu^{-1}z).
\]
After absorbing the constant factors into $C$, combining the powers of $z$, and relabeling the parameters, we obtain
\[
 p(z)=Cz^\ell p_-(z)p_+(z),\qquad
 p_+(z)=e^{\gamma z}\prod_{\nu=1}^{\infty}
 \frac{1+\alpha_\nu z}{1-\beta_\nu z},\qquad
 p_-(z)=e^{\delta z^{-1}}\prod_{\nu=1}^{\infty}
 \frac{1+\widetilde\alpha_\nu z^{-1}}
      {1-\widetilde\beta_\nu z^{-1}}.
\]
Here $\ell\in\mathbb Z$, and all four parameter sequences remain
summable and have terms in $[0,1)$. Their terms therefore have a
common upper bound less than one. Thus $p_+$ and $p_+^{-1}$ are
analytic in a disk of radius greater than one, whereas $p_-$ and
$p_-^{-1}$ are analytic outside a circle of radius less than one,
including infinity. Both factors have winding number zero, so
$\ell=0$.

The Laurent coefficients of $p_+$ and $p_+^{-1}$ vanish at negative
indices, and those of $p_-$ and $p_-^{-1}$ vanish at positive indices.
All four coefficient sequences are absolutely summable.
The Toeplitz product rule \cite{BottcherSilbermann2006} gives
$T(p)=C T(p_-)T(p_+)$ and
$T(p)^{-1}=C^{-1}T(p_+^{-1})T(p_-^{-1})$.
Using $\|T(q)\|\le\max_{z\in\mathbb T}|q(z)|$ for continuous $q$
on $\mathbb T$ \cite{BottcherSilbermann2006}, we obtain
\[
 \|T(p)\|\le p(1),\qquad
 \|T(p)^{-1}\|\le C^{-1}
 \max_{z\in\mathbb T}|p_+(z)^{-1}|
 \max_{z\in\mathbb T}|p_-(z)^{-1}|
 =|p(-1)|^{-1},
\]
since $|p_+|$ and $|p_-|$ both attain their minima at $-1$.

The Hartman--Wintner spectral inclusion
\cite[Theorem~2.30(a)]{BottcherSilbermann2006} gives
$p(\mathbb T)\subset\sigma(T(p))$.
Hence $\|T(p)\|\ge p(1)$, and spectral mapping for the
inverse gives $\|T(p)^{-1}\|\ge |p(-1)|^{-1}$.
These inequalities complete the proof.
\end{proof}

Fix $h>0$ and an orthonormal basis $(e_j)_{j\in\mathbb N}$ of $H$.
Set $s_j=e^{-jh}$ for $j\in\mathbb N$, and define a positive
operator $A$ and a vector $f$ by $Ae_j=e^{-s_j}e_j$ for
$j\in\mathbb N$ and $f=\sum_{j=1}^{\infty}\sqrt{s_j}e_j$.
Since $s_j/s_{j+1}=e^h>1$ for every $j\in\mathbb N$, the argument
in Section~\ref{sec:common-positive} shows that
$(A^nf)_{n\in\mathbb N_0}$ is a frame for $H$.

Let $g(u)=\exp(u/2-e^u)$ for $u\in\mathbb R$, as in
Lemma~\ref{lem:logarithmic-kernel}.
For a phase parameter $\xi\in\mathbb R$, set
$t_k=e^{\xi+kh}$ for $k\in\mathbb N$ and define
$p_{h,\xi}(z)=\sum_{n\in\mathbb Z}g(\xi+nh)z^n$.
The series converges absolutely for $|z|>e^{-h/2}$, so its annulus
of convergence contains $\mathbb T$.
\begin{thm}
\label{thm:geometric-optimal-weights}
The following hold for the operator $A$ and vector $f$ defined above.
\begin{enumerate}[(i)]
\item If $\mu=\sum_{k=k_0}^{\infty}w_k\delta_{t_k}+\nu$ is
admissible, where $k_0\in\mathbb N$, $w_k>0$ for every integer
$k\ge k_0$, and $\nu$ is a nonnegative measure with finite support
in $[0,\infty)$, then $p_{h,\xi}(-1)\ne0$ and
\[
\kappa(S_\mu)\ge
\left(\frac{p_{h,\xi}(1)}{|p_{h,\xi}(-1)|}\right)^2.
\]
\item If $p_{h,\xi}(-1)\ne0$ and $p_{h,\xi}$ has winding number
zero, then $\mu_{h,\xi}=\sum_{k=1}^{\infty}t_k\delta_{t_k}$
attains equality in (i). Its sampled family
$(\sqrt{t_k}A^{t_k}f)_{k\in\mathbb N}$ is a Riesz basis for $H$
with optimal frame bounds $|p_{h,\xi}(-1)|^2$ and $p_{h,\xi}(1)^2$.
Such phases exist for every $h>0$.
\item If $p_{h,\xi}(-1)\ne0$, equality in (i) is attained after
adding or deleting finitely many points of $\{t_k:k\in\mathbb N\}$,
with each weight equal to its sampling time.
\end{enumerate}
\end{thm}

\begin{proof}
Write $p=p_{h,\xi}$ and $a_n=g(\xi+nh)$ for $n\in\mathbb Z$.
For every $m\in\mathbb N$ and integers
$i_1<\cdots<i_m$ and $j_1<\cdots<j_m$, total positivity of $g$
from Lemma~\ref{lem:logarithmic-kernel} gives
$\det[g((\xi+hi_r)-hj_s)]_{r,s=1}^m\ge0$.
Thus $(a_n)_{n\in\mathbb Z}$ is a P\'olya frequency sequence.

For (i), let $\alpha,\beta$ be the optimal frame bounds for $\mu$.
Define $(d_k)_{k\in\mathbb Z}$ by $d_k=w_k/t_k$ for $k\ge k_0$
and $d_k=0$ for $k<k_0$.
The upper frame bound gives
$d_k a_0^2=w_k s_k e^{-2t_ks_k}
\le \langle S_\mu e_k,e_k\rangle\le\beta$
for every integer $k\ge k_0$.
Since $a_0=g(\xi)>0$, we have
$M:=\|d\|_\infty\le\beta/a_0^2<\infty$.
Set $K=S_\nu$, which is bounded and has finite rank, and let
$L=(a_{r-j})_{r,j\in\mathbb Z}$ be the bounded Laurent operator
on $\ell^2(\mathbb Z)$.

For a finitely supported $v=(v_j)_{j\in\mathbb Z}\in\ell^2(\mathbb Z)$,
choose $n_v\in\mathbb N$ such that $n_v+j\ge1$ for every
$j\in\operatorname{supp}v$. For integers $n\ge n_v$, set
$v^{(n)}=\sum_{j\in\operatorname{supp}v}v_je_{n+j}\in H$.
Then $\|v^{(n)}\|=\|v\|$, and the frame inequalities give
\[
 \alpha\|v\|^2
 \le
 \sum_{r\in\mathbb Z}d_{n+r}|(Lv)_r|^2
 +\langle Kv^{(n)},v^{(n)}\rangle
 \le\beta\|v\|^2.
\]
Since $v^{(n)}\rightharpoonup0$ as $n\to\infty$ and $K$ has
finite rank, $\langle Kv^{(n)},v^{(n)}\rangle\to0$.

Boundedness of $d$ gives a sequence $(N_\ell)_{\ell\in\mathbb N}$
of positive integers such that $N_\ell\to\infty$ and
$N_\ell^{-1}\sum_{n=1}^{N_\ell}d_n\to\overline d$
as $\ell\to\infty$.
For every $r\in\mathbb Z$ and $N\in\mathbb N$,
\[
 \left|
 \frac1N\sum_{n=1}^{N}d_{n+r}
 -\frac1N\sum_{n=1}^{N}d_n
 \right|
 \le\frac{2|r|M}{N}.
\]
Omitting the first $n_v-1$ terms changes the average by at most
$(n_v-1)M/N$. Consequently, for every fixed $r\in\mathbb Z$,
$\frac1{N_\ell}\sum_{n=n_v}^{N_\ell}d_{n+r}\to\overline d$ and
$\frac1{N_\ell}\sum_{n=n_v}^{N_\ell}
\langle Kv^{(n)},v^{(n)}\rangle\to0$ as $\ell\to\infty$.
Since $Lv\in\ell^2(\mathbb Z)$ and the coefficient averages lie
in $[0,M]$ for all sufficiently large $\ell$, dominated convergence gives
\[
 \sum_{r\in\mathbb Z}
 \left(\frac1{N_\ell}\sum_{n=n_v}^{N_\ell}d_{n+r}\right)
 |(Lv)_r|^2
 \longrightarrow\overline d\,\|Lv\|^2.
\]
Averaging the frame inequalities over $n_v\le n\le N_\ell$
and taking limits therefore yields
$\alpha\|v\|^2\le\overline d\,\|Lv\|^2\le\beta\|v\|^2$.
Thus $\overline d>0$, and density extends these inequalities
to every $v\in\ell^2(\mathbb Z)$.

The Fourier transform identifies $L$ with multiplication by $p$;
see \cite{BottcherSilbermann2006}. Hence $\alpha\le\overline d\,|p(e^{i\theta})|^2\le\beta$
for almost every $\theta\in[0,2\pi]$.
By continuity, these inequalities hold for every
$\theta\in[0,2\pi]$.
In particular, $p(-1)\ne0$.
Lemma~\ref{lem:pf-toeplitz-conditioning}(i) now gives
\[
 \kappa(S_\mu)
 =\frac{\beta}{\alpha}
 \ge
 \frac{\overline d\,p(1)^2}{\overline d\,|p(-1)|^2}
 =
 \left(\frac{p(1)}{|p(-1)|}\right)^2.
\]
This proves (i).

For (ii), the analysis operator of
$(\sqrt{t_k}A^{t_k}f)_{k\in\mathbb N}$ has matrix entries
$U_{kj}=\sqrt{t_ks_j}\,e^{-t_ks_j}
=g(\xi+(k-j)h)=a_{k-j}$ for $k,j\in\mathbb N$,
in the orthonormal basis $(e_j)_{j\in\mathbb N}$ of $H$.
Lemma~\ref{lem:pf-toeplitz-conditioning}(ii) shows that $U=T(p)$
is bounded and invertible, with
$\|U\|=p(1)$ and $\|U^{-1}\|=|p(-1)|^{-1}$.
This gives the Riesz basis property, the optimal frame bounds,
and equality in (i).

The Zak transform argument in the proof of
Theorem~\ref{thm:critical-regular-spectra}, based on \cite{Kloos2015},
gives a phase $\xi_0\in\mathbb R$ for which $p_{h,\xi_0}$ has no
zero on $\mathbb T$.
For every $q\in\mathbb Z$, a change of index gives
$p_{h,\xi+qh}(z)=z^{-q}p_{h,\xi}(z)$ for $|z|>e^{-h/2}$.
If $m_0\in\mathbb Z$ is the winding number of $p_{h,\xi_0}$,
then $p_{h,\xi_0+m_0h}$ has winding number zero and no zero on
$\mathbb T$. This proves the existence assertion in (ii).

For (iii), Lemma~\ref{lem:pf-toeplitz-conditioning}(i) shows that
$p$ has no zero on $\mathbb T$. Let $m\in\mathbb Z$ be its winding
number and set $\widetilde\xi=\xi+mh$.
Then $p_{h,\widetilde\xi}(z)=z^{-m}p(z)$ and
$|p_{h,\widetilde\xi}(z)|=|p(z)|$ for every $z\in\mathbb T$.
Thus $p_{h,\widetilde\xi}$ has winding number zero and the same
modulus as $p$ on $\mathbb T$.
By (ii), the design $\mu_{h,\widetilde\xi}$ attains equality in (i),
with each weight equal to its sampling time.
Its support is
$\operatorname{supp}\mu_{h,\widetilde\xi}
=\{e^{\xi+kh}:k\in\mathbb Z,\ k\ge m+1\}$,
which differs from $\{t_k:k\in\mathbb N\}$ by exactly $|m|$ points.
\end{proof}

For the operator $A$ and vector $f$ defined above, let
$\Gamma_{\mathrm{geom}}(h)$ denote the infimum of $\kappa(S_\mu)$
over all admissible designs of the form in
Theorem~\ref{thm:geometric-optimal-weights}(i),
with the phase $\xi\in\mathbb R$ also allowed to vary.

\begin{cor}
\label{cor:geometric-phase-conditioning}
The following hold.
\begin{enumerate}[(i)]
\item For every $h>0$,
\[
\Gamma_{\mathrm{geom}}(h)
=\min_{\substack{0\le\xi<h\\p_{h,\xi}(-1)\ne0}}
\left(\frac{p_{h,\xi}(1)}{|p_{h,\xi}(-1)|}\right)^2.
\]
For some $\xi\in\mathbb R$, the design $\mu_{h,\xi}$ attains
this minimum.
\item As $h\downarrow0$,
$\Gamma_{\mathrm{geom}}(h)=\frac18 e^{\pi^2/h}
\left(1+O(e^{-\pi^2/h})\right)$.
\end{enumerate}
\end{cor}

\begin{proof}
For (i), fix $h>0$. The series for $p_{h,\xi}(1)$ and
$p_{h,\xi}(-1)$ converge uniformly for $\xi\in[0,h]$.
Since $p_{h,\xi}(1)>0$ and
$p_{h,\xi+h}(z)=z^{-1}p_{h,\xi}(z)$, the function
$\xi\mapsto |p_{h,\xi}(-1)|/p_{h,\xi}(1)$ is continuous and
periodic with period $h$.
By Theorem~\ref{thm:geometric-optimal-weights}(ii) and periodicity,
there exists $\xi_*\in[0,h)$ such that
\[
 \frac{|p_{h,\xi_*}(-1)|}{p_{h,\xi_*}(1)}
 =
 \max_{0\le\xi\le h}
 \frac{|p_{h,\xi}(-1)|}{p_{h,\xi}(1)}
 >0.
\]
Thus $\xi_*$ minimizes the ratio in (i).
If $m\in\mathbb Z$ is the winding number of $p_{h,\xi_*}$,
then $p_{h,\xi_*+mh}$ has winding number zero and the same modulus
on $\mathbb T$, as in the proof of
Theorem~\ref{thm:geometric-optimal-weights}.
Part (ii) of that theorem gives an attaining design, and part (i)
shows that no design in the definition of
$\Gamma_{\mathrm{geom}}(h)$ has a smaller condition number.

For (ii), write $\Gamma_E$ for Euler's gamma function and define
$\widehat g(\omega)=\int_{\mathbb R}g(u)e^{-i\omega u}\,du$
for $\omega\in\mathbb R$.
Substitution of $v=e^u$ gives
$\widehat g(\omega)=\Gamma_E(1/2-i\omega)$ for $\omega\in\mathbb R$.
We also use the identity
$|\Gamma_E(1/2+it)|^2=\pi/\cosh(\pi t)$ for $t\in\mathbb R$;
see \cite[(5.4.4)]{DLMF}.
Since $g$ is smooth and $g,g',g''\in L^1(\mathbb R)$,
Poisson summation \cite[(1.8.14)]{DLMF} gives
\[
 p_{h,\xi}(e^{i\theta})
 =\frac1h\sum_{\ell\in\mathbb Z}
 \widehat g\left(\frac{2\pi\ell-\theta}{h}\right)
 e^{i(2\pi\ell-\theta)\xi/h},
 \qquad \theta,\xi\in\mathbb R.
\]
Set $Q_h=e^{-\pi^2/h}$ and
$G_h=\Gamma_E(1/2+i\pi/h)$.
The gamma function identity gives
$|\widehat g(\omega)|\le\sqrt{2\pi}\,e^{-\pi|\omega|/2}$
for every $\omega\in\mathbb R$. Hence
\[
 \sum_{\ell\ne0}
 \left|\widehat g\left(\frac{2\pi\ell}{h}\right)\right|
 \le \frac{2\sqrt{2\pi}\,Q_h}{1-Q_h},
 \qquad
 \sum_{\ell\notin\{0,1\}}
 \left|\widehat g\left(\frac{2\pi\ell-\pi}{h}\right)\right|
 \le \frac{2\sqrt{2\pi}\,Q_h^{3/2}}{1-Q_h}.
\]
Retaining the term $\ell=0$ at $\theta=0$ and the terms
$\ell=0,1$ at $\theta=\pi$, we obtain
\[
 p_{h,\xi}(1)=\frac{\sqrt\pi}{h}\left(1+O(Q_h)\right),
 \qquad
 p_{h,\xi}(-1)
 =\frac2h\operatorname{Re}\left(G_he^{-i\pi\xi/h}\right)
 +O\left(h^{-1}Q_h^{3/2}\right),
\]
as $h\downarrow0$, uniformly for $\xi\in\mathbb R$.

Let $M_h=\max_{0\le\xi\le h}|p_{h,\xi}(-1)|$.
Since $\max_{0\le\xi\le h}
|\operatorname{Re}(G_he^{-i\pi\xi/h})|=|G_h|$,
the uniform remainder estimate gives
$\left|M_h-\frac{2|G_h|}{h}\right|
=O(h^{-1}Q_h^{3/2})$.
Moreover, $|G_h|^2=\pi/\cosh(\pi^2/h)
=2\pi Q_h/(1+Q_h^2)$,
so $M_h=(2|G_h|/h)(1+O(Q_h))$ as $h\downarrow0$.
The uniform estimate for $p_{h,\xi}(1)$ and part (i) therefore yield
\[
 \Gamma_{\mathrm{geom}}(h)
 =\frac{\pi}{h^2M_h^2}\left(1+O(Q_h)\right)
 =\frac{\cosh(\pi^2/h)}4\left(1+O(Q_h)\right)
 =\frac18 e^{\pi^2/h}\left(1+O(Q_h)\right)
 \qquad\text{as }h\downarrow0.
\]
This proves (ii).
\end{proof}

\section{Sampling and conditioning for normal operators}
\label{sec:normal}

We study how spectral geometry and generating vectors affect
weighted sampling for normal operators. Under a condition on the
angular distribution of the spectrum, we construct integer
sampling designs with counting bounds in terms of the nonempty
radial bands. The sampled frame operators approximate $S$ to
any prescribed relative accuracy.
We then construct a normal operator with two generating families
for which the optimal growth orders of $N_\mu(R)$ are $\log R$
and $R$, respectively, for both real and integer sampling.
For the family requiring linear growth, we estimate the infimum
of the upper sampling density under a prescribed bound on the
frame condition number and determine its asymptotic behavior
as this bound tends to infinity.

\subsection{Spectral geometry and generating vectors}
\label{sec:normal-geometry}

Let $(A,C)$ be an observable system with $A$ normal, and let
$\mu=\sum_{t\in T}w_t\delta_t$ be a sampling design such that
$S_\mu$ is bounded. Suppose first that $A$ has an orthonormal
eigenbasis $(e_j)_{j\in\mathbb N}$ with nonzero eigenvalues
$\lambda_j=e^{-s_j+i\theta_j}$, where $s_j>0$ and
$\theta_j\in(-\pi,\pi]$.
Set $B=C^*C$ and $B_{jk}=\langle Be_k,e_j\rangle$. Then
\[
 \langle S_\mu e_k,e_j\rangle
 =B_{jk}\sum_{t\in T}w_t e^{-(s_j+s_k)t}
                  e^{it(\theta_k-\theta_j)}
 \qquad j,k\in\mathbb N.
\]
If $A$ is positive, then $\theta_j=0$ for every $j\in\mathbb N$,
so $e^{it(\theta_k-\theta_j)}=1$ for all $t\ge0$ and
$j,k\in\mathbb N$.

For a general normal operator $A$ satisfying $AS=SA$,
the Stein identity in Lemma~\ref{lem:stein} gives
$B=S-A^*SA=S(I-|A|^2)$, and hence
\[
 S^{-1/2}S_\mu S^{-1/2}
 =\sum_{t\in T}w_t(I-|A|^2)|A|^{2t}.
\]

Let $E_A$ be the spectral measure of $A$. For $j\in\mathbb N$, set
$\Omega_j=\{z\in\sigma(A)\setminus\{0\}:
2^{-j-1}<-\log|z|\le2^{-j}\}$.
As in Section~\ref{sec:multiple}, define
$K_{|A|}(R)=\#\{j\in\mathbb N:2^j\le R,\ E_A(\Omega_j)\ne0\}$
for $R\ge2$.
Only radial bands with nonzero spectral projections are counted.

\begin{thm}\label{thm:normal-angular}
Let $(A,C)$ be an observable system with $A$ normal and frame
bounds $a,b$. Suppose that $M\in\mathbb N$ and $L_0>0$ satisfy the
following condition: for every $j\in\mathbb N$, the set
$\{z/|z|:z\in\Omega_j\}$ is covered by at most $M$ arcs of
$\mathbb T$, each of length at most $L_0 2^{-j}$.
For every $0<\varepsilon<1$, there is an integer sampling design
$\mu$ such that
\[
 (1-\varepsilon)S\le S_\mu\le(1+\varepsilon)S,\qquad
 N_\mu(R)\le L\bigl(K_{|A|}(R)+1\bigr)
 \qquad R\ge2.
\]
The constant $L$ depends only on $a,b,M,L_0,\varepsilon$.
\end{thm}

\begin{proof}
For each $j\in\mathbb N$, partition $\Omega_j$ into Borel sets
$\Omega_{j,r}$, $1\le r\le M_j$, where $M_j\in\{0,\ldots,M\}$ and
the directions of each set lie in one of the prescribed arcs.
Let $P_{j,r}=E_A(\Omega_{j,r})$, and let $\theta_{j,r}$ be the
central angle of that arc. Set
$\mathcal I=\{(j,r):j\in\mathbb N,\ 1\le r\le M_j,\ P_{j,r}\ne0\}$
and $P_0=E_A(\{z\in\mathbb C:|z|<e^{-1/2}\})$.
For $i=(j,r)\in\mathcal I$, write $P_i=P_{j,r}$,
$A_i=A|_{P_iH}$, $\theta_i=\theta_{j,r}$, and $q_i=j\log2$.
Lemma~\ref{lem:stein} gives
$P_0+\sum_{i\in\mathcal I}P_i=I$, with strong convergence.

Compressing the Stein identity to $P_iH$ gives
$\|CP_i\|^2\le2b\|e^{-i\theta_i}A_i-I\|
\le2b(1+L_0)2^{-j}$ for $i=(j,r)\in\mathcal I$.
Since $\|A_i^n\|\le e^{-n2^{-j-1}}$, there are constants
$C_0,c>0$ such that
$\|\sqrt n\,CA^nP_i\|\le C_0\sqrt b\,e^{-c|\log n-q_i|}$
for $n\in\mathbb N$ and $i\in\mathcal I$.
The same estimate holds for $P_0$ with $q_0=0$, since
$\|A^nP_0\|\le e^{-n/2}$.

For $k\in\mathbb N_0$, put
$\mathcal N_k=\{n\in\mathbb N:e^k\le n<e^{k+1}\}$, and define
$\nu=\sum_{n=1}^{\infty}n^{-1}\delta_n$.
Then $\nu(\mathcal N_k)\le2$ and
\[
 \int_{\mathbb N}\|\sqrt n\,CA^nx\|^2\,d\nu(n)
 =\sum_{n=1}^{\infty}\|CA^nx\|^2
 \qquad x\in H.
\]
Set $\mathcal Q=\{0\}\cup\{q_i:i\in\mathcal I\}$.
For $D,W\ge1$ and $K_0\in\mathbb N$, define
\[
 \mathcal I_k=\{i\in\mathcal I:|q_i-k|\le W\},
 \qquad
 P^{(k)}=\sum_{i\in\mathcal I_k}P_i
 \qquad k\in\mathbb N_0.
\]
There is a constant $C_1>0$ such that, for every positive measure
$\xi$ on $\mathbb N$ satisfying $\xi(\mathcal N_k)\le2$
for all $k\in\mathbb N_0$, we have
\[
 \sum_{\substack{k\ge0\\ \dist(k,\mathcal Q)>D}}
 \int_{\mathcal N_k}\|\sqrt n\,CA^nx\|^2\,d\xi(n)
 \le C_1^2b\,e^{-cD}\|x\|^2
\]
and
\[
 \sum_{k\ge K_0}
 \int_{\mathcal N_k}
 \|\sqrt n\,CA^n(I-P^{(k)})x\|^2\,d\xi(n)
 \le C_1^2b
 \bigl(e^{-cW/2}+e^{-cK_0/2}\bigr)^2\|x\|^2.
\]
Set $\mathcal J=\mathcal I\cup\{0\}$ and
$a_{ki}=e^{-c|k-q_i|/2}$. Since each center $j\log2$ occurs at most
$M$ times,
\[
 R_*:=\sup_{k\ge0}\sum_{i\in\mathcal J}a_{ki}
 \le1+\frac{2M}{1-2^{-c/2}},\qquad
 C_*:=\sup_{i\in\mathcal J}\sum_{k\ge0}a_{ki}
 \le\frac{2}{1-e^{-c/2}}.
\]
For $\rho\ge0$, we have
$e^{-c|k-q_i|}\le e^{-c\rho/2}a_{ki}$ whenever
$|k-q_i|\ge\rho$. Hence Cauchy--Schwarz gives
\[
 \begin{aligned}
 \sum_{k\ge0}
 \left(
 \sum_{\substack{i\in\mathcal J\\ |k-q_i|\ge\rho}}
 e^{-c|k-q_i|}\|P_ix\|
 \right)^2
 &\le
 e^{-c\rho}\sum_{k\ge0}
 \left(\sum_{i\in\mathcal J}a_{ki}\|P_ix\|\right)^2\\
 &\le
 R_*C_*e^{-c\rho}\sum_{i\in\mathcal J}\|P_ix\|^2\\
 &=R_*C_*e^{-c\rho}\|x\|^2.
 \end{aligned}
\]
For $n\in\mathcal N_k$, we have $|\log n-k|<1$ and
$\xi(\mathcal N_k)\le2$.
Apply the preceding bound with $\rho=D$ to the omitted intervals
and with $\rho=W$ to the omitted indices in $\mathcal I$.
For $P_0$, the identities $q_0=0$ and $|k-q_0|=k\ge K_0$
allow $\rho=K_0$, giving the factor $e^{-cK_0/2}$ after taking
square roots.

Keep time zero and all times in $\mathcal N_k$ for $0\le k<K_0$,
with unit weights. For later times, retain the intervals indexed by
$\mathcal K=\{k\in\mathbb N_0:k\ge K_0,\ \dist(k,\mathcal Q)\le D\}$.
There are constants $Q_1,Q_2$ such that
$\#\mathcal I_k\le Q_1$ for $k\in\mathbb N_0$ and
$\#\{k\in\mathbb N_0:i\in\mathcal I_k\}\le Q_2$ for $i\in\mathcal I$.

Fix $k\in\mathcal K$ and put $T=e^k$.
For $i=(j,r)\in\mathcal I_k$, choose a Borel argument $\vartheta_i$
on $\Omega_{j,r}$ satisfying
$|\vartheta_i(z)-\theta_i|\le L_0 2^{-j}/2$.
Spectral calculus applied to
$-\log|z|-i(\vartheta_i(z)-\theta_i)$ defines an operator $Z_i$
on $P_iH$ such that $\|Z_i\|\le(1+L_0)2^{-j}$ and
$A_i^n=e^{in\theta_i}e^{-nZ_i}$ for $n\in\mathbb N$.
Since $n$ is an integer, the latter identity remains valid even
when the arc crosses the negative real axis.
For $n\in\mathcal N_k$ and $i\in\mathcal I_k$, we have
$\|nZ_i\|\le(1+L_0)e^{W+1}$ and
$\sqrt n\,\|CP_i\|\le\sqrt{2b(1+L_0)e^{W+1}}$.
For $p\in\mathbb N_0$, define
$\phi_{i,l}(n)=\sqrt{n/T}\,(n/T)^l e^{in\theta_i}$ and
$V_{i,l}=\frac{\sqrt T}{l!}C(-TZ_i)^lP_i$
for $i\in\mathcal I_k$ and $0\le l\le p$.
Given $\eta>0$, Taylor's formula gives, for sufficiently large $p$,
\[
 \left\|
 \sqrt n\,CA^nP_ix-\sum_{l=0}^p\phi_{i,l}(n)V_{i,l}x
 \right\|
 \le\eta\sqrt b\,\|P_ix\|.
\]
The estimate is uniform in $k\in\mathcal K$, $n\in\mathcal N_k$,
$i\in\mathcal I_k$, and the angles $\theta_i$.

On each retained interval, set
$F_k(n)=\sum_{i\in\mathcal I_k}\sum_{l=0}^p\phi_{i,l}(n)V_{i,l}$.
Cauchy--Schwarz and the bounds on $\xi(\mathcal N_k)$,
$Q_1$, and $Q_2$ give
\[
 \sum_{k\in\mathcal K}\int_{\mathcal N_k}
 \|\sqrt n\,CA^nP^{(k)}x-F_k(n)x\|^2\,d\xi(n)
 \le
 2Q_1Q_2\eta^2b\sum_{i\in\mathcal I}\|P_ix\|^2
 \le2Q_1Q_2\eta^2b\|x\|^2.
\]

For each $k\in\mathcal K$, let $\Phi(n)$ be the vector consisting
of the constant coordinate one and
$(\phi_{i,l}(n))_{i\in\mathcal I_k,\,0\le l\le p}$.
Its dimension satisfies
$d_k=1+(p+1)\#\mathcal I_k\le d_0:=1+(p+1)Q_1$.
Apply \cite[Corollary~1]{BayerTeichmann2006} in the real vector
space of Hermitian $d_k\times d_k$ matrices.
There are coefficients $\nu_{k,n}\ge0$, $n\in\mathcal N_k$, with
$\#\{n\in\mathcal N_k:\nu_{k,n}>0\}\le d_k^2$, such that
$\sum_{n\in\mathcal N_k}\nu_{k,n}\Phi(n)\Phi(n)^*
=\sum_{n\in\mathcal N_k}n^{-1}\Phi(n)\Phi(n)^*$.
The constant coordinate gives
$\sum_{n\in\mathcal N_k}\nu_{k,n}=\nu(\mathcal N_k)\le2$.
Expanding the squared norms yields
$\sum_{n\in\mathcal N_k}\nu_{k,n}\|F_k(n)x\|^2
=\sum_{n\in\mathcal N_k}n^{-1}\|F_k(n)x\|^2$
for $x\in H$.
This identity holds without any restriction on $\dim Y$.

Define
\[
 \nu'
 =\sum_{k=0}^{K_0-1}\sum_{n\in\mathcal N_k}n^{-1}\delta_n
 +\sum_{k\in\mathcal K}\sum_{n\in\mathcal N_k}\nu_{k,n}\delta_n,
 \qquad
 \mu=\delta_0+\sum_{n=1}^{\infty}n\nu'(\{n\})\delta_n.
\]
Both $\nu$ and $\nu'$ have mass at most two on each $\mathcal N_k$.
Thus the truncation and Taylor estimates apply to both measures,
while the approximate energies agree on every retained interval.
These estimates also show that the resulting analysis operator is bounded.

Choose $\delta>0$ so that
$\delta(2\sqrt b+\delta)\le\varepsilon a$.
Taking $D$, then $W,K_0$ sufficiently large, followed by $\eta$
sufficiently small and $p$ sufficiently large, the triangle
inequality gives
$\left|\langle S_\mu x,x\rangle^{1/2}
-\langle Sx,x\rangle^{1/2}\right|\le\delta\|x\|$
for $x\in H$.
Since $aI\le S\le bI$, it follows that
$|\langle(S_\mu-S)x,x\rangle|
\le\delta(2\sqrt b+\delta)\|x\|^2
\le\varepsilon\langle Sx,x\rangle$.
Hence $(1-\varepsilon)S\le S_\mu\le(1+\varepsilon)S$.

Each later retained interval contains at most $d_0^2$ sampling times.
If it contains a sampling time $n\le R$, then $k\le\log R$ and
$|k-q|\le D$ for some $q\in\mathcal Q\cap[0,\log R+D]$.
Each center accounts for at most $2D+2$ interval indices, and
$\#\bigl(\mathcal Q\cap[0,\log R+D]\bigr)
\le K_{|A|}(R)+\left\lceil\frac D{\log2}\right\rceil+2$.
Including time zero and the initial positive times gives
\[
 N_\mu(R)\le\lceil e^{K_0}\rceil
 +d_0^2(2D+2)
 \left(K_{|A|}(R)+\left\lceil\frac D{\log2}\right\rceil+2\right)
 \le L\bigl(K_{|A|}(R)+1\bigr)
 \qquad R\ge2.
\]
The parameters can be chosen in terms of $a,b,M,L_0,\varepsilon$
only, which gives the stated dependence of $L$.
For fixed $A$ and spectral partition, the vectors $\Phi(n)$ and
the coefficients $\nu_{k,n}$ can be chosen independently of $C$.
\end{proof}

For positive $A$, the angular condition holds with $M=1$, so
Theorem~\ref{thm:normal-angular} gives sampling designs
approximating $S$ to arbitrary relative accuracy.
However, a fixed observable system may admit a sampling design
with $N_\mu(R)=O(\log R)$ even when the angular condition fails;
see Proposition~\ref{prop:ring-observation}.

Let $\Omega$ be a nonempty Borel subset of the open unit disk
$\mathbb D=\{z\in\mathbb C:|z|<1\}$.
For $j\in\mathbb N$, set
$\Omega_j=\{\lambda\in\Omega\setminus\{0\}:
2^{-j-1}<-\log|\lambda|\le2^{-j}\}$.
Define
$K_\Omega(R)=\#\{j\in\mathbb N:2^j\le R,\ \Omega_j\ne\varnothing\}$
for $R\ge2$.

\begin{cor}
\label{cor:uniform-spectral-designs}
Fix $\kappa\ge1$, $M\in\mathbb N$, and $L_0>0$.
Suppose that, for every $j\in\mathbb N$, the set
$\{\lambda/|\lambda|:\lambda\in\Omega_j\}$ is covered by at most
$M$ arcs of $\mathbb T$, each of length at most $L_0 2^{-j}$.
Consider all observable systems $(A,C)$ with $A$ normal,
$E_A(\mathbb C\setminus\Omega)=0$, and frame operator $S$
satisfying $aI\le S\le\kappa aI$ for some $a>0$.
Then the following hold.
\begin{enumerate}[(i)]
\item For every $0<\varepsilon<1$, there is a single integer sampling
design $\mu=\mu_{\Omega,\kappa,\varepsilon}$ such that
\[
 (1-\varepsilon)S\le S_\mu\le(1+\varepsilon)S,\qquad
 N_\mu(R)\le L\bigl(K_\Omega(R)+1\bigr)
 \qquad R\ge2
\]
for every system in this class.
The constant $L$ depends only on $\kappa,M,L_0,\varepsilon$.
\item If $K_\Omega$ is unbounded, every real or integer sampling
design $\nu$ that is admissible for every system in this class
satisfies $N_\nu(R)\ge c_\nu K_\Omega(R)-C_\nu$ for every $R\ge2$,
with constants $c_\nu>0$ and $C_\nu\ge0$.
\end{enumerate}
\end{cor}

\begin{proof}
For (i), first assume $I\le S\le\kappa I$.
Set $\mathcal J_\Omega=\{j\in\mathbb N:
\Omega_j\ne\varnothing\}$.
For each $j\in\mathcal J_\Omega$, fix a Borel partition
$\Omega_j=\bigsqcup_{r=1}^{m_j}\Omega_{j,r}$ with $m_j\le M$,
such that the directions of each $\Omega_{j,r}$ lie in one of
the prescribed arcs, and fix the central angle of that arc.
Set $P_{j,r}=E_A(\Omega_{j,r})$ and
$P_0=E_A(\{z:|z|<e^{-1/2}\})$.
Since $E_A(\mathbb C\setminus\Omega)=0$, we have
$P_0+\sum_{j\in\mathcal J_\Omega}\sum_{r=1}^{m_j}P_{j,r}=I$
in the strong operator topology.
Use the construction in the proof of
Theorem~\ref{thm:normal-angular} with these projections and all
centers $j\log2$ for $j\in\mathcal J_\Omega$.

Keep the scalar functions $\phi_{i,l}$ in that construction even
when $P_i=0$; in that case, $V_{i,l}=0$.
Thus the vectors $\Phi(n)$ depend only on the fixed partition,
the integer times, and the approximation parameters.
Fix coefficients $\nu_{k,n}$ as in that construction such that
$\sum_{n\in\mathcal N_k}\nu_{k,n}\Phi(n)\Phi(n)^*
=\sum_{n\in\mathcal N_k}n^{-1}\Phi(n)\Phi(n)^*$.
The constant coordinate gives
$\sum_{n\in\mathcal N_k}\nu_{k,n}
=\sum_{n\in\mathcal N_k}n^{-1}\le2$.
The estimates for omitted spectral subspaces and Taylor remainders
are uniform over the class, since $m_j\le M$ and the same angular
width bound holds on each set.
Counting the intervals near all centers $j\log2$ with
$j\in\mathcal J_\Omega$ gives
$N_\mu(R)\le L(K_\Omega(R)+1)$ for $R\ge2$.

For general $a>0$, set $\widetilde C=a^{-1/2}C$.
The corresponding full and sampled frame operators satisfy
$\widetilde S=a^{-1}S$, $\widetilde S_\mu=a^{-1}S_\mu$, and
$I\le\widetilde S\le\kappa I$.
Hence the same times and weights give the required relative
inequalities for $C$.
If $K_\Omega$ is bounded, the counting estimate gives
$\#\operatorname{supp}\mu
=\lim_{R\to\infty}N_\mu(R)<\infty$.

For (ii), $\mathcal J_\Omega$ is infinite.
Choose $\lambda_j\in\Omega_j$ for every
$j\in\mathcal J_\Omega$.
Fix an orthonormal basis $(e_j)_{j\in\mathcal J_\Omega}$ of $H$
and define $A_\Omega e_j=\lambda_j e_j$ for
$j\in\mathcal J_\Omega$ and
$C_\Omega=(I-A_\Omega^*A_\Omega)^{1/2}$.
The operator $A_\Omega$ is normal,
$E_{A_\Omega}(\mathbb C\setminus\Omega)=0$, and
$S e_j=(1-|\lambda_j|^2)
\left(\sum_{n=0}^{\infty}|\lambda_j|^{2n}\right)e_j=e_j$
for $j\in\mathcal J_\Omega$.
Thus $S=I$, and this system belongs to the class for every
$\kappa\ge1$.

Let $\nu=\sum_{t\in T}w_t\delta_t$ be admissible for every system
in the class.
For $(A_\Omega,C_\Omega)$, its sampled frame operator satisfies
\[
 S_\nu e_j
 =(1-|\lambda_j|^2)
   \left(\sum_{t\in T}w_t|\lambda_j|^{2t}\right)e_j
 \qquad j\in\mathcal J_\Omega.
\]
The positive system $(|A_\Omega|,C_\Omega)$ has the same full
and sampled frame operators, and
$K_{|A_\Omega|}(R)=K_\Omega(R)$ for $R\ge2$.
The lower bound in Theorem~\ref{thm:band-count} proves (ii).
\end{proof}

For a prescribed $\Gamma>\kappa$, take
$\varepsilon=(\Gamma-\kappa)/(\Gamma+\kappa)\in(0,1)$.
Corollary~\ref{cor:uniform-spectral-designs} gives a common integer
sampling design $\mu$ satisfying
$\kappa(S_\mu)\le\kappa(1+\varepsilon)/(1-\varepsilon)=\Gamma$
and $N_\mu(R)\le L(K_\Omega(R)+1)$ for every $R\ge2$.

\subsection{Two generating families for the same operator}
\label{sec:normal-model}

Set $z=e^{-2}$ and $d=1-z$, and let
$\mathcal H=\bigoplus_{j=1}^{\infty}\mathbb C^{2^j}$.
For $m=2^j$ with $j\in\mathbb N$, set $\omega_m=e^{2\pi i/m}$
and define
\begin{equation}\label{eq:ring-operator}
 A_m=e^{-1/m}\operatorname{diag}(1,\omega_m,\ldots,\omega_m^{m-1}),
 \qquad A=\bigoplus_{j=1}^{\infty}A_{2^j}.
\end{equation}

For $j\in\mathbb N$, let $f_j\in\mathcal H$ have all coordinates
equal to $\sqrt{d/2^j}$ on the block $\mathbb C^{2^j}$ and zero
coordinates on the other blocks.
Write $x=(x_j)_{j\in\mathbb N}\in\mathcal H$, where
$x_j=(x_{j,a})_{a=0}^{2^j-1}$.
The observation operator $C:\mathcal H\to\ell^2(\mathbb N)$ is given by
$(Cx)_j=\langle x,f_j\rangle
=\sqrt{d/2^j}\sum_{a=0}^{2^j-1}x_{j,a}$ for $j\in\mathbb N$,
and $\|C\|=\sqrt d$.
The corresponding map $C_m:\mathbb C^m\to\mathbb C$ on each block
is $C_my=\sqrt{d/m}\sum_{a=0}^{m-1}y_a$ for
$y=(y_a)_{a=0}^{m-1}\in\mathbb C^m$.
For $t\ge0$, we have $\|C_mA_m^t\|^2=d e^{-2t/m}$.

\begin{lem}
\label{lem:ring-residues}
Let $\mu=\sum_{n=0}^{\infty}w_n\delta_n$ be an integer sampling
design.
For $m=2^j$, $j\in\mathbb N$, define $g_{m,k}\in[0,\infty]$ by
\[
 g_{m,k}=d\sum_{\ell=0}^{\infty}w_{k+\ell m}
 e^{-2(k+\ell m)/m}
 \qquad k\in\{0,\ldots,m-1\}.
\]
For $0<\alpha\le\beta<\infty$, the sampled family has frame bounds
$\alpha,\beta$ if and only if $\alpha\le g_{m,k}\le\beta$ for
every such $m,k$.
If $S_\mu$ is bounded, its block on $\mathbb C^m$ has eigenvalues
$(g_{m,k})_{k=0}^{m-1}$.
\end{lem}

\begin{proof}
Fix $m=2^j$ with $j\in\mathbb N$.
The vectors $u_{m,k}=m^{-1/2}(\omega_m^{-ak})_{a=0}^{m-1}$,
$k\in\{0,\ldots,m-1\}$, form an orthonormal Fourier basis of
$\mathbb C^m$.
For $n\in\mathbb N_0$ with $n\equiv k\pmod m$ and
$y\in\mathbb C^m$, we have
$C_mA_m^ny=\sqrt d\,e^{-n/m}\langle y,u_{m,k}\rangle$.
In particular,
$\sum_{n=0}^{\infty}w_n|C_mA_m^nu_{m,k}|^2=g_{m,k}$.
The frame inequalities applied to these unit vectors therefore
imply $\alpha\le g_{m,k}\le\beta$.

Conversely, suppose that $\alpha\le g_{m,k}\le\beta$ for every
such $m,k$.
For $x=(x_j)_{j\in\mathbb N}\in\mathcal H$, we may rearrange
the nonnegative sums to obtain
\[
 \sum_{n=0}^{\infty}w_n\|CA^nx\|^2
 =\sum_{j=1}^{\infty}\sum_{n=0}^{\infty}
   w_n|C_{2^j}A_{2^j}^nx_j|^2
 =\sum_{j=1}^{\infty}\sum_{k=0}^{2^j-1}
   g_{2^j,k}|\langle x_j,u_{2^j,k}\rangle|^2.
\]
Parseval's identity yields the required frame bounds.
If $S_\mu$ is bounded, the same computation shows that its block
on $\mathbb C^m$ is
$\sum_{k=0}^{m-1}g_{m,k}u_{m,k}u_{m,k}^*$,
which proves the eigenvalue assertion.
\end{proof}

For unit weights $w_n=1$ for all $n\in\mathbb N_0$,
Lemma~\ref{lem:ring-residues} gives $g_{m,k}=e^{-2k/m}$.
Since $g_{m,0}=1$ and $g_{2^j,2^j-1}\to z$ as $j\to\infty$,
the optimal frame bounds are $z$ and $1$.

\begin{lem}
\label{lem:ring-real-bound}
Let $\mu$ be an admissible real or integer sampling design with frame
bounds $0<\alpha\le\beta<\infty$, and put $\Gamma=\beta/\alpha$.
For every $L>0$,
\[
 \limsup_{R\to\infty}\frac{N_\mu(R)}R
 \ge\frac{1-2e^2\Gamma(L+1)e^{-2L}}L.
\]
Moreover, $N_\mu(R)\ge c_\mu R$ for some $c_\mu>0$ and all
sufficiently large $R$.
\end{lem}

\begin{proof}
Set $W(R)=\mu([0,R])$ for $R\ge0$.
For $m=2^j$, $j\in\mathbb N$, let $S_{\mu,m}$ denote the block
of $S_\mu$ on $\mathbb C^m$.
The upper frame inequality gives
$d\int_{[0,\infty)}e^{-2t/m}\,d\mu(t)
=\operatorname{tr}S_{\mu,m}\le\beta m$.
For $R\ge2$, choose $m=2^j$ with $R\le m<2R$;
for $0\le R<2$, use $m=2$.
Then
\[
 W(R)\le\frac{2e^2\beta}{d}
 \begin{cases}
 R,&R\ge2,\\
 1,&0\le R<2.
 \end{cases}
\]
Hence $W(R)\le K\beta(1+R)$ for every $R\ge0$, where $K=2e^2/d$.

Fix $L>0$ and $m=2^j$, $j\in\mathbb N$.
Integration by parts gives
\[
 \begin{aligned}
 d\int_{(Lm,\infty)}e^{-2t/m}\,d\mu(t)
 &=-d e^{-2L}W(Lm)
   +\frac{2d}{m}\int_{Lm}^{\infty}e^{-2t/m}W(t)\,dt\\
 &\le dK\beta e^{-2L}(1+Lm+m/2)\\
 &\le2e^2\beta m(L+1)e^{-2L},
 \end{aligned}
\]
where the last inequality uses $m\ge2$.

Let $B_m$ and $D_m$ be the contributions to $S_{\mu,m}$ from
$[0,Lm]$ and $(Lm,\infty)$, respectively, and let $P_m$ be the
orthogonal projection onto $\ker B_m$.
Then $\operatorname{rank}B_m\le N_\mu(Lm)$ and
$P_mD_mP_m=P_mS_{\mu,m}P_m\ge\alpha P_m$, so
\[
 \alpha(m-N_\mu(Lm))
 \le\alpha\operatorname{rank}P_m
 \le\operatorname{tr}(P_mD_mP_m)
 \le\operatorname{tr}D_m
 =d\int_{(Lm,\infty)}e^{-2t/m}\,d\mu(t).
\]
Consequently,
$N_\mu(Lm)\ge m\bigl(1-2e^2\Gamma(L+1)e^{-2L}\bigr)$.
Divide by $Lm$ and let $j\to\infty$, with $m=2^j$, to obtain
the density bound.

Choose $L>0$ so that $q=1-2e^2\Gamma(L+1)e^{-2L}>0$.
For $R\ge2L$, choose $m=2^j$, $j\in\mathbb N$, with
$R/(2L)<m\le R/L$.
Then $N_\mu(R)\ge N_\mu(Lm)\ge qm\ge qR/(2L)$,
which proves the last assertion with $c_\mu=q/(2L)$.
\end{proof}

\begin{prop}
\label{prop:ring-observation}
For the operator $A$ defined in \eqref{eq:ring-operator}, set
$C_0=(I-A^*A)^{1/2}:\mathcal H\to\mathcal H$.
Both $(A,C)$ and $(A,C_0)$ are observable.
\begin{enumerate}[(i)]
\item Every admissible real or integer design for $(A,C)$ satisfies
$N_\mu(R)\ge c_\mu R$ for some $c_\mu>0$ and all sufficiently
large $R$. There is an admissible integer design with
$N_\mu(R)=\lfloor R\rfloor+1$ for $R\ge0$.
\item Every admissible real or integer design for $(A,C_0)$ satisfies
$N_\mu(R)\ge c_\mu\log R-C_\mu$ for $R\ge2$, with constants
$c_\mu>0$ and $C_\mu\ge0$.
There is an admissible integer design with
$N_\mu(R)=\lfloor\log_2R\rfloor+1$ for $R\ge1$.
\end{enumerate}
\end{prop}

\begin{proof}
For (i), sampling at all times in $\mathbb N_0$ with unit weights
gives a frame with optimal bounds $z$ and $1$, by Lemma~\ref{lem:ring-residues},
and the stated counting function. The lower bound follows from
Lemma~\ref{lem:ring-real-bound}.

For (ii), $A^nx\to0$ as $n\to\infty$ for every $x\in\mathcal H$.
Hence $\sum_{n=0}^{N}(A^n)^*C_0^2A^n
=I-(A^{N+1})^*A^{N+1}\longrightarrow I$
in the strong operator topology, so $S=I$ for $(A,C_0)$.
For every $t\ge0$,
$(A^t)^*C_0^2A^t=(I-|A|^2)|A|^{2t}$.
Thus $(A,C_0)$ and $(|A|,C_0)$ have the same sampled frame operators.
Since $K_{|A|}(R)=\lfloor\log_2R\rfloor$ for $R\ge2$,
Theorem~\ref{thm:band-count} gives the logarithmic lower bound.

Define the integer design
$\mu_0=\sum_{k=0}^{\infty}2^k\delta_{2^k}$.
For $(A,C_0)$, its sampled frame operator satisfies
\[
 \left.S_{\mu_0}\right|_{\mathbb C^m}
 =m(1-e^{-2/m})
  \left(\sum_{q=-j}^{\infty}2^qe^{-2\cdot2^q}\right)
  I_{\mathbb C^m}
 \qquad m=2^j,\quad j\in\mathbb N.
\]
The series $\sum_{q\in\mathbb Z}2^qe^{-2\cdot2^q}$ converges.
Since $2(1-e^{-1})\le m(1-e^{-2/m})\le2$ and the term with
$q=0$ equals $e^{-2}$, we obtain
\[
 2e^{-2}(1-e^{-1})I
 \le S_{\mu_0}
 \le 2\left(\sum_{q\in\mathbb Z}2^qe^{-2\cdot2^q}\right)I.
\]
Hence $\mu_0$ is admissible, and
$N_{\mu_0}(R)=\lfloor\log_2R\rfloor+1$ for $R\ge1$.
\end{proof}

Both observation operators have infinite rank. Moreover,
$\sum_{j=1}^{\infty}2^j(1-e^{-2/2^j})=\infty$.
The necessary condition in
\cite[Section~2]{CabrelliMolterSuarez2021} therefore shows that
no finite family of generating vectors gives an orbit frame
for this operator.
\subsection{Condition numbers and sampling density}
\label{sec:ring-conditioning}

We fix the pair $(A,C)$ from Section~\ref{sec:normal-model}.
For a locally finite set $T\subset[0,\infty)$, define its upper
density by
$\overline d(T)=\limsup_{R\to\infty}\#(T\cap[0,R])/R$.
For $\Gamma\ge1$, define
\[
 \rho_{\mathbb Z}(\Gamma)
 =\inf\left\{\overline d(T):
       T\subset\mathbb N_0,\ \kappa(S_\mu)\le\Gamma\right\},
\]
where the infimum is over all admissible sampling designs
$\mu=\sum_{t\in T}w_t\delta_t$ with $w_t>0$ for every $t\in T$.
Define $\rho_{\mathbb R}(\Gamma)$ in the same way with
$T\subset[0,\infty)$, and use the convention
$\inf\varnothing=+\infty$.
We call $\Gamma$ feasible for integer sampling if an admissible
integer design satisfies $\kappa(S_\mu)\le\Gamma$.

The following constants will be used in the condition number
estimates. For $J\in\{0,1\}$, define
\[
 B_J=e\sum_{k=0}^{\infty}z^{(2J+1)2^k-J},\qquad
 c_J=1-\sum_{k=0}^{\infty}
 \frac{z^{(2J+1)2^k+1}(1-z^{2^{k+1}-1})}
      {d(1-z^{2^{k+1}})}.
\]
Both series converge absolutely. Since $z<1/4$, we have
$0\le1-c_1\le1-c_0\le d^{-1}\sum_{h=1}^{\infty}z^{h+1}
=z^2/d^2<1/9$ and hence $c_1\ge c_0>8/9$.
Set $\Gamma_-=B_0+c_0d/z$ and $\Gamma_P=B_1+c_1de^3$.
Numerically, $\Gamma_-\approx6.6698199$ and
$\Gamma_P\approx17.4112141$.
Since $d>3/4$, the lower bound on $c_1$ gives
$\Gamma_P>2e^3/3>e^2$.

\begin{lem}
\label{lem:ring-scales}
Let $\mu=\sum_{n=0}^{\infty}w_n\delta_n$ be an integer design
satisfying $\Gamma^{-1}I\le S_\mu\le I$ for some $\Gamma\ge1$.
Set $v_n=dw_n$ for $n\in\mathbb N_0$ and
$M=\sup_{n\in\mathbb N_0}v_n$.
Then $M\le1$ and $c_JM+B_J/\Gamma\le1$ for $J\in\{0,1\}$.
\end{lem}

\begin{proof}
Fix $n\in\mathbb N_0$ and let $m=2^j>n$, where $j\in\mathbb N$.
The upper frame bound and Lemma~\ref{lem:ring-residues} give
$v_n+\sum_{\ell=1}^{\infty}v_{n+\ell m}z^\ell\le e^{2n/m}$.
Dropping the nonnegative sum and letting $j\to\infty$ gives
$v_n\le1$, so $M\le1$.

Fix $J\in\{0,1\}$ and let $h=2^k$, where $k\in\mathbb N_0$.
The lower frame bound and Lemma~\ref{lem:ring-residues} give
$g_{2hm,n+hm}\ge\Gamma^{-1}$, and hence
$\sum_{p=0}^{\infty}v_{n+hm(2p+1)}z^p
\ge e^{1+n/(hm)}/\Gamma\ge e/\Gamma$.
For a sequence $(a_p)_{p\in\mathbb N_0}$ with $0\le a_p\le M$,
the coefficient
$z^{h(2p+1)}-z^{h+(2h-1)J+p}$ is nonnegative for $p\le J$
and nonpositive for $p>J$. Hence
\[
 \sum_{p=0}^{\infty}a_p
 \bigl(z^{h(2p+1)}-z^{h+(2h-1)J+p}\bigr)
 \ge
 -M\sum_{p=J+1}^{\infty}
 \bigl(z^{h+(2h-1)J+p}-z^{h(2p+1)}\bigr)
 =
 -M\frac{z^{h+2hJ+1}(1-z^{2h-1})}
          {d(1-z^{2h})}.
\]
Apply this inequality with $a_p=v_{n+hm(2p+1)}$ for
$p\in\mathbb N_0$, and sum over $h=2^k$, $k\in\mathbb N_0$.
Every $\ell\in\mathbb N$ has a unique representation
$\ell=2^k(2p+1)$ with $k,p\in\mathbb N_0$.
All series converge absolutely, so the definitions of $B_J,c_J$
and the upper bound give
$e^{2n/m}\ge v_n+B_J/\Gamma-M(1-c_J)$.
Letting $j\to\infty$ with $n$ fixed and then taking the supremum
over $n\in\mathbb N_0$ yields $c_JM+B_J/\Gamma\le1$.
\end{proof}

\begin{prop}
\label{prop:ring-threshold}
The infimum $\Gamma_{\mathrm{crit}}$ of the condition numbers of
admissible integer designs is attained and satisfies
$\Gamma_-\le\Gamma_{\mathrm{crit}}\le e^2$.
\end{prop}

\begin{proof}
Let $\mu=\sum_{n=0}^{\infty}w_n\delta_n$ be an admissible
integer design and set $\Gamma=\kappa(S_\mu)$.
Normalize its weights so that $\Gamma^{-1}I\le S_\mu\le I$.
Set $M=\sup_{n\in\mathbb N_0}dw_n$.
For $m=2^j$, $j\in\mathbb N$, Lemma~\ref{lem:ring-residues}
at $k=m-1$ gives
$\Gamma^{-1}\le g_{m,m-1}\le M e^{-2(m-1)/m}/d$.
Letting $j\to\infty$ yields $M\ge d/(\Gamma z)$.
Since $c_0>0$, Lemma~\ref{lem:ring-scales} with $J=0$ gives
$\Gamma\ge\Gamma c_0M+B_0\ge B_0+c_0d/z=\Gamma_-$.
Sampling at all times in $\mathbb N_0$ with unit weights gives
optimal frame bounds $z$ and $1$, so
$\Gamma_{\mathrm{crit}}\le1/z=e^2$.

The pair $(A,C)$ is observable, and $A$ is normal with $\|A\|=1$.
Every integer design satisfies
$N_\mu(R)\le\lfloor R\rfloor+1$ for $R\ge0$.
Theorem~\ref{thm:counting-condition-attainment}, applied with
$T_0=\mathbb N_0$ and $F(R)=\lfloor R\rfloor+1$ for $R\ge0$,
therefore gives attainment over all admissible integer designs.
\end{proof}

\begin{thm}
\label{thm:ring-plateau}
For $1\le\Gamma<\Gamma_{\mathrm{crit}}$,
$\rho_{\mathbb Z}(\Gamma)=+\infty$.
Moreover, the following hold.
\begin{enumerate}[(i)]
\item For every $\Gamma\ge\Gamma_{\mathrm{crit}}$,
$\rho_{\mathbb Z}(\Gamma)
\ge 2/\log((\Gamma-B_1)/(c_1d))$.
\item If $\Gamma_{\mathrm{crit}}\le\Gamma\le\Gamma_P$, then
$\rho_{\mathbb Z}(\Gamma)=1$, and the infimum is attained.
\end{enumerate}
\end{thm}

\begin{proof}
For (i), let $\mu=\sum_{n=0}^{\infty}w_n\delta_n$ be an
admissible integer design with support $T$ and
$\kappa(S_\mu)\le\Gamma$.
Normalize its weights so that $\Gamma^{-1}I\le S_\mu\le I$.
Set $M=\sup_{n\in\mathbb N_0}dw_n$.
For $m=2^j$, $j\in\mathbb N$, and $k\in\{0,\ldots,m-1\}$,
let $\nu_{m,k}$ be the least element of $T$ congruent to $k$
modulo $m$. It exists by the lower frame bound.
Lemmas~\ref{lem:ring-residues} and \ref{lem:ring-scales} give
$e^{2\nu_{m,k}/m}\le\Gamma M/d\le(\Gamma-B_1)/(c_1d)$.
Since $\nu_{m,m-1}\ge m-1$, letting $j\to\infty$ gives
$(\Gamma-B_1)/(c_1d)\ge e^2$.
Thus $q=\tfrac12\log((\Gamma-B_1)/(c_1d))\ge1$ and
$\nu_{m,k}\le qm$ for every such $m,k$.
The $m$ residue classes therefore give $m$ distinct sampling times
in $[0,qm]$, so $N_\mu(qm)\ge m$.
Letting $j\to\infty$ yields $\overline d(T)\ge1/q$.
Taking the infimum over $\mu$ proves (i).

For (ii), suppose
$\Gamma_{\mathrm{crit}}\le\Gamma\le\Gamma_P$, so $q\le3/2$.
If $n\in\mathbb N\setminus T$, choose $m=2^j$, $j\in\mathbb N$,
such that $m/2\le n<m$.
Then $3m/2\le n+m\le\nu_{m,n}\le qm\le3m/2$.
Thus $n=m/2$, and hence
$\mathbb N\setminus T\subseteq\{2^r:r\in\mathbb N_0\}$.
It follows that $N_\mu(R)=R+O(\log R)$ as $R\to\infty$,
and therefore $\overline d(T)=1$.
The design attaining $\Gamma_{\mathrm{crit}}$ in
Proposition~\ref{prop:ring-threshold} satisfies the condition number
bound for every $\Gamma\in[\Gamma_{\mathrm{crit}},\Gamma_P]$.
It therefore attains the infimum in (ii).
\end{proof}

\begin{remark}
Let $\mu=\sum_{n=0}^{\infty}w_n\delta_n$ be an integer design
satisfying $\Gamma^{-1}I\le S_\mu\le I$ for
$1\le\Gamma\le\Gamma_P$.
Then
\[
 \frac{z(\Gamma_P-\Gamma)}{\Gamma d^2c_1}
 \le w_n\le\frac{\Gamma-B_1}{\Gamma dc_1}
 \qquad n\in\mathbb N.
\]
Put $M=\sup_{r\in\mathbb N_0}dw_r$.
For $n\in\mathbb N$, choose $m=2^j$, $j\in\mathbb N$, with
$m/2\le n<m$.
Lemma~\ref{lem:ring-residues} gives
$e/\Gamma\le e^{2n/m}/\Gamma\le dw_n+Mz/d$.
Lemma~\ref{lem:ring-scales} gives
$M\le(\Gamma-B_1)/(\Gamma c_1)$.
Substituting this estimate and using
$\Gamma_P=B_1+c_1de/z$ yields the stated inequalities.
\end{remark}

\begin{prop}
For the pair $(A,C)$ from Section~\ref{sec:normal-model},
the following hold.
\begin{enumerate}[(i)]
\item For every odd $L\in\mathbb N$, sampling on $L\mathbb N_0$
with constant positive weights gives an admissible design with upper
density $1/L$ and condition number $e^{2L}$.
For even $L\in\mathbb N$, the sampled family is incomplete.
\item We have $\rho_{\mathbb Z}(\Gamma)\sim 2/\log\Gamma$ and
$\rho_{\mathbb R}(\Gamma)\sim 2/\log\Gamma$ as $\Gamma\to\infty$.
\end{enumerate}
\end{prop}

\begin{proof}
For (i), suppose $L$ is odd and let $m=2^j$, $j\in\mathbb N$.
Since $\gcd(L,m)=1$, multiplication by $L$ permutes the residue
classes modulo $m$.
For $a\in\{0,\ldots,m-1\}$, let $k_a$ be the residue of $La$
modulo $m$.
With unit weights, Lemma~\ref{lem:ring-residues} gives
\[
 g_{m,k_a}
 =d\sum_{\ell=0}^{\infty}e^{-2L(a+\ell m)/m}
 =\frac{de^{-2La/m}}{1-e^{-2L}}.
\]
Taking the infimum and supremum over all blocks gives
$\inf_{m,k}g_{m,k}=de^{-2L}/(1-e^{-2L})$ and
$\sup_{m,k}g_{m,k}=d/(1-e^{-2L})$.
Thus the design is admissible with
$\kappa(S_\mu)=e^{2L}$.
Changing the common positive weight does not change this ratio.
The counting formula $N_\mu(R)=\lfloor R/L\rfloor+1$ for $R\ge0$
gives upper density $1/L$.
If $L$ is even, then $C_2A_2^{Ln}u_{2,1}=0$ for every
$n\in\mathbb N_0$, where $u_{2,1}=2^{-1/2}(1,-1)$.
Hence the sampled family is incomplete.

For (ii), let $\Gamma\ge e^2$ and choose the largest odd integer
$L\le\tfrac12\log\Gamma$.
Then $0\le\tfrac12\log\Gamma-L<2$, so part (i) gives
\[
 \rho_{\mathbb R}(\Gamma)
 \le\rho_{\mathbb Z}(\Gamma)
 \le\frac1L
 \sim\frac{2}{\log\Gamma}
 \qquad \Gamma\to\infty.
\]
Set $L_\Gamma=\tfrac12\log\Gamma+\log\log\Gamma$.
Then $\Gamma e^{-2L_\Gamma}=(\log\Gamma)^{-2}$ and
$L_\Gamma\sim\tfrac12\log\Gamma$ as $\Gamma\to\infty$.
Applying Lemma~\ref{lem:ring-real-bound} with $L=L_\Gamma$
and taking the infimum over all admissible real designs with
condition number at most $\Gamma$ gives
\[
 \rho_{\mathbb R}(\Gamma)
 \ge\frac{1-2e^2(L_\Gamma+1)/(\log\Gamma)^2}{L_\Gamma}
 \sim\frac{2}{\log\Gamma}
 \qquad \Gamma\to\infty.
\]
Together with the upper bound, this proves (ii).
\end{proof}

Letting $L\to\infty$ through odd positive integers shows that,
without a condition number bound, the infimum of the upper sampling
density is zero for both real and integer designs.
This infimum is not attained: by Lemma~\ref{lem:ring-real-bound},
every admissible design has strictly positive upper density.

We next consider an explicit sampling design with nonperiodic support.
Let $(\tau(n))_{n\in\mathbb N_0}$ be the Thue--Morse sequence,
where $\tau(n)\in\{0,1\}$ is the parity of the number of ones in
the binary expansion of $n$ \cite{AlloucheShallit2003}.
Set $T=\{n\in\mathbb N_0:\tau(n)=0\}$,
$P(z)=\prod_{q=0}^{\infty}(1-z^{2^q})$, and
$F_\pm=(d^{-1}\pm P(z))/2$.

\begin{prop}
For the Thue--Morse sampling set $T$, the following hold.
\begin{enumerate}[(i)]
\item Sampling on $T$ with unit weights gives an admissible design
with upper density $1/2$ and condition number
\[
 \Gamma_{\mathrm{TM}}=\frac{F_+}{zF_-}
 =\frac{1+dP(z)}{z(1-dP(z))}
 \approx48.1062486.
\]
In particular, $\rho_{\mathbb Z}(\Gamma_{\mathrm{TM}})\le1/2$.
\item The infimum of the condition numbers of admissible designs
with support exactly $T$ is attained.
\end{enumerate}
\end{prop}

\begin{proof}
For (i), the identities $\tau(2n)=\tau(n)$ and
$\tau(2n+1)=1-\tau(n)$ for $n\in\mathbb N_0$ show that exactly
one of $2n,2n+1$ belongs to $T$. Hence $\overline d(T)=1/2$.
Expanding the finite products and passing to the limit gives
\[
 \sum_{n=0}^{\infty}(-1)^{\tau(n)}z^n=P(z),\qquad
 \sum_{n\in T}z^n=F_+,\qquad
 \sum_{n\in\mathbb N_0\setminus T}z^n=F_-.
\]
All three series converge absolutely.
Since $\sum_{q=0}^{\infty}z^{2^q}<z/d<1$ and
$\prod_{q=0}^{N-1}(1-z^{2^q})\ge1-\sum_{q=0}^{N-1}z^{2^q}$
for every $N\in\mathbb N$, we have $P(z)>0$.
Also, $1\notin T$ gives $F_-\ge z$, so $F_+>F_->0$.

For $m=2^j$, $j\in\mathbb N$, $k\in\{0,\ldots,m-1\}$, and
$\ell\in\mathbb N_0$, the binary digits of $k$ and $\ell m$
do not overlap. Thus
$\tau(k+\ell m)\equiv\tau(k)+\tau(\ell)\pmod{2}$.
Lemma~\ref{lem:ring-residues} gives
\[
 g_{m,k}=\begin{cases}
 d z^{k/m}F_+,&\tau(k)=0,\\
 d z^{k/m}F_-,&\tau(k)=1.
 \end{cases}
\]
The supremum over all blocks is $dF_+$, attained at $k=0$.
The infimum is $dzF_-$, approached with $k=m-1$ as $j\to\infty$
through odd positive integers, for which $\tau(m-1)=1$.
Their ratio is $\Gamma_{\mathrm{TM}}$.
Since this design has upper density $1/2$, the definition of
$\rho_{\mathbb Z}$ gives
$\rho_{\mathbb Z}(\Gamma_{\mathrm{TM}})\le1/2$, proving (i).

For (ii), apply Theorem~\ref{thm:counting-condition-attainment} with
$T_0=T$ and $F(R)=\#(T\cap[0,R])$ for $R\ge0$.
The design in (i) shows that this class is nonempty, so there is
an optimal design $\mu$ supported on a subset of $T$, with condition
number at most $\Gamma_{\mathrm{TM}}$.
Normalize its weights so that
$\Gamma_{\mathrm{TM}}^{-1}I\le S_\mu\le I$.
Write $w_n=0$ at omitted times and set $v_n=dw_n$ for
$n\in\mathbb N_0$.
Lemma~\ref{lem:ring-scales} gives $0\le v_n\le1$ for every
$n\in\mathbb N_0$.
For $n\in T$, choose $m=2^j>n$ with $j\in\mathbb N$.
The parity identity above and Lemma~\ref{lem:ring-residues} give
\[
 \Gamma_{\mathrm{TM}}^{-1}
 \le z^{n/m}\left(v_n+
       \sum_{\ell\in T\setminus\{0\}}v_{n+\ell m}z^\ell\right)
 \le v_n+F_+-1.
\]
The least positive element of $T$ is $3$, so $F_+-1\le z^3/d$.
Together with $F_-\ge z$ and $F_+\le d^{-1}$, this gives
$v_n\ge\Gamma_{\mathrm{TM}}^{-1}-(F_+-1)
\ge dz^2-z^3/d>0$ for $n\in T$,
where the last inequality follows from $z<1/4$.
The support of $\mu$ is therefore exactly $T$, which proves (ii).
\end{proof}

\section{Common sampling supports}
\label{sec:questions}

The corresponding lower bounds for a common sampling measure with
uniform sampled frame bounds were proved in
\cite[Corollaries~3.7 and~4.3, Theorem~6.2]{WuUniversalWeighted2026}.
Here a common set of available sampling times is fixed, while
the selected subset, weights, and sampled frame bounds may depend
on the orbit frame.

\begin{cor}
Let $T\subset[0,\infty)$ be locally finite, and put
$N_T(R)=\#(T\cap[0,R])$ for $R\ge0$.
\begin{enumerate}[(i)]
\item Fix $\kappa>1$. Suppose that every frame generated by a single
orbit of a positive operator, with frame condition number at most
$\kappa$, admits an admissible sampling design supported on $T$.
Then
\[
 \liminf_{R\to\infty}\frac{N_T(R)}{\log R}
 \ge\frac{1}{h_\kappa},\qquad
 h_\kappa=2\log\left(1+\frac{4(\kappa+1)}{\kappa-1}\right).
\]
\item If every frame generated by a single orbit of a positive operator
admits an admissible sampling design supported on $T$, then
$N_T(R)/\log R\to\infty$ as $R\to\infty$.
\item Fix $\kappa>1$. Suppose that every frame generated by finitely
or countably many orbits of a normal operator, with frame condition
number at most $\kappa$, admits an admissible sampling design supported
on $T$. Then $N_T(R)\ge cR-C$ for all $R\ge0$, with constants
$c>0$ and $C\ge0$.
\end{enumerate}
\end{cor}

\begin{proof}
For (i) and (ii), fix $h>0$ and consider the positive orbit frame
from \cite[Proposition~3.5]{WuUniversalWeighted2026}.
If this frame admits an admissible design $\mu_h$ supported on $T$,
then $N_{\mu_h}(R)\le N_T(R)$ for every $R\ge0$.
Together with \cite[Theorem~3.6]{WuUniversalWeighted2026}, this gives
\[
 \liminf_{R\to\infty}\frac{N_T(R)}{\log R}
 \ge\liminf_{R\to\infty}\frac{N_{\mu_h}(R)}{\log R}
 \ge\frac1h.
\]
For $h=h_\kappa$, the frame condition number is at most $\kappa$,
as shown in the proof of
\cite[Corollary~3.7]{WuUniversalWeighted2026}.
This proves (i).
Under the assumption in (ii), the same counting bound holds
for every $h>0$. Letting $h\downarrow0$ proves (ii).

For (iii), set $\varepsilon=\tfrac12\log\kappa>0$ and
$d_\varepsilon=1-e^{-2\varepsilon}$.
Modify the pair in Section~\ref{sec:normal-model} by replacing
$e^{-1/m}$ with $e^{-\varepsilon/m}$ and $d$ with $d_\varepsilon$.
For $m=2^j$, $j\in\mathbb N$, its frame operator has block
eigenvalues $e^{-2\varepsilon k/m}$ for $k\in\{0,\ldots,m-1\}$.
The optimal frame bounds are $e^{-2\varepsilon}$ and $1$,
so the frame condition number is $\kappa$.

By assumption, this pair has an admissible design $\mu$ supported
on $T$. Let $0<\alpha\le\beta<\infty$ be its sampled frame bounds,
and put $W(R)=\mu([0,R])$ for $R\ge0$.
Let $S_{\mu,m}$ denote the block of its sampled frame operator
on $\mathbb C^m$.
Taking the trace gives
\[
 d_\varepsilon
 \int_{[0,\infty)}e^{-2\varepsilon t/m}\,d\mu(t)
 =\operatorname{tr}S_{\mu,m}
 \le\beta m.
\]
For $R\ge2$, choose $m=2^j\in[R,2R)$; for $0\le R<2$, use $m=2$.
It follows that
$W(R)\le 2e^{2\varepsilon}\beta(1+R)/d_\varepsilon$
for every $R\ge0$.
Integration by parts, as in Lemma~\ref{lem:ring-real-bound}, gives
\[
 d_\varepsilon\int_{(Lm,\infty)}e^{-2\varepsilon t/m}\,d\mu(t)
 \le2e^{2\varepsilon}\beta m
 \left(L+1+\frac{1}{2\varepsilon}\right)e^{-2\varepsilon L}
 \qquad L>0.
\]
Choose $L>0$, independent of $m$, so that this bound is at most
$\alpha m/2$.

Write $S_{\mu,m}=B_m+D_m$, where $B_m$ and $D_m$ are the
contributions from times in $[0,Lm]$ and $(Lm,\infty)$,
respectively.
Then $\operatorname{rank}B_m\le N_\mu(Lm)$ and
$\operatorname{tr}D_m\le\alpha m/2$.
Let $P_m$ be the orthogonal projection onto $\ker B_m$.
Since $P_mD_mP_m=P_mS_{\mu,m}P_m\ge\alpha P_m$, we obtain
\[
 \alpha\bigl(m-N_\mu(Lm)\bigr)
 \le\alpha\operatorname{rank}P_m
 \le\operatorname{tr}(P_mD_mP_m)
 \le\operatorname{tr}D_m
 \le\frac{\alpha m}{2}.
\]
Thus $N_\mu(Lm)\ge m/2$.
For $R\ge2L$, choose a power of two $m$ with
$R/(2L)<m\le R/L$. Then
\[
 N_T(R)\ge N_\mu(R)\ge N_\mu(Lm)
 \ge\frac m2\ge\frac{R}{4L}.
\]
Consequently, $N_T(R)\ge R/(4L)-1/2$ for every $R\ge0$,
which proves (iii).
\end{proof}

The linear counting order in (iii) is optimal. The common support
$T=\mathbb N_0$ with unit weights satisfies
$N_T(R)=\lfloor R\rfloor+1$ for $R\ge0$ and preserves the frame
operator of every frame in the class.

Part (ii) requires a common support for every frame generated by a
single orbit of a positive operator, over all finite frame condition
numbers. Thus (ii) does not imply
$N_T(R)/\log R\to\infty$ as $R\to\infty$ for a class with a fixed
bound $\kappa$ on the frame condition number.
\section*{Acknowledgments}
The author is grateful to Ilya Krishtal for his helpful comments and suggestions,
which improved the presentation of the paper and clarified its connections
with the existing literature.

\end{document}